\documentclass[10pt,reqno]{amsart}
\usepackage{amsmath,amsopn,amssymb,amsthm,multicol}
\usepackage{hyperref}
\usepackage{graphicx}

\DeclareMathOperator{\Ad}{Ad}
\DeclareMathOperator{\ad}{ad}
\DeclareMathOperator{\Aut}{Aut}

\DeclareMathOperator{\GL}{GL}

\DeclareMathOperator{\Ric}{Ric}

\DeclareMathOperator{\Mark}{Mrk}
\DeclareMathOperator{\Span}{span}

\newcommand{\fr}{\mathfrak}
\newcommand{\al}{\alpha}
\newcommand{\be}{\beta}

\newcommand{\bb}{\mathbb}
\DeclareMathOperator{\rnk}{rk}
\DeclareMathOperator{\SO}{SO}
\DeclareMathOperator{\Sp}{Sp}
\DeclareMathOperator{\SU}{SU}
\DeclareMathOperator{\U}{U}
\DeclareMathOperator{\G}{G}
\DeclareMathOperator{\F}{F}
\DeclareMathOperator{\E}{E}
\newcommand{\thickline}{\noalign{\hrule height 1pt}}

 \newtheorem{lemma} {Lemma} [section]
\newtheorem{theorem}[lemma]{Theorem} 
 
\newtheorem{prop} [lemma]{Proposition}  
\newtheorem{definition}[lemma] {Definition} 
\newtheorem{corol}[lemma] {Corollary} 
\newtheorem{example}[lemma] {Example}

\begin{document}

\title{Homogeneous Einstein metrics on generalized flag manifolds with
five isotropy summands} 
\author{Andreas Arvanitoyeorgos, Ioannis Chrysikos, and Yusuke Sakane}
\address{University of Patras, Department of Mathematics, GR-26500 Rion, Greece}
\email{arvanito@math.upatras.gr}
\address{Masaryk University, Department of Mathematics and Statistics, Masaryk University, Brno 611 37, Czech Republic}
\email{chrysikosi@math.muni.cz}
 \address{Osaka University, Department of Pure and Applied Mathematics, Graduate School of Information Science and Technology, Toyonaka, 
Osaka 560-0043, Japan}
 \email{sakane@math.sci.osaka-u.ac.jp}
\medskip
\noindent
 \thanks{The second  author    was full-supported   
  by Masaryk University under the Grant Agency of Czech Republic, project no. P 201/12/G028}

   \begin{abstract}
We construct the homogeneous Einstein equation for generalized flag manifolds $G/K$ of a compact simple Lie group $G$  whose isotropy representation  decomposes into five inequivalent irreducible $\Ad(K)$-submodules.  
 To this end we apply a new technique  which is based on a fibration of a flag manifold 
over another flag manifold and the theory of Riemannian submersions. 
We classify all generalized flag manifolds with five isotropy summands,  
and we use  Gr\"obner bases to study the corresponding polynomial systems for the Einstein equation.
For the generalized flag manifolds
$\E_6/(\SU(4)\times \SU(2)\times\U (1)\times\U (1))$ and 
 $\E_7/(\U(1)\times\U(6))$  we find explicitely
all invariant Einstein metrics up to isometry.  For the generalized flag manifolds 
$\SO(2\ell +1)/(\U(1)\times\U (p)\times\SO(2(\ell -p-1)+1))$ and 
 $\SO(2\ell)/(\U(1)\times\U (p)\times\SO(2(\ell -p-1)))$ we prove existence of at least two non K\"ahler-Einstein metrics.  For small values of $\ell$ and $p$ we give the precise
 number of invariant Einstein metrics.
  

   \end{abstract}

\maketitle


%

 \section{Introduction}
\markboth{Andreas Arvanitoyeorgos, Ioannis Chrysikos and Yusuke Sakane}{A new method for homogeneous Einstein metrics on generalized flag manifolds}

A Riemannian manifold $(M, g)$ is called Einstein if it has constant Ricci curvature, i.e. $\Ric_{g}=\lambda\cdot g$ for some $\lambda\in\bb{R}$.   
We are concerned with homogeneous Einstein metrics on reductive homogeneous spaces whose isotropy representation decomposes into a direct sum
of irreducible non equivalent summands.  Then the Einstein equation reduces to a non linear algebraic system of equations.
The computation of the Ricci tensor is in general a difficult task, especially when the number of isotropy summands increases.
In this paper we introduce a method for computing the Ricci tensor for a homogeneous space via Riemannian submersions, and we apply this
for a large class of homogeneous spaces the generalized flag manifolds. 
These are compact homogeneous spaces of the form $G/K=G/C(S)$, where $G$ is a compact, connected semisimple Lie group and $C(S)$ is the centralizer of a torus $S\subset G$.  These spaces exhaust all compact  simply connected homogeneous K\"ahler manifolds of a compact, connected and semisimple Lie group.  Their classification is based on the  painted Dynkin diagrams (cf. \cite{Ale}, \cite{AA})  and their K\"ahler geometry is very interesting on its own right (cf. \cite{AP}, \cite{Bo}).  For example, $M=G/K=G/C(S)$ admits a finite number  of invariant complex structures, and  for each complex structure  there is a unique homogeneous K\"ahler--Einstein metric.
   
    Nowadays, homogeneous Einstein metrics on flag manifolds have been better understood.  They have been completely classified   for any flag manifold  $M=G/K$ (of a compact simple Lie group $G$) with two (\cite{Chry1}, \cite{Sakane}),     three (\cite{Arv}, \cite{Kim})  and four isotropy summands (\cite{Chry2}, \cite{ACS1}, \cite{ACS2}).     For full flag manifolds corresponding to classical Lie groups the existence problem has  also been studied by several authors  (cf. \cite{Arv}, \cite{DSN}, \cite{Sak}), but  a full classification is still unknown, except for some low dimensional cases.  On the other hand,  in a recent work of the authors (\cite{ACS3})   all $\G_2$-invariant Einstein metrics were obtained on the exceptional full flag manifold $\G_2/T$ (a homogeneous space with six isotropy summands).   However,  we are still far from  general results and a complete classification of  invariant Einstein metrics seems to be difficult (by means of the traditional methods).  
    
    In the present paper we classify flag manifolds whose isotropy representation  decomposes into five irreducible $\Ad(K)$-submodules
   \begin{equation}\label{isodec}
\fr{m}=T_{o}M=\fr{m}_1\oplus\fr{m}_2\oplus\fr{m}_3\oplus\fr{m}_4\oplus\fr{m}_5,
\end{equation}
and use the new method to compute the Ricci tensor for each of these spaces.  
Then we study the existence of non K\"ahler Einstein metrics.
 The first results about Einstein metrics on flag manifolds with five isotropy summands were obtained in a recent work of the second author \cite{Chry}, where he studied $SO(7)$-invariant Einstein metrics  for flag manifolds of the form $SO(7)/K$.   
 Also, in \cite{CS}  the last two authors classified all  homogeneous Einstein metrics for the (unique) exceptional flag manifold $G/K$ with second Betti number $b_{2}(M)=1$, whose isotropy representation satisfies (\ref{isodec}). This flag manifold corresponds to $G=\E_8$. 
 As we will see in this paper there are also flag manifolds with five isotropy summands with $b_2(M)=2$. 
  In fact  the cases $b_{2}(M)=1$ or $b_{2}(M)=2$ exhaust all flag manifolds with five isotropy summands, both classical and exceptional.
 
 Recall that any flag manifold $G/K$ of a compact simple Lie group $G$ is determined by a choice of a pair $(\Pi, \Pi_{0})$, where $\Pi$ is a system of simple roots for $G$ and $\Pi_{0}\subset\Pi$.  By painting black the nodes in the Dynkin diagram  $\Gamma(\Pi)$ of $G$ corresponding to the simple roots of the set $\Pi\backslash\Pi_{0}$, we obtain the {\it painted Dynkin diagram} of $G/K$.  The semisimple part of the reductive subgroup $K$ is obtained by the subdiagram of white roots, and any black root gives rise to a $\U(1)$-component  (the $\U(1)$-components form the center of $K$ whose dimension is equal to  the second Betti number of $M$, see Section 3).  In terms  of painted Dynkin diagrams, flag manifolds $G/K$ of a simple Lie group $G$ whose tangent space   $\fr{m}=T_{o}(G/K)$ decomposes as (\ref{isodec})  can be obtained  as follows: 
  
   \smallskip
 
 {\bf (a)} Paint black one  simple root of Dynkin mark\footnote{The Dynkin mark of a simple root $\al_{i}\in\Pi$  $(i=1, \ldots, \ell)$ is  the  positive integer   $m_{i}$ in the expression of the highest root $\widetilde{\al}=\sum_{k=1}^{\ell}m_{k}\al_{k}$ in terms of simple roots.  We will denote by $\Mark$ the function $\Mark : \Pi\to\bb{Z}^{+}$, $\al_{i}\mapsto m_{i}$.} 5, that is 
 $$\Pi\backslash\Pi_{0}=\{\al_{p} : \Mark(\al_{p})=5\}.$$
   As mentioned earlier, case (a) appears only for $G=\E_{8}$. 
 
 {\bf  (b)} Paint black two simple roots, one of Dynkin mark 1 and  one of Dynkin mark 2,   that is $$\Pi\backslash\Pi_{0}=\{\al_{i}, \al_{j} : \Mark(\al_{i})=1, \ \Mark(\al_{j})=2\}.$$   
 
 {\bf  (c)} Paint black two simple roots both of Dynkin mark 2,  that is $$\Pi\backslash\Pi_{0}=\{\al_{i}, \al_{j} : \Mark(\al_{i})= \Mark(\al_{j})=2\}.$$  
 
 \smallskip
 
\noindent   We call the pairs $(\Pi, \Pi_{0})$ arising form cases (b) and (c) as pairs of {\it Type A} and  {\it Type B} respectively. We will use the same name for the corresponding painting Dynkin diagrams and for the flag manifolds determined by them.  
According to  \cite[Propositions 5 and 6]{Chry2} if $M=G/K$ is of Type A or B  then the corresponding isotropy decomposition is given as follows:

\smallskip 
 {\bf Type A} $\Rightarrow$   
$\fr{m}=\fr{m}_{1}\oplus\fr{m}_{2}\oplus\fr{m}_{3}\oplus\fr{m}_{4}$ \  or \   
$\fr{m}=\fr{m}_{1}\oplus\fr{m}_{2}\oplus\fr{m}_{3}\oplus\fr{m}_{4}\oplus\fr{m}_{5}$.
 
\smallskip 
{\bf Type B} $\Rightarrow$ 
$\fr{m}=\fr{m}_{1}\oplus\fr{m}_{2}\oplus\fr{m}_{3}\oplus\fr{m}_{4}\oplus\fr{m}_{5}$  \ or  \ 
$\fr{m}=\fr{m}_{1}\oplus\fr{m}_{2}\oplus\fr{m}_{3}\oplus\fr{m}_{4}\oplus\fr{m}_{5}\oplus\fr{m}_{6}$.

\medskip

 In Table 1  we give the pairs $(\Pi, \Pi_{0})$  of Type  A and B,  which determine flag manifolds $G/K$ 
 with $\fr{m}= \fr{m}_{1}\oplus\cdots\oplus\fr{m}_{5}$.    The explicit form of these flag manifolds is given in Table 2.
 
\medskip
\smallskip
{\small{\begin{center}
{\bf Table 1.} \ {\small Pairs $(\Pi, \Pi_{0})$ of Type  A and B which determine flag manifolds with five isotropy summnads}  
\end{center}
\begin{center}
 \begin{tabular}{|r|l|l|}
 \hline
      Classical Lie group $G$    & $B_{\ell}=\SO(2\ell+1)$  & $D_{\ell}=\SO(2\ell)$    \\
                \thickline
    Type A   &  $\Pi\backslash\Pi_{0}=\{\al_1, \al_{p+1}  :  2\leq p\leq \ell-1\}$ & $\Pi\backslash\Pi_{0}=\{\al_1, \al_{p+1} : 2\leq p\leq \ell-3\}$  \\
     Type B & $\Pi\backslash\Pi_{0}=\{\al_p, \al_{p+1} : 2\leq p\leq \ell-1\}$ & $\Pi\backslash\Pi_{0}=\{\al_p, \al_{p+1} :   2\leq p\leq \ell- 3\}$ \\
       \hline
    Exceptional Lie group $G$   & $\E_6$ & $\E_7$ \\ 
       \thickline
     Type A   &  $\Pi\backslash\Pi_{0}=\{\al_1, \al_4\}$ & $\Pi\backslash\Pi_{0}=\{\al_1, \al_7\}$ \\
     Type A   &  $\Pi\backslash\Pi_{0}=\{\al_2, \al_5\}$ & \\
     Type B   &  $\Pi\backslash\Pi_{0}=\{\al_4, \al_{6}\}$  & $\Pi\backslash\Pi_{0}=\{\al_6, \al_7\}$ \\
      Type B  &  $\Pi\backslash\Pi_{0}=\{\al_2, \al_{6}\}$  & \\
  \hline
 \end{tabular}
 \end{center}}}
\medskip
\smallskip

 For any Lie group $G$ which appears in Table 1, one can see that the flag manifolds $G/K$ of Type A and B  are equivalent to each other (since the isotropy subgroups are conjugate).\footnote{Two flag manifolds $G/K$ and $G/K'$ are called equivalent if there exists an automorphism $\phi\in\Aut(G)$ such that $\phi(K)=K'$.  Such an automorphism defines a diffeomorphism $\tilde{\phi}: G/K\to G/K'$ given by $\tilde{\phi}(gK)=\phi(g)K'$, which satisfies $\tilde{\phi}(gx)=\phi(g)\tilde{\phi}(x)$ for all $g\in G, x\in G/K$.}  In fact, in Section 4 we will prove   that there is an isometry   arising from the action of the Weyl group   of $G$ and makes the corresponding  flag manifolds $G/K$ of Type A and B isometric to each other, as real manifolds.  
 For this reason there are only four non isometric flag manifolds (as real manifolds) with $b_2(M)=2$ with five isotropy summands as shown in Table 2.

\medskip
\smallskip  
{\small{\begin{center}
{\bf Table 2.} \ {\small Generalized flag manifolds with five isotropy summands and $b_2(M)=2$}  
\end{center}
\begin{center}
\begin{tabular}{ll}
$M=G/K$ classical &  $M=G/K$ exceptional \\
\thickline
$\SO(2\ell+1)/\U(1)\times \U(p)\times \SO(2(\ell-p-1)+1)$ & $\E_6/\SU(4)\times \SU(2)\times \U(1)^{2}$ \\
 $\SO(2\ell)/\U(1)\times \U(p)\times \SO(2(\ell-p-1))$  &  $\E_7/\SU(6)\times \U(1)^{2}$
 \end{tabular}
 \end{center}}}

   \smallskip

\noindent The classification of flag manifolds with five isotropy summands is given in Section 4.

The main difficulty in constructing the  Einstein equation for a $G$-invariant metric on a flag manifold in Table 2 is  the calculation of the non zero structure constants $\displaystyle{k \brack {ij}}$ of $G/K$ with respect to the decomposition (\ref{isodec}) (see Section 2).   
A first step towards this procedure is to use the known K\"ahler-Einstein metric on  any flag manifold.  This metric can be computed   
by using the Koszul formula (see Section 5).
Secondly, and this is the main contribution of the present paper, we  take advantage of a fibration of a flag manifold  over another  flag manifold  and  use methods of   Riemannian submersions to compare Ricci tensors of total space and base space.
In this way we are able to calculate  $\displaystyle{k \brack {ij}}$ in terms of the dimension of the submodules $\fr{m}_i$ in the
decomposition (\ref{isodec}).
We point out that this new technique can be useful for the study of homogeneous Einstein metrics for more general homogeneous spaces, whose isotropy representation satisfies certain conditions.
 In this way the Einstein equation reduces to a  polynomial system of four equations in four unknowns.
  For the exceptional flag manifolds we classify all homogeneous Einstein metrics.    
 For the classical flag manifolds a complete classification of homogeneous Einstein metrics in the general case  is a difficult task, because the corresponding systems of equations depend on  four positive parameters (which define the invariant Riemannian metric), the Einstein constant $\lambda>0$ and the positive integers $\ell$ and $p$.    
 However, by using Gr\"obner bases we can show that the equations are reduced to a polynomial equation of   one variable and then prove the existence of non K\"aler Einstein metrics.  
 In fact, this is another contribution of the present paper, because we prove existence of real solutions
 for polynomial equations whose coefficients depend on parameters ($\ell$ and $p$).

 The paper is organized as follows: In Section  2 we discuss $G$-invariant metrics on homogeneous spaces $G/K$ and compare them with 
 Riemannian submersion metrics  for a fibration $L/K\to G/K \to G/L$.  
 Then we give expression for the Ricci tensor for a submersion metric.
 In Section 3 we recall various facts about generalized flag manifolds which will be used in Section 4 for the classification of such spaces with
 five isotropy summands.  Combined with the work \cite{CS} these spaces are 
 $\E_8/ (\U(1)\times \SU(4)\times \SU(5))$ with second Betti number 1 and the spaces in Table 2 with second Betti number 2.
 In Section 5 we give the K\"ahler-Einstein metrics for flag manifolds with five isotropy summands and in Section 6 we compute the
 Ricci tensor for these spaces by using our method of Riemannian submersions and the known K\"ahler-Einstein metrics.
 Section 7 is devoted to the study of the algebraic systems of equations by using Gr\"obner bases techniques.

\medskip
 \noindent
 {\sc Theorem A.}  {\it Let $M=G/K$ be one of the flag manifolds $\E_6/(\SU(4)\times \SU(2)\times\U (1)\times\U (1))$ or
 $\E_7/(\U(1)\times\U(6))$. Then $M$ admits exactly seven $G$-invariant Einstein metrics up to isometry.
 There are two K\"ahler-Einstein metrics and five non-K\"ahler Einstein metrics (up to scalar).  These metrics are given in Theorems  
 7.1 and 7.2}.

  \noindent
 {\sc Theorem B.} {\it Let $M=G/K$ be one of the flag manifolds $\SO(2\ell +1)/(\U(1)\times\U (p)\times\SO(2(\ell -p-1)+1))$  $(\ell \ge 3, \ 3\le p\le\ell -1)$ or    
 $\SO(2\ell)/(\U(1)\times\U (p)\times\SO(2(\ell -p-1)))$ $(\ell\ge 5, \ 3\le p\le\ell -3)$.
 Then $M$ admits at least two $G$-invariant non-K\"ahler Einstein metrics
 (cf. Theorem 7.3).  
 }
 
 \noindent
 {\sc Theorem C.} {\it Let $M=G/K$ be one of the flag manifolds 
 $\SO(2\ell +1)/(\U(1)\times\U (2)\times\SO(2\ell -5))$  $(\ell \ge 6)$ or    
 $\SO(2\ell)/(\U(1)\times\U (2)\times\SO(2(\ell -3)))$ $(\ell\ge 7)$.
 Then $M$ admits at least four $G$-invariant non-K\"ahler Einstein metrics
 (cf. Theorem 7.4).  
 }

 \noindent
 
 Note that the special case $\ell=3, p=2$ (the space $\SO(7)/\U(1)\times \U(2)$)  was studied among other results in \cite{Chry}.
 This flag manifold admits (up to isometry) precisely three non-K\"ahler Einstein metrics and precisely 
 two K\"ahler-Einstein metrics.
 For small values of $\ell$ and $p$ it is possible to obtain the precise number of all non isometric
 invariant Einstein metrics.  We discuss this at the end of the paper (cf. Table 4).

\section{Reductive homogeneous spaces  and  Riemannian submersions}
In this section we describe the Einstein equation for any $G$-invariant metric on a compact reductive homogeneous manifold $G/K$, and  
give expression of the Ricci tensor of a submersion metric associated to a certain fibration $G/K\to G/L$.
This expression will be used in Chapter 6 to calculate $\displaystyle{k \brack {ij}}$ for flag manifolds with five isotropy summands.

\subsection{The Ricci tensor for a reductive homogeneous spaces}
Let $G$ be a compact semisimple Lie group, $K$ a connected closed subgroup of $G$  and  
let  $\frak g$ and $\fr{k}$  be  the corresponding Lie algebras. 
The Killing form of $\frak g$ is negative definite, so we can define an $\mbox{Ad}(G)$-invariant inner product $B$ on 
  $\frak g$ given by 
$B  = $  $-$ Killing form of $\fr{g}$. 
Let $\frak g$ = $\frak k \oplus
\frak m$ be a reductive decomposition of $\frak g$ with respect to $B$ so that $\left[\,\frak k,\, \frak m\,\right] \subset \frak m$ and
$\frak m\cong T_o(G/K)$.
 We assume that $ {\frak m} $ admits a decomposition into mutually non equivalent irreducible $\mbox{Ad}(K)$-modules as follows: \ 
\begin{equation}\label{iso}
{\frak m} = {\frak m}_1 \oplus \cdots \oplus {\frak m}_q.
\end{equation} 
Then any $G$-invariant metric on $G/K$ can be expressed as  
\begin{eqnarray}
 \langle  \,\,\, , \,\,\, \rangle  =  
x_1   B|_{\mbox{\footnotesize$ \frak m$}_1} + \cdots + 
 x_q   B|_{\mbox{\footnotesize$ \frak m$}_q},  \label{eq2}
\end{eqnarray}
for positive real numbers $(x_1, \cdots, x_q)\in\bb{R}^{q}_{+}$.  Note that  $G$-invariant symmetric covariant 2-tensors on $G/K$ are 
of the same form as the Riemannian metrics (although they  are not necessarilly  positive definite).  
 In particular, the Ricci tensor $r$ of a $G$-invariant Riemannian metric on $G/K$ is of the same form as (\ref{eq2}), that is 
 \[
 r=y_1 B|_{\mbox{\footnotesize$ \frak m$}_1}  + \cdots + y_{q} B|_{\mbox{\footnotesize$ \frak m$}_q} ,
 \]
 for some real numbers $y_1, \ldots, y_q$.

Let $\lbrace e_{\alpha} \rbrace$ be a $B$-orthonormal basis 
adapted to the decomposition of $\frak m$,    i.e. 
$e_{\alpha} \in {\frak m}_i$ for some $i$, and
$\alpha < \beta$ if $i<j$. 
We put ${A^\gamma_{\alpha
\beta}}=B \left(\left[e_{\alpha},e_{\beta}\right],e_{\gamma}\right)$ so that
$\left[e_{\alpha},e_{\beta}\right]
= \displaystyle{\sum_{\gamma}
A^\gamma_{\alpha \beta} e_{\gamma}}$ and set 
$\displaystyle{k \brack {ij}}=\sum (A^\gamma_{\alpha \beta})^2$, where the sum is
taken over all indices $\alpha, \beta, \gamma$ with $e_\alpha \in
{\frak m}_i,\ e_\beta \in {\frak m}_j,\ e_\gamma \in {\frak m}_k$ (cf. \cite{Wa}).  
Then the positive numbers $\displaystyle{k \brack {ij}}$ are independent of the 
$B$-orthonormal bases chosen for ${\frak m}_i, {\frak m}_j, {\frak m}_k$,
and 
$\displaystyle{k \brack {ij}}\ =\ \displaystyle{k \brack {ji}}\ =\ \displaystyle{j \brack {ki}}.  
 \label{eq3}
$

Let $ d_k= \dim{\frak m}_{k}$. Then we have the following:

\begin{lemma}\label{ric2}\textnormal{(\cite{PS})}
The components ${ r}_{1}, \dots, {r}_{q}$ 
of the Ricci tensor ${r}$ of the metric $ \langle  \,\,\, , \,\,\, \rangle $ of the
form {\em (\ref{eq2})} on $G/K$ are given by 
\begin{equation}
{r}_k = \frac{1}{2x_k}+\frac{1}{4d_k}\sum_{j,i}
\frac{x_k}{x_j x_i} {k \brack {ji}}
-\frac{1}{2d_k}\sum_{j,i}\frac{x_j}{x_k x_i} {j \brack {ki}}
 \quad (k= 1,\ \dots, q),    \label{eq51}
\end{equation}
where the sum is taken over $i, j =1,\dots, q$.
\end{lemma} 
Since by assumption the submodules $\fr{m}_{i}, \fr{m}_{j}$ in the decomposition (\ref{iso}) are matually non equivalent for any $i\neq j$, it will be $r(\fr{m}_{i}, \fr{m}_{j})=0$ whenever $i\neq j$.  Thus by Lemma \ref{ric2} it follows that    $G$-invariant Einstein metrics on $M=G/K$ are exactly the positive real solutions $g=(x_1, \ldots, x_q)\in\bb{R}^{q}_{+}$  of the  polynomial system $\{r_1=\lambda, \ r_2=\lambda, \ \ldots, \ r_{q}=\lambda\}$, where $\lambda\in \bb{R}_{+}$ is the Einstein constant.

\subsection{Riemannian submersions}
 Let $G$ be a compact semisimple Lie group and $K$, $L$ two closed subgroups of $G$ with $K \subset L$. 
 Then there is a natural fibration $\pi  :  G/K  \to G/L $ with fiber $L/K$. 
    
     Let $\frak p$ be the orthogonal complement of $\frak l$ in
$\frak g$ with respect to $B $,  and  $\frak q$ be the orthogonal complement of $\frak k$ in $\frak l$. Then we have $\frak g$ = $\frak l \oplus \frak p = \frak k  \oplus \frak q  \oplus \frak p$. An $\mbox{Ad}_G (L)$-invariant scalar product on $\frak p$ defines a $G$-invariant metric $\check{g}$ on $G/L$, and an $\mbox{Ad}_L (K)$-invariant scalar product on $\frak q$ defines an $L$-invariant metric $\hat{g}$ on $L/K$. The orthogonal direct sum for these scalar products on $\frak q  \oplus \frak p$ defines a $G$-invariant metric $ g$ on $G/K$, called {\it submersion metric}. 

\begin{theorem}\textnormal{\cite[p. 257]{Be}}
The map $\pi$ is a Riemannian submersion from $( G/K, \,  g )$ to $( G/L, \, \check{g} )$ with totally geodesic fibers isometric to $( L/K, \, \hat{g} )$. 
\end{theorem}
Note that $\frak q$ is the vertical subspace of the submersion and $\frak p$ is the horizontal subspace. 
   
   For a Riemannian submersion, O'Neill \cite{O'Neill} has introduced two tensors $ A$ and $T$. 
   Since in our case the fibers are totally geodesic it is  $T =0$. We also have that 
   $$A_X Y = \frac{1}{2} [X, \ Y ]_{\frak q} \quad \mbox{  for }  \ X, Y \in {\frak p}.  $$ 

  Let $\{X_i\}$ be an orthonormal basis of $\frak p$ and  $\{U_j\}$  be an orthonormal basis of $\frak q$. For $X, Y \in  \frak p$ we put $\displaystyle  g(A_X, \ A_Y) =  \sum_{i} g( A_X X_i, A_Y X_i ). $ 
  Then we have  that
   \begin{eqnarray}\label{submeq1}
    g(A_X, \ A_Y) =  \frac{1}{4}\sum_{i} \hat{g}( [X, \ X_i ]_{\frak q}, \  [ Y, \ X_i ]_{\frak q}).
   \end{eqnarray}

 Let $r$, $\check{r}$ be the Ricci tensors of the metrics $g$, $\check{g}$ respectively. Then we have
 (\cite[p. 244]{Be}) 
   \begin{eqnarray} \label{submeq2}
r(X, Y) = \check{r}(X, Y) - 2 g(A_X, A_Y) \quad \mbox{  for }  \ X, Y \in {\frak p}. 
     \end{eqnarray} 
     We remark that there is a corresponding expression $r (U,V)$ for vertical vectors, but it does not contribute
   additional information in our approach.                
   
Let  $${\frak p} = {\frak p}_1 \oplus \cdots \oplus {\frak p}_{\ell}, \quad {\frak q} = {\frak q}_1 \oplus \cdots \oplus {\frak q}_{s}$$ be a decomposition of $ {\frak p} $ into irreducible $\mbox{Ad}(L)$-modules and   a decomposition of $ {\frak q} $ into irreducible $\mbox{Ad}(K)$-modules respectively,  
and assume that  the $\mbox{Ad}(L)$-modules ${\frak p}_j$ $( j = 1, \cdots, \ell )$  are   mutually non equivalent.  
 Note that each irreducible component $ {\frak p}_{j}$ as $\mbox{Ad}(L)$-module can be decomposed into irreducible $\mbox{Ad}(K)$-modules.  
  To compute  the values $\displaystyle {k \brack {ij}}$ for $G/K$, we use information from the Riemannian submersion
     $\pi : ( G/K, \,  g )  \to  ( G/L, \, \check{g} )$ with totally geodesic fibers isometric to $( L/K, \, \hat{g} )$. 
 We consider a $G$-invariant metric on $G/K$ defined by a  Riemannian submersion  $\pi : ( G/K, \,  g )  \to  ( G/L, \, \check{g} )$ given by 
\begin{eqnarray}
g  =  
y_1   B|_{\mbox{\footnotesize$ \frak p$}_1} + \cdots + 
 y_{\ell}   B|_{\mbox{\footnotesize$ \frak p$}_{\ell}} + z_1 B|_{\mbox{\footnotesize$ \frak q$}_1} + \cdots + 
 z_{s}   B|_{\mbox{\footnotesize$ \frak q$}_{s}} \label{eq4}
 \end{eqnarray}
for positive real numbers $ y_1, \cdots, y_{\ell}, z_1, \cdots, z_s$. 
   
     Then we decompose each irreducible component ${\frak p}_{j}$  into irreducible $\mbox{Ad}(K)$-modules
 $$ {\frak p}_{j} = {\frak m}_{j, 1} \oplus \cdots \oplus {\frak m}_{j,  \, k_j}, $$  where the 
 $\mbox{Ad}(K)$-modules ${\frak m}_{j,  t} $ $( j= 1, \cdots, \ell,\  t = 1, \cdots, k_j )$  are  mutually non equivalent and are chosen
 to be  (up to reordering) submodules from the decomposition (2). 
Then the submersion metric (\ref{eq4}) can be written as 
\begin{eqnarray}
g  =  y_1 \sum_{t = 1}^{k_1}  B|_{\mbox{\footnotesize$ \frak m$}_{1, t}} + \cdots + 
 y_{\ell}   \sum_{t = 1}^{k_{\ell}} B|_{\mbox{\footnotesize$ \frak m$}_{\ell, t}} + z_1 B|_{\mbox{\footnotesize$ \frak q$}_1} + \cdots +  z_{s}   B|_{\mbox{\footnotesize$ \frak q$}_{s}} \label{metric_submersion}
 \end{eqnarray}
 and this is a special case of the $G$-invariant metric (\ref{eq2}). 
 
 
 \begin{lemma} \label{submersion_ricci} 
 Let $ d_{j, t} = \dim {\frak m}_{j, t}$. 
The components $ r_{(j, \, t)}$  $( j= 1, \cdots, \ell,\  t = 1, \cdots, k_j )$
of the Ricci tensor ${r}$ for the metric {\em (\ref{metric_submersion})} on $G/K$ are given by 
\begin{equation}
 r_{(j, \, t)} = \check{r}_{j} -  \frac{1}{2 d_{j,\,  t} }\sum_{i=1}^s \sum_{j',\,  t'} \frac{z_i}{y_{j} y_{j'}} {i \brack {(j, t) \  (j', t')}},  
\end{equation} 
where $\check{r}_{j}$ are the components of Ricci tensor $\check{r}$ for the metric $\check{g}$ on $G/L$. 

\end{lemma}
\begin{proof}
 Let $\lbrace e_{\alpha}^{(j, t)},  e_{\beta}^{(i)} \rbrace$ be a  $B$-orthonormal basis 
adapted to the decomposition of $ \displaystyle {\frak p} \oplus{ \frak q} = \sum_{j}\sum_{t = 1}^{k_j} \frak m_{j, t}\oplus \sum_{i}  {\frak q}_i  $   (with $ e_{\alpha}^{(j, t)} \in  \frak m_{j, t}$ and $e_{\beta}^{(i)} \in {\frak q}_i$). 
 Put $\displaystyle {X_{\alpha}^{(j, t)} = \frac{1}{\sqrt{y_j}} e_{\alpha}^{(j, t)} }$ and $\displaystyle X_{\beta}^{(i)} = \frac{1}{\sqrt{z_i} } e_{\beta}^{(i)}$. Then $\left\{ X_{\alpha}^{(j, t)},  X_{\beta}^{(i)} \right\}$ is an orthonormal basis of $ \displaystyle {\frak p} \oplus{ \frak q}$ for the metric $g$. 
 Then, by using equations (\ref{submeq1}) and (\ref{submeq2}), we obtain that
\begin{equation*} 
\sum_{\gamma=1}^{ d_{j, t}}  r(X_{\gamma}^{(j, t)}, X_{\gamma}^{(j, t)} ) = \sum_{\gamma=1}^{ d_{j, t}}  \check{r}(X_{\gamma}^{(j, t)}, X_{\gamma}^{(j, t)} ) - \frac{1}{2}\sum_{i} \sum_{j', t'} \frac{z_i}{y_{j} y_{j'}} {i \brack {(j, t) \  (j', t')}}.
\end{equation*} 
Noting that $\left\{ X_{\gamma}^{(j, t)} \right\}_{\gamma = 1}^{d_{j, t}} $ is an orthonormal basis of ${ \frak m}_{j, t}$,  we obtain our claim. 
\end{proof}
Notice that when metric (\ref{eq4}) is viewed as a metric (\ref{eq2}) then the horizontal part of  $r_{(j, \, t)}$ equals to $\check{r}_{j}$ ($j=1,\dots ,\ell$), i.e. it is independent of $t$.

\section{Generalized flag manifolds}
We recall some facts about generalized flag manifolds, concerning painted Dynkin diagrams, isotropy representation and $\fr{t}$-roots.   For simplicity we   work with simple Lie algebras and groups (the results in the semisimple case are obtained by piecing together the simple factors). 

  \subsection{Description of flag manifolds in terms of painted Dynkin diagrams}
 Flag manifolds can be described  in terms of root systems as follows:   Let  $G$ be a compact connected simple Lie group with Lie algebra $\fr{g}$, and let $\fr{h}$ a maximal abelian subalgebra of $\fr{g}$.  We denote by $\fr{g}^{\bb{C}}$ and $\fr{h}^{\bb{C}}$ their complexifications and we assume that $\dim_{\bb{C}}\fr{h}^{\bb{C}}=l={\rm rank} G$. We identify an element of the root system $\Delta$ of
 ${\frak g }^{\mathbb C}_{}$ with respect to the Cartan subalgebra  
 ${\frak h}^{\mathbb C}_{}$ with an element of ${\frak h}_0 =  \sqrt{-1}\frak h$, by the
duality defined by the Killing form of ${\frak
g}^{\mathbb C}_{}$. This means that for any $\al\in\Delta$ we can define $H_{\al}\in\fr{h}_{0}$ by $\al(H)=B(H_{\al}, H)$ for any $H\in\fr{h}^{\bb{C}}$.   Consider the root space decomposition of  ${\frak g }^{\mathbb C}_{}$
relative to  ${\frak h}^{\mathbb C}_{}$, that is $ {\frak g }^{\mathbb C}_{}  =  {\frak h}^{\mathbb C}_{} \oplus
\sum^{}_{\alpha \in \Delta} {\frak g }^{\mathbb C}_{\alpha}$, and let  $\Pi$ = $\{\alpha^{}_1, \dots, \alpha^{}_l\}$
be a system of simple roots $\Delta$.  We denote by $\{\Lambda^{}_1, \dots,
\Lambda^{}_l\}$ the fundamental weights of ${\frak g }^{\mathbb C}_{}$ 
corresponding to $\Pi$, that is $\displaystyle\frac{2(\Lambda^{}_i, \alpha^{}_j)}{(\alpha^{}_j, \alpha^{}_j)} =
\delta^{}_{ij}$ for any $1 \le i, j \le l$.
  
Let $\Pi^{}_0$ be a subset of $\Pi$ and set $\Pi_{\fr{m}}=\Pi \backslash \Pi^{}_0$ =
$\{\alpha^{}_{i_1}, \dots, \alpha^{}_{i_r}\}$, where  $1 \le
{i_1} < \cdots <{i_r} \le l $. We put
$
\Delta_{0}= \Delta\cap\{\Pi^{}_0\}^{}_{\mathbb Z}=\{\be\in \Delta : \be=\sum_{\al_{i}\in\Pi_{0}}k_{i}\al_{i}, \ k_{i}\in\bb{Z}\},
$
where    
 $\{\Pi^{}_0\}^{}_{\mathbb Z}$ denotes   the set  of roots generated by
  $\Pi_{0}$  with integer coefficients (this is a  the subspace of ${\frak h}_0$).   Then $\Delta_{0}$ is a root subsystem of $\Delta$, which means that for any $\al, \be\in \Delta_{0}$ with $\al+\be\in \Delta$ it is also $\al+\be\in \Delta_{0}$. Thus $\Delta_{0}$ generates a  maximal complex reductive Lie subalgebra  $\fr{k}^{\mathbb{C}}=\fr{h}^{\mathbb{C}}\oplus\sum_{\be\in \Delta_{0}}\fr{g}_{\be}^{\mathbb{C}}$ of $\fr{g}^{\bb{C}}$, that is 
  $
  \fr{k}^{\bb{C}}=\fr{z}\oplus\fr{k}_{ss}^{\bb{C}},
  $
     where $\fr{z}$ is the center of $\fr{k}^{\bb{C}}$ and 
     $\fr{k}_{ss}^{\bb{C}}=[\fr{k}^{\bb{C}}, \fr{k}^{\bb{C}}]$ 
     is its semisimple part.  In fact,
      $\Delta_{0}$ is the root system of $\fr{k}^{\bb{C}}_{ss}$, and 
        $\Pi_{0}$ is the correpsonding system of simple roots. Thus we can obtain the decomposition $\fr{k}_{ss}^{\bb{C}}=\fr{h}^{\bb{C}}_{K}\oplus\sum_{\al\in\Delta_{0}}\fr{g}_{\al}^{\bb{C}}$.  Here  $\fr{h}^{\bb{C}}_{K}=\Span_{\bb{C}}\{H_{\al} : \al\in\Pi_{0}\}\subset\fr{h}^{\bb{C}}$ is the Cartan subalgebra   of $\fr{k}_{ss}^{\bb{C}}$ in $\fr{h}^{\bb{C}}$. Note that the center $\fr{z}$ (always non trivial) can be considered as the orthogonal complement of  $\fr{h}^{\bb{C}}_{K}$  in $\fr{h}^{\bb{C}}$ (with respect to the Killing form), that is $\fr{h}^{\bb{C}}=\fr{h}^{\bb{C}}_{K}\oplus\fr{z}$.

\begin{definition}\label{compl}
The roots of the set $\Delta_{\fr{m}}=\Delta\backslash \Delta_{0}$ are called  complemetary roots.
\end{definition}

Note that $\Delta_{\fr{m}}$ is not a root system in general. 
Choose a system of positive roots $\Delta^{+}$ for $\fr{g}^{\bb{C}}$ with respect to $\Pi$ and set $
\Delta_{\fr{m}}^{\pm}=\Delta^{\pm}\backslash \Delta_{0}^{\pm}$, where $\Delta_{0}^{\pm}=\Delta^{\pm}\cap\{\Pi_{0}\}_{\mathbb{Z}}$ and $\Delta^{-}=\{-\al : \al\in \Delta^{+}\}$.
  Then, the set $\Delta_{\Pi_{0}}=\Delta_{0}^{-}\cup \Delta^{+}=\Delta_{0}\cup (\Delta^{+}\backslash\Delta_{0}^{+})=\Delta_{0}\cup\Delta_{\fr{m}}^{+}$ is a root subsystem of $\Delta$ (\cite[p. 16]{Ale}) and   the  subalgebra 
\begin{equation}\label{parabol}
\fr{p}_{\Pi_{0}}=\fr{h}^{\bb{C}}\oplus\sum_{\al\in\Delta_{0}^{-}\cup\Delta^{+}}\fr{g}_{\al}^{\bb{C}}=\fr{h}^{\mathbb{C}}\oplus\sum_{\al\in \Delta_{0}\cup \Delta_{\fr{m}}^{+}}\fr{g}_{\al}^{\bb{C}}=\fr{h}^{\bb{C}}\oplus\sum_{\al\in \Delta_{0}}\fr{g}_{\al}^{\mathbb{C}}\oplus\sum_{\al\in \Delta_{\fr{m}}^{+}}\fr{g}^{\mathbb{C}}_{\al}
\end{equation}
is a parabolic subalgebra of $\fr{g}^{\mathbb{C}}$, since it contains the Borel subalgebra 
$\fr{b}=\fr{h}^{\mathbb{C}}\oplus\sum_{\al\in \Delta^+}\fr{g}^{\mathbb{C}}_{\al}\subset\fr{g}^{\bb{C}}$.  In particular,  we have a direct decomposition ${\frak p}_{\Pi_{0}}  =  {\frak k}^{\mathbb C}_{} \oplus {\frak r}$, where $\fr{r}=\sum_{\al\in\Delta_{\fr{m}}^{+}}\fr{g}_{\al}^{\bb{C}}$
is the nilradical of $\fr{p}$ (a regular nilpotent subalgebra of $\fr{g}^{\bb{C}}$).  It is known that any parabolic subalgebra is conjugate to a subalgebra of the form $\fr{p}_{\Pi_{0}}$ for some subset $\Pi_{0}\subset\Pi$,    (cf. \cite{Ale}, \cite{GOV}).  Note that the cases $\Pi_{0}=\emptyset$ and $\Pi_{0}=\Pi$ define the spaces $\fr{b}$ and $\fr{g}^{\bb{C}}$ respectively.  In this way we can construct a   flag manifold $M=G^{\mathbb{C}}/P$, where $G^{\mathbb{C}}$ is  the simply connected complex simple Lie 
group whose Lie algebra is $\fr{g}^{\mathbb{C}}$ and   $P\subset G^{\mathbb{C}}$ is the parabolic subgroup generated 
by $\fr{p}_{\Pi_{0}}$.  Since  $P$  is always connected, the flag manifold is a (compact) simply connected complex homogeneous manifold. The real representation $M=G/K=G/C(S)$ is obtained by the transitive action of $G$ on $M=G^{\mathbb{C}}/P$, where  the close connected subgroup $K=P\cap G$  is identified with the centralizer $C(S)$ of a torus $S\subset G$ (cf. \cite{Ale}, \cite{GOV}). Thus we always have $\rnk G=\rnk K$. 

Fix now a Weyl basis $E_{\alpha} \in {\frak g}^{\mathbb C}_{\alpha}
\,\,(\alpha \in \Delta )$ with
$$\begin{array}{ll} 
\left[E_{\alpha}, E_{-\alpha}\right] &=  -H_{\alpha} \,(\alpha \in \Delta)
 \\
{} \left[E_{\alpha}, E_{\beta}\right] &=  
\left\{
 \begin{array}{ll} N_{\alpha, \, \beta}E_{\alpha + \beta}& \mbox{if \quad} \alpha +\beta \in
   \Delta
   \\
   0 & \mbox{if \quad} \alpha +\beta \not\in \Delta, 
 \end{array}
\right.
\end{array}$$
where $N_{\alpha, \, \beta}$ = 
$N_{-\alpha, \, -\beta} \in {\mathbb R}.$
Then we have 
\begin{equation}\label{realg}
{\frak g} = {\frak h} + \sum_{\alpha \in \Delta} \left\{
{\mathbb R}(E_{\alpha} + E_{-\alpha}) + {\mathbb R} \sqrt{-1} (E_{\alpha}
- E_{-\alpha})\right\}
\end{equation}   
The Lie algebra $\fr{k}=\fr{p}_{\Pi_{0}}\cap\fr{g}$ of the isotropy subgroup $K$ is a Lie subalgebra of $\frak g$, given
by 
\begin{equation}\label{realk}
{\frak k} ={\frak h} + \sum_{\alpha \in \Delta_{0}^{+}}
 \left\{
{\mathbb R} (E_{\alpha} + E_{-\alpha}) + {\mathbb R}\sqrt{-1} (E_{\alpha}
- E_{-\alpha})\right\}.
\end{equation}
    As a real reductive subalgebra, $\fr{k}$ decomposes into a direct sum of its center $\fr{t}$ and its semisimple part  $[\fr{k}, \fr{k}]$. Note that
     \[
     {\frak t} =\fr{z}\cap \fr{h}_{0}= \Big\{ H \in   {\frak h}_0 :  ( H, \ \Pi_{0}) =   0   \Big\},
     \]
 where $( \ , \ )$ denotes the inner product on $\fr{h}_{0}$ (or on the dual space $\fr{h}_{0}^{*}$) induced by the Killing form and $\fr{z}$ is the center of $\fr{k}^{\bb{C}}$. One can also show that the fundamental weights $\{ \Lambda^{}_{i_1}, \cdots, \Lambda^{}_{i_r}\}$ form a basis of ${\frak t}$ and that $\fr{t}$ is a real form of $\fr{z}$.  If we set $ \frak s = \sqrt{-1}{\frak t} $ then  ${\frak k} $ is given by 
     ${\frak k} = {\frak z}({\frak s})$ (the Lie algebra of the centralizer of a torus $S$ in $G$).   

 All information which is contained in the pair $(\Pi, \Pi_{0})$ can be presented graphicaly by the  painted Dynkin diagram  of $M=G^{\bb{C}}/P=G/K$, which is defined as follows: 
\begin{definition}
 Let $\Gamma(\Pi)$ be the Dynkin diagram of $\Pi$.  By painting black in    $\Gamma(\Pi)$ the simple roots  $\al_{i}\in\Pi_{\fr{m}}=\Pi\backslash\Pi_{0}$
    we obtain the painted Dynkin diagram $\Gamma(\Pi_{\fr{m}})$ of $M$.
    \end{definition}  
  The isotropy subgroup $K$ can be determined from the painted Dynkin diagram $\Gamma(\Pi_{\fr{m}})$as follows: its  semisimple part 
      is defined by the subdiagram of white roots (which 
    is not necessarily connected), 
    and  each black root  gives  rise to a $\U(1)$-component 
    which determines the center $Z(K)$ of $K$.  
    We will  often make use of the diffeomorphism $\SU(n)\times \U(1)\cong \U(n)$.

 \subsection{Isotropy  summands, $\fr{t}$-roots and $G$-invariant Riemannian metrics}
Following the notation of the previous paragraph,  we assume that  a flag manifold
$M=G^{\bb{C}}/P=G/K$ is defined by a subset $\Pi_{0}\subset\Pi$, such that $\Pi_{\fr{m}}=\Pi \backslash \Pi^{}_0$ =
$\{\alpha^{}_{i_1}, \dots, \alpha^{}_{i_r}\}$, where  $1 \le
{i_1} < \cdots <{i_r} \le l $, and   let $\fr{g}=\fr{k}\oplus\fr{m}$ be a   reductive decomposition of the Lie algebra $\fr{g}$  with respect $B$. 
We identify the isotropy representation $\chi : K\to \GL(\fr{m})$ of $G/K $   with the adjoint representation $\Ad|_{K}$ restricted to $\fr{m}$.  
In view of relations (\ref{realg}), (\ref{realk}) and the splitting $\Delta_{\fr{m}}^{+}=\Delta^{+} \backslash \ \Delta_{0}^{+}$ 
it follows that 
\begin{equation}\label{realm}
{\frak m} = \sum_{\alpha \in \Delta_{\fr{m}}^{+}}
 \left\{
{\mathbb R} (E_{\alpha} + E_{-\alpha}) \oplus {\mathbb R}\sqrt{-1} (E_{\alpha}
- E_{-\alpha})\right\}.
\end{equation}
Thus a basis of $\fr{m}$ consists of the vectors $\{A_{\al}=(E_{\alpha} + E_{-\alpha}), \ B_{\al}=\sqrt{-1} (E_{\alpha}
- E_{-\alpha}) : \al\in \Delta_{\fr{m}}^{+}\}$.

   For integers $j_1, \dots, j_r$ with $(j_1, \dots, j_r) \neq (0, \dots, 0)$ we set 
       $$\Delta^{\fr{m}}(j_1, \dots, j_r ) = \left\{ \  \sum_{j=1}^{l} m_j \alpha_j \in \Delta^{+} : m_{i_1} = j_1, \dots, m_{i_r} = j_r  \ \right\}\subset \Delta^{+}. 
       $$
       Then it is $\displaystyle  \Delta_{\frak m}^{+} =  \Delta^+ \backslash\Delta_{0}^{+}= \bigcup_{j_1, \dots, \, j_r} \Delta^{\fr{m}}( j_1, \dots, j_r ) $.    
        For $  \Delta^{\fr{m}}( j_1, \dots, j_r ) \neq \emptyset$  we define an $\mbox{Ad}(K)$-invariant subspace ${\frak m}( j_1, \dots, j_r )$ of $\frak g$ by 
\[
{\frak m}( j_1, \dots,  j_r ) 
 =  \sum_{\alpha \in  \Delta^{\fr{m}}( j_1, \dots, \, j_r )}
 \left\{ {\mathbb R}A_{\al}+\bb{R}B_{\al}\right\}.
 \]
Thus we have a decomposition of $\frak m$ into mutually non equivalent irreducible $\mbox{Ad}_G(K)$-modules  $ {\frak m}( j_1, \dots,  j_r )$ as
  $\frak m = \sum_{j_1, \dots, \, j_r} {\frak m}( j_1, \dots,  j_r ). 
  $

 We consider  the restriction map $ \kappa  :  {\frak h}_0^* \to {\frak t}^*$, $\alpha \mapsto \alpha\vert_{\frak t}$ and note that this is a linear map. We set $\Delta_{\fr{t}} = \kappa(\Delta)$, $\kappa(\Delta_{0})=0$.
  
  \begin{definition}
  The elements of $\Delta_{\fr{t}}$ are called $\fr{t}$-roots. 
  \end{definition}
  
 Let $\fr{m}^{\bb{C}}=T_{o}(G/K)^{\bb{C}}$ be the complexification of $\fr{m}$. Then it is 
 $ \fr{m}^{\bb{C}}=\sum_{\al\in\Delta_{\fr{m}}}\fr{g}_{\al}^{\bb{C}}$  and thus a basis of $\fr{m}^{\bb{C}}$ is given by the root vectors $\{E_{\al} : \al\in\Delta_{\fr{m}}\}$.  
  
 \begin{prop}\label{tcor}\textnormal{(\cite{Ale}, \cite{AP})}
   There exists a 1-1 correspondence between $\fr{t}$-roots $\xi$ and irreducible submodules $\displaystyle {\frak m}_{\xi}$ of the $\mbox{Ad}_G(K)$-module ${\frak m}^{\mathbb C}_{}$ given by 
 $$\displaystyle \Delta_{\fr{t}}  \ni \xi \mapsto   {\frak m}_{\xi} = \sum_{\{\al\in \Delta_{\fr{m}} : \kappa(\alpha) = \xi\}} {\frak g }^{\mathbb C}_{\alpha}.$$ 
 \end{prop}
  By using  Proposition \ref{tcor} and the definition of $\fr{t}$-roots, it follows that  the $\mbox{Ad}_G(K)$-module ${\frak m}^{\mathbb C}_{}$ admits the decomposition ${\frak m}^{\mathbb C}_{} = \sum_{\xi \in \Delta_{\fr{t}}} {\frak m}_{\xi}$.  
  If we  denote by $\Delta_{\fr{t}}^{+}$   the set of all positive $\fr{t}$-roots (this is the restricton of the root system $\Delta^+$ under the map $\kappa$), then 
the nilradical is given by  ${\frak r}  = \sum_{\xi \in \Delta_{\fr{t}}^{+}} {\frak m}_{\xi}$. 

In order to obtain a decomposition of the real $\Ad(K)$-module $\fr{m}$ in terms of $\fr{t}$-roots, we use  the complex conjugation $\tau$ of ${\frak g }^{\mathbb C}_{}$ with respect to $\frak g$ (note that $\tau$ interchanges ${\frak g }^{\mathbb C}_{\alpha}$ and ${\frak g }^{\mathbb C}_{-\alpha}$). 
For a complex subspace $W$ of ${\frak g }^{\mathbb C}_{}$ we denote by 
 $W^{\tau}_{}$  the set of all fixed points of $\tau$.   Then  
 \begin{equation}\label{tangent}
 { \frak m} = \sum_{\xi \in \Delta_{\fr{t}}^{+}} \left( {\frak m}_{\xi} \oplus {\frak m}_{-\xi} \right)^{\tau}.
 \end{equation}
 
 Let $\Delta_{\fr{t}}^{+}=\{\xi_1, \ldots, \xi_q\}$. Then Proposition \ref{tcor} and relations (\ref{realm}), (\ref{tangent}) 
 imply that each real irreducible $\ad(\fr{k})$-submodule $\fr{m}_{i}=(\fr{m}_{\xi_{i}}\oplus  \fr{m}_{-\xi_{i}})^{\tau}$ $(1\leq i\leq q)$
corresponding to the positive $\fr{t}$-root $\xi_i$   is given by
\begin{equation}\label{bas}
\fr{m}_{i}=\sum_{\{\al\in \Delta_{\fr{m}}^+ \ :\ \kappa(\al)=\xi_{i}\}}\left\{
{\mathbb R} (E_{\alpha} + E_{-\alpha}) + {\mathbb R}\sqrt{-1} (E_{\alpha}
- E_{-\alpha})\right\}.
\end{equation}

The results obtained in the previous discussion are summarized in the following:

 \begin{prop}\label{DEC} Let $M=G/K$ be a generalized flag manifold defined by a subset $\Pi_{0}\subset\Pi$ such that $\Pi_{\fr{m}}=\Pi\backslash\Pi_{0}=\{\al_{i_{1}}, \ldots, \al_{i_{r}}\}$ with $1\leq i_{1}\leq\cdots \leq i_{r}\leq \ell$. Assume that $\fr{g}=\fr{k}\oplus\fr{m}$ is a $B$-orthogonal reductive decomposition. Then
 
 1) There exists a natural one-to-one correspondence between elements of the set $\Delta^{\fr{m}}( j_1, \dots, j_r ) \neq \emptyset$ and the set of positive $\fr{t}$-roots $\Delta_{\fr{t}}^{+}=\{\xi_1, \ldots, \xi_{q}\}$.   Thus there is a decomposition of $\frak m$ into $q$ mutually non-equivalent   irreducible $\Ad(K)$-modules 
  \[
  \displaystyle \frak m =  \sum_{\xi \in \Delta_{\fr{t}^{+}}} \left( {\frak m}_{\xi} \oplus {\frak m}_{-\xi} \right)^{\tau} = \sum_{i=1}^{q}(\fr{m}_{\xi_{i}}\oplus  \fr{m}_{-\xi_{i}})^{\tau} =\sum_{j_1, \dots, \, j_r} {\frak m}( j_1, \dots,  j_r ),
  \]
  for appropriate positive integers $j_1, \dots,  j_r$.
  
2) The    dimensions of the real $\Ad(K)$-modules $\fr{m}_{i}$ $(i=1, \ldots, q)$ corresponding to  the $\fr{t}$-root $\xi_i\in\Delta_{\fr{t}}^{+}$ are given by $\dim_{\bb{R}}\fr{m}_{i}=2\cdot|\{ \al\in \Delta_{\fr{m}}^{+} : \kappa(\al)=\xi_{i}\}|=2\cdot |\Delta^{\fr{m}}(j_{1}, \ldots, j_{r})|$, for appropriate positive integers $j_{1}, \ldots, j_{r}$.\footnote{We denote by $|S|$ the cardinality of a finite set $S$.}   
 
 3) Any $G$-invariant Riemannian metric $g$ on $G/K$ can be expressed as  
\begin{eqnarray}
 g  =\sum_{\xi \in \Delta_{\fr{t}}^{+}} x_{\xi} B|_{\left( {\frak m}_{\xi} + {\frak m}_{-\xi} \right)^{\tau}}  = \sum_{i=1}^{q} x_{\xi_{i}} B|_{\left( {\frak m}_{\xi_{i}} + {\frak m}_{-\xi_{i}} \right)^{\tau}}=\sum_{j_1, \dots, j_r}
x_{j_1 \cdots j_r}  B|_{ {\frak m}( j_1, \dots,  j_r )}  \label{eq22}
\end{eqnarray}
for positive real numbers $ x_{\xi}$, $x_{\xi_{i}}$, $ x_{j_1 \cdots j_r}$. Thus $G$-invariant Riemannian metrics on $M=G/K$ are parametrized by $q$ real positive parameters.
\end{prop}

 We now show how we can find explicitly the set of $\fr{t}$-roots $\Delta_{\fr{t}}$. 
   Let 
 $
  \Pi_{\fr{t}}=\{\overline{\al}_{i_{j}}=\al_{i_{j}}|_{\fr{t}} : \al_{i_{j}}\in \Pi_{\fr{m}}\}. 
 $
 This set is a basis  of $\fr{t}^*$  in the sense that any $\fr{t}$-root can be written  as a linear combination of its elements with integer coefficients of the same sign.   
 In particular, by using the fact that $\kappa(\Delta_{0})=0$ we have that
   \begin{eqnarray}\label{troots}
  \kappa(\al)  
  =k_{i_{1}}\overline{\al}_{i_{1}}+\cdots+k_{i_{r}}\overline{\al}_{i_{r}}, \ \ (\al\in\Delta_{\fr{m}}^{+}).
  \end{eqnarray}
  Here the positive integers $k_{i_{j}}$ satisfy 
  $0\leq k_{i_{j}}\leq m_{i_{j}}$, where $m_{i_{j}}$ is the Dynkin mark of the simple root  $\al_{i_{j}}\in\Pi_{\fr{m}}$,
  and are not simultaneously zero.
 Therefore, by using the expressions of  the complementary roots in terms of simple roots,
  and applying formula (\ref{troots}), we can easily determine all positive $\fr{t}$-roots.  Elements  of $\Pi_{\fr{t}}$ are called {\it simple $\fr{t}$-roots} and they generalize the notion of simple roots  (this means that a simple $\fr{t}$-root $\kappa(\al_{i_{j}})=\al_{i_{j}}|_{\fr{t}}=\overline{\al}_{i_{j}}\in\Pi_{\fr{t}}$ is a positive $\fr{t}$-root, which can not be written as the sum of two positive $\fr{t}$-roots).

   \begin{example}\label{ISO2}{\sc \bf Flag manifolds  of $C_{\ell}=\Sp (\ell)$ (\cite{ACS}).
\textnormal{Consider the flag manifolds  $M=G/K= \Sp(\ell)/(\U(p)\times \U(q)\times \Sp(\ell-p-q ))$ with $\ell \geq 3$ and 
$1\le p, q, p+q \leq \ell -1$.  
This space  is defined by the painted Dynkin diagram $\Gamma(\Pi_{\fr{m}})$ with  $\Pi_{\fr{m}} = \{\al_{p}, \al_{p+q} : \Mark(\al_{p})=\Mark(\al_{p+q})=2 \}$ that is 
{{\small   \begin{center} 
   \vspace{10pt}
{\small
\begin{picture}(200,0)(0,-2)
\put(0, 0){\circle{4}}
\put(0,8.5){\makebox(0,0){$\alpha_1$}}
\put(2, 0){\line(1,0){10}}
\put(20, 0){\makebox(0,0){$\ldots$}}
\put(28, 0){\line(1,0){10}}
\put(40, 0){\circle{4.4}}
\put(40,8.5){\makebox(0,0){$\alpha_{p-1}$}}
\put(42, 0.3){\line(1,0){14}}
\put(58, 0){\circle*{4.4}}
\put(58,8.5){\makebox(0,0){$\alpha_p$}}
\put(58, 0){\line(1,0){14}}
\put(74, 0){\circle{4.4}}
\put(76, 0){\line(1,0){10}}
\put(94, 0){\makebox(0,0){$\ldots$}}
\put(102, 0){\line(1,0){10}}
\put(114, 0){\circle{4}}
\put(132,8.5){\makebox(0,0){$\alpha_{p+q}$}}
\put(116, 0){\line(1,0){14}}
\put(132, 0){\circle*{4}}
\put(134, 0){\line(1,0){14}}
\put(150, 0){\circle{4}}
\put(152, 0){\line(1,0){8}}
\put(170, 0){\makebox(0,0){$\ldots$}}
\put(180, 0){\line(1,0){8}}
\put(190, 0){\circle{4}}
\put(190,8.5){\makebox(0,0){$\alpha_{\ell-1}$}}
\put(194.5, 1.2){\line(1,0){14}}
\put(194.5, -1.2){\line(1,0){14}}
\put(191.16, -1.9){\scriptsize $<$}
\put(210.5, 0){\circle{4}}
\put(210,8.5){\makebox(0,0){$\alpha_\ell$}}
\put(215.5,-1){\makebox(0,0){.}}  
\end{picture}
\vspace{10pt}
}\end{center}}}
 \noindent and thus $M$ is of Type  B (cf. Introduction).
   We will show that   $\fr{m}=T_{o}M$ decomposes into a direct sum of six pairwise inequivalent $\Ad(K)$-submodules, thus $M$ does not appear in Table 2. 
   } 
 \textnormal{Following the notation of \cite[p. 3781]{AA},  we consider an orthonormal basis of $\bb{R}^{\ell}$ given by $\{e^1_i, e^2_{j}, \pi_{k}\}$  with $1\leq i\leq p$ and $1\leq j\leq q$ and $1\leq k\leq \ell-p-q$. Then  a system of positive roots   for $C_{\ell}$  is given by  
\begin{eqnarray*}
 \Delta^{+} &=& \{e^1_i\pm e^1_{j} : 1\leq i<j\leq p\}\cup\{e^{1}_{i}\pm e_{j}^{2} : 1\leq i\leq p, \ 1\leq j\leq q\}\cup \ \{e^{2}_{i}\pm e^{2}_{j} : 1\leq i<j\leq q\} \\ 
 && \ \cup \ \{e^1_i\pm\pi_{k} : 1\leq i\leq p, \ 1\leq k\leq \ell-p-q\} \cup\{e_{j}^{2}\pm\pi_{k} : 1\leq j\leq q, \ 1\leq k\leq \ell-p-q \}\\
 && \ \cup \ \{\pi_{i}\pm\pi_{j} : 1\leq i<j\leq\ell-p-q\}\cup\{2e^{1}_{i},  \ 2e^{2}_{j}, \  2\pi_{k} \}.
 \end{eqnarray*}
A  basis of simple roots is given by
   \begin{eqnarray*}
   \Pi &=&\{\al_1=e_1^1-e^1_2,\ldots,  \al_{p-1}=e^1_{p-1}-e^{1}_{p}, \ \al_{p}=e^{1}_{p}-e^{2}_{1}\} \\
   && \ \cup \ \{\al_{p+1}=e^{2}_{1}-e^{2}_{2}, \ldots, \al_{p+q-1}=e^{2}_{q-1}-e^{2}_{q}, \ldots, \al_{p+q}=e^{2}_{q}-\pi_{1}\} \\
   && \ \cup \ \{ \phi_1=\pi_1-\pi_2,  \ldots, \phi_{\ell-p-q-1}=\pi_{\ell-p-q-1}-\pi_{\ell-p-q}, \ \phi_{\ell-p-q}=2\pi_{\ell-p-q}\}.
   \end{eqnarray*}
The root system of the semisimple part of the isotropy subgroup $K$ is given by 
\[
\Delta_{0}^{+}=\{e^1_{i}-e^1_{j} : 1\leq i< j\leq p\}\cup\{e^2_{i}-e^2_{j} : 1\leq i< j\leq q\}\cup\{\ \pi_{i}\pm\pi_{j}, \ 2\pi_{k} : 1\leq i<j\leq\ell-p-q\},
\]
thus the positive complementary roots are
  $
  \Delta_{\fr{m}}^{+}=\{e^{1}_i+ e^{1}_{j}, \ e^1_{i}\pm e^{2}_{j}, \ e^{2}_i+ e^{2}_{j}, \ e^1_i\pm\pi_{k}, \ e^{2}_{j}\pm\pi_{k}, \ 2e^1_i, \  2e^{2}_{j}\}.
  $
 Let  $\al=\sum_{k=1}^{\ell}c_{k}\al_{k}\in \Delta_{\fr{m}}^{+}$.  Since $\Pi_{\fr{m}}=\{\al_{p}, \al_{p+q}\}$, by applying  relation (\ref{troots})  we obtain  that   $\kappa(\al)=c_{p}\overline{\al}_{p}+c_{p+q}\overline{\al}_{p+q}\in\Delta_{\fr{t}}^{+}$.  Here the  coefficients $c_{p}, c_{p+q}$ are such that $0\leq c_p, c_{p+q}\leq 2$,   and they are not simultaneously equal to zero. 
   In particular,  by expressing the positive complementary roots in terms of   the simple roots we obtain that
  \begin{eqnarray*}
\kappa(e^1_i+e^1_j) &=& \kappa(\al_{i}+\cdots+\al_{j}+2\al_{j+1}+\cdots+2\al_{p}+\cdots+2\al_{p+q}+2\phi_{1}+\cdots+ \phi_{\ell-p-q}) \\
  &=& 2\overline{\al}_p+2\overline{\al}_q, \\
\kappa(e^1_i+e^2_j) &=& \kappa(\al_{i}+\cdots+\al_{p}+\cdots+\al_{p+j}+2\al_{p+j+1}+\cdots+2\al_{p+q}+2\phi_{1}+\cdots+ \phi_{\ell-p-q}) \\
&=& \overline{\al}_p+2\overline{\al}_{p+q}, \\
\kappa(e^1_i-e^2_j) &=& \kappa(\al_{i}+\cdots+ \al_{p}+\cdots +\al_{p+j})= \overline{\al}_{p}, \\
\kappa(e^2_i+e^2_j) &\stackrel{i>p}{=}& \kappa(\al_{i}+\cdots+\al_{j}+2\al_{j+1}+\cdots+2\al_{p+q}+2\phi_{1}+\cdots+\phi_{\ell-p-q})=  2\overline{\al}_{p+q}, \\
\kappa(e^1_i-\pi_k) &=& \kappa(\al_i+\cdots + \al_{p}+\cdots+\al_{p+q}+\phi_{1}+\cdots + \phi_{k-1})=\overline{\al}_p+\overline{\al}_{p+q} \\
\kappa(e^1_i+\pi_k) &=& \kappa(\al_i+\cdots + \al_{p}+\cdots+\al_{p+q}+\phi_{1}+\cdots +\phi_{k-1}+2\phi_{k}+\cdots +\phi_{\ell-p-q}) = \overline{\al}_p+\overline{\al}_{p+q}, \\
 \kappa(e^2_j-\pi_k) &\stackrel{j>p}{=}& \kappa(\al_{j}+\cdots\al_{p+q}+\phi_{1}+\cdots+\phi_{k-1}) =\overline{\al}_{p+q}, \\
 \kappa(e^2_j+\pi_k) &\stackrel{j>p}{=}& \kappa(\al_{j}+\cdots\al_{p+q}+\phi_{1}+\cdots+\phi_{k-1}+2\phi_{k}+\cdots +\phi_{\ell-p-q}) =\overline{\al}_{p+q}, \\  
\kappa(2e^1_i) &=& \kappa(2\al_{i}+\cdots+2\al_{p}+\cdots+2\al_{p+q}+2\phi_{1}+\cdots+2\phi_{\ell-p-q-1}+\phi_{\ell-p-q})= 2\overline{\al}_{p}+2\overline{\al}_{p+q}, \\
\kappa(2e^2_j) &\stackrel{j>p}{=}& \kappa(2\al_{j}+\cdots+2\al_{p+q}+2\phi_{1}+\cdots+2\phi_{\ell-p-q-1}+\phi_{\ell-p-q})=2\overline{\al}_{p+q}.
\end{eqnarray*}
Thus the  set of $\fr{t}$-roots is given by 
$\Delta_{\fr{t}}^+ = \{\overline{\al}_{p}, \ \overline{\al}_{p+q}, \ \overline{\al}_{p}+\overline{\al}_{p+q}, 
 \ 2\overline{\al}_{p+q},  \ \overline{\al}_{p} + 2 \overline{\al}_{p+q},   \ 2 \overline{\al}_{p} + 2 \overline{\al}_{p+q} \}$ and 
  according to Proposition \ref{DEC}  (1) we obtain the decomposition
 $\fr{m} = \fr{m}_1\oplus\fr{m}_{2}\oplus\fr{m}_{3}\oplus\fr{m}_{4}\oplus\fr{m}_{5}\oplus\fr{m}_{6}  = \fr{m}(1, 0)\oplus\fr{m}(0, 1)\oplus\fr{m}(1, 1)\oplus\fr{m}(0, 2)\oplus\fr{m}(1, 2)\oplus\fr{m}(2, 2)$.   
 Note that the exception  $p+q = \ell$ determines the space $M= \Sp(\ell)/(\U(p)\times \U(q) )$  with  four isotropy summands,   since in this case the $\fr{t}$-roots are given by 
$\Delta_{\fr{t}}^+ = \{\overline{\al}_{p}, \ \overline{\al}_{\ell}, \ \overline{\al}_{p}+\overline{\al}_{\ell}, 
  \ 2 \overline{\al}_{p} +  \overline{\al}_{\ell} \}$ (\cite{Chry2}). } }
  \end{example}


\section{Flag manifolds with five isotropy summnads}

 \subsection{On the isotropy represantation of flag manifolds}
 
 Proposition \ref{DEC}  provides all the necessary ingredients for the 
classification of flag manifolds with a certain number of isotropy summands. 
However, we essentially need to work on a case by case basis, which means that in the Dynkin diagram $\Gamma(\Pi)$ of each simple Lie group $G$ we need to   paint  black all possible combinations of roots of certain 
 Dynkin marks. 
 	A systematic approach for flag manifolds determined by a classical Lie group is given  
in \cite{Ale2}, but be aware of certain misprints.  Recall that isotropy irreducible flag manifolds are the isotropy irreducible Hermitian symmetric spaces of compact type, and are determined by painting  black  exactly one simple root of Dynkin mark 1.  
Flag manifolds with two  isotropy summnads were classified in \cite{AChry}.  These spaces are determined  by pairs $(\Pi, \Pi_{0})$ such that $\Pi\backslash\Pi_{0}=\{\al_{p} : \Mark(\al_{p})=2\}$.  Flag manifolds with three isotropy summands were classified in \cite{Kim}, where it was shown that such spaces are defined  by  pairs $(\Pi, \Pi_{0})$ with that $\Pi\backslash\Pi_{0}=\{\al_{p} : \Mark(\al_{p})=3\}$ or $\Pi\backslash\Pi_{0}=\{\al_{i}, \al_{j} :  \Mark(\al_{i})=\Mark(\al_{j})=1\}$.  Finally, the classification of all flag manifolds with four isotropy summands was given in \cite{Chry2}, where  it  was shown that such spaces are determined by pairs $(\Pi, \Pi_{0})$ such that $\Pi\backslash\Pi_{0}=\{\al_{p} : \Mark(\al_{p})=4\}$ or $\Pi\backslash\Pi_{0}=\{\al_{i}, \al_{j} :   \Mark(\al_{i})=1, \ \Mark(\al_{j})=2\}$ (however the correpsondence with the second type of pairs is not one-to-one, see the Introduction). 

In this section we will prove that the only generalized flag manifolds $G/K$  (different from the space $\E_8/\U(1)\times \SU(4)\times \SU(5)$), whose isotropy representation decomposes into five isotropy summands  are the spaces determined  by the pairs $(\Pi, \Pi_{0})$ of Types A and B presented in Table 1  in the Introduction. 
Furthermore, we will show that these pairs define {\it isometric flag manifolds} (as real manifolds), in the sense that {\it there is an isometry which permutes the associated isotropy summands and identifies the different reductive decompositions} (which are defined by the different pairs $(\Pi, \Pi_{0})$).  Therefore, our study focuces at the spaces listed in Table 2.
  This isometry (which is induced by the action of the associated Weyl group  on the root  system of $G$) enables us to  study the classification problem of homogeneous Einstein metrics only for one possible pair $(\Pi, \Pi_{0})$, therefore we will only work with flag manifolds
  of Type A.

\subsection{The classification of flag manifolds with five isotropy summnads}

As mentioned in the Introduction, all  flag manifolds of Types A and B are such that $b_{2}(M)=2$, which means that  
$\Pi_{\fr{m}}=\Pi\backslash\Pi_{0}$ contains only two simple roots. 
 For  convenience of the reader in Table 3 we present all possible pairs $(\Pi, \Pi_{0})$ which determine flag manifolds with $b_{2}(M)=2$ (for completeness we also  include  those which determine flag  manifolds with $b_{2}(M)=1$).    
   In this table, the first column contains the first Betti number, the second column indicates the Dynkin marks of the roots
 painted black, the third column shows the number $q$ of isotropy summands of the  flag manifolds obtained (in some cases there are more than one possibilities),
 the fourth column shows for which Lie groups can arise such pairs (and thus such flag manifolds),  and the last column  gives references for the homogeneous Einstein metrics on the corresponding spaces.
 
 \medskip
{\small{\begin{center}
{\bf Table 3.} \ {\small The isotropy representation and Einstein metrics  on flag manifolds $M=G/K$ with $b_{2}(M)=1$ or $2$ }  
\end{center}
\begin{center}
 \begin{tabular}{c|l|l|l|l}
  $b_{2}(M)$ &  Dynkin marks of $\Pi\backslash\Pi_{0}$   & $\fr{m}= \bigoplus_{i=1}^{q}\fr{m}_{q}$    & Type of $G$  & Einstein metrics \\
                \thickline
1 &  $\Mark(\al_{p})=1$                 & $q=1$           & Irred. Symmetric space  & \cite{Hel} \\
1 & $\Mark(\al_{p})=2$                  & $q=2$           & $B_{\ell}, C_{\ell}, D_{\ell}, \G_2, \F_4, \E_6, \E_7, \E_8$  & \cite{Chry1} \\
1 & $\Mark(\al_{p})=3$                  & $q=3$           & $\G_2, \F_4, \E_6, \E_7, \E_8$ & \cite{Kim}, \cite{AC} \\
1 & $\Mark(\al_{p})=4$                  & $q=4$           & $\F_4, \E_7, \E_8$ & \cite{Chry2} \\
1 & $\Mark(\al_{p})=5$                  & $q=5$           & $\E_8$        & \cite{CS} \\
1 & $\Mark(\al_{p})=6$                  & $q=6$           & $\E_8$        & \cite{CS}, open  \\
2 & $\Mark(\al_{p})=1, \Mark(\al_{q})=1$ & $q=3$           & $A_{\ell}, D_{\ell}, \E_6$ & \cite{Kim} \\
2 & $\Mark(\al_{p})=1, \Mark(\al_{q})=2$ & $q=4, 5$        & $B_{\ell}, C_{\ell},  D_{\ell}, \E_6, \E_7$ & \cite{Chry2}, \cite{ACS1}, \cite{ACS2} \\
2 & $\Mark(\al_{p})=1, \Mark(\al_{q})=3$ & $q=6$           & $\E_6, \E_7$ & open \\ 
2 & $\Mark(\al_{p})=1, \Mark(\al_{q})=4$ & $q=8$           & $\E_7$ & open \\ 
2 & $\Mark(\al_{p})=2, \Mark(\al_{q})=2$ & $q=5, 6$        &  $B_{\ell}, C_{\ell}, D_{\ell},   \F_4, \E_6, \E_7, \E_8$ &  \cite{ACS}, \cite{Chry}, open \\ 
2 & $\Mark(\al_{p})=2, \Mark(\al_{q})=3$ & $q=6, 7, 8$   
  & $\G_2, \F_4, \E_6, \E_7, \E_8$ &  \cite{ACS3},  open \\
2 & $\Mark(\al_{p})=2, \Mark(\al_{q})=4$ & $q= 7, 8, 9$      & $\F_4, \E_7, \E_8$ &    open \\ 
2 & $\Mark(\al_{p})=2, \Mark(\al_{q})=5$ & $q=10$           & $\E_8$ &    open \\
2 & $\Mark(\al_{p})=2, \Mark(\al_{q})=6$ & $q=11$           & $\E_8$ &    open \\
2 & $\Mark(\al_{p})=3, \Mark(\al_{q})=3$ & $q=8, 9$         & $\E_7, \E_8$ &    open \\
2 & $\Mark(\al_{p})=3, \Mark(\al_{q})=4$ & $q=8, 9, 10$     & $\F_4, \E_7, \E_8$ &    open \\
2 & $\Mark(\al_{p})=3, \Mark(\al_{q})=5$ & $q=10, 11$       & $\E_8$ &    open \\
2 & $\Mark(\al_{p})=3, \Mark(\al_{q})=6$ & $q=10, 14$       & $\E_8$ &    open \\
2 & $\Mark(\al_{p})=4, \Mark(\al_{q})=4$ & $q=12$           & $\E_8$ &    open \\
2 & $\Mark(\al_{p})=4, \Mark(\al_{q})=5$ & $q=10, 11$       & $\E_8$ &    open \\
2 & $\Mark(\al_{p})=4, \Mark(\al_{q})=6$ & $q=11, 14$       & $\E_8$ &    open \\
2 & $\Mark(\al_{p})=5, \Mark(\al_{q})=6$ & $q=12$           & $\E_8$ &    open \\
 \end{tabular}
 \end{center}}}
\medskip

In order to show that flag manifolds  with five isotropy summnands (different from the space $\E_8/\U(1)\times \SU(4)\times \SU(5)$ with $b_{2}(M)=1$)  are determined only by the pairs $(\Pi, \Pi_{0})$ presented in Table 1 of the Introduction, we proceed into two steps. 
First we show that the spaces determined by the pairs in Table 1 have in fact five isotropy summands. 
Next, we prove that the other existing  pairs $(\Pi, \Pi_{0})$ of Type B  determine flag manifolds whose positive $\fr{t}$-root system $\Delta_{\fr{t}}^{+}$ contains more than five elements.  
Note that pairs of Type A are excluded form the study due to \cite[Prop. 5]{Chry2} and the first step.
All other pairs $(\Pi, \Pi_{0})$ such that $\Pi\backslash\Pi_{0}=\{\al_{i}, \al_{j}\}$ and different from Types A and B 
can be treated in a similar manner, so we refer
to \cite[Prop. 6]{Chry2} for further details and Table 3 the final results. 

We need the following useful

\begin{lemma}\label{b2}  Generalized flag manifolds $M=G/K$ with $b_2(M)\ge 3$ have more than five isotropy summands. 
\end{lemma}

\begin{proof}
At first  we consider the case of $b_2(M)=3$ and  the simple Lie groups $A_{\ell}$. 
We assume that the subset $\Pi_{0}\subset\Pi$ is such that $\Pi_{\fr{m}}=\Pi\backslash\Pi_{0}= \{\al_{i}, \al_{j}, \ \al_{k} \}$ where $i, j, k$ are different each other. Then  $\fr{t}$ is 3-dimensional and  a $\fr{t}$-basis is given by $\Pi_{\fr{t}}=\{\overline{\al}_{i}=\al_{i}\big|_{\fr{t}}, \overline{\al}_{j}=\al_{j}\big|_{\fr{t}}, \overline{\al}_{k}=\al_{k}\big|_{\fr{t}}\}$ with  $\Mark(\al_{i})=\Mark(\al_{j})=\Mark(\al_{k})=1$.   Let $\al=\sum_{p=1}^{\ell}c_{p}\al_{p}\in \Delta^{+}_{\fr{m}}$ be a positive complementary root. Then, by applying (\ref{troots}) we conclude  that any positive $\fr{t}$-root is given by 
$\kappa(\al)=c_{i}\overline{\al}_{i}+c_{j}\overline{\al}_{j}+c_{k}\overline{\al}_{k}$, where $0\leq c_{i}, c_{j}, c_{k}\leq 1$ cannot be simultaneosuly equal to zero and we see   that   the system $\Delta_{\fr{t}}^{+}$ consists of the $\fr{t}$-roots 
$
\overline{\al}_{i}, \ \overline{\al}_{j}, \ \overline{\al}_{k}, \ \overline{\al}_{i}+\overline{\al}_{j},\  \overline{\al}_{j}+\overline{\al}_{k}, \ \overline{\al}_{i}+\overline{\al}_{j}+\overline{\al}_{k}$. 
Thus $|\Delta_{\fr{t}}^{+}|=6$ and $\fr{m}=T_{o}M$  decomposes into more than five isotropy summands.  
 If  $b_2(M)>3$, then   the system $\Delta_{\fr{t}}^{+}$ contains the $\fr{t}$-roots of the form 
$
\overline{\al}_{i}, \ \overline{\al}_{j}, \ \overline{\al}_{k}, \ \overline{\al}_{i}+\overline{\al}_{j},\  \overline{\al}_{j}+\overline{\al}_{k}, \ \overline{\al}_{i}+\overline{\al}_{j}+\overline{\al}_{k}$, and hence  $|\Delta_{\fr{t}}^{+}|>5$. 

\noindent Now consider the case when the Dynkin  diagram of a simple Lie algebra contains the Dynkin subdiagram of type  $A_{m}$ and it contains these roots $  \{\al_{i}, \al_{j}, \ \al_{k} \}$. Then we find that $|\Delta_{\fr{t}}^{+}|>5$. The other cases  are $ D_{\ell}$ with $ \{\al_{i}, \al_{\ell -1}, \ \al_{\ell} \}$,
$\E_6$, $\E_7$ and $\E_8$. But for the case of $ D_{\ell}$  with $ \{\al_{i}, \al_{\ell -1}, \ \al_{\ell} \}$ we see that $|\Delta_{\fr{t}}^{+}|>5$. 
If $\E_6$, $\E_7$ or $\E_8$ contains the Dynkin subdiagram of type $ D_{m}$ which contains these roots $\{\al_{i}, \al_{m-1}, \ \al_{m} \}$, then it follows  that  $|\Delta_{\fr{t}}^{+}|>5$. For the case  $\E_{6}$ with $ \{\al_{1}, \al_{5}, \ \al_{6} \}$ we get also $|\Delta_{\fr{t}}^{+}|>5$. If $\E_{7}$  or $\E_8$ contains the Dynkin subdiagram of type $\E_{6}$ which contains these roots $\{\al_{i}, \al_{j}, \ \al_{k} \}$, then we see that  $|\Delta_{\fr{t}}^{+}|>5$. The remaining cases are $\E_{7}$ with $ \{\al_{1}, \al_{6}, \ \al_{7} \}$ and $\E_8$ with $ \{\al_{i}, \al_{7}, \ \al_{8} \}$ where $i = 1,2, 3$. For these cases  one can easily prove  that  $|\Delta_{\fr{t}}^{+}|>5$.  
\end{proof}

Thus flag manifolds with five isotropy summands are determined by pairs $(\Pi, \Pi_0)$ with 
$|\Pi\setminus\Pi _0|=2$.

\begin{prop}\label{classif}
Let $G$ be a compact connected simple Lie group and let $\Pi=\{\al_1, \ldots, \al_{\ell}\}$ be a system of simple roots of the associated root system of $G$.  Consider a subset $\Pi_{0}\subset \Pi$    of simple roots, such that  $\Pi\backslash\Pi_{0}$  contains  exactly  two simple roots.  Then, the only pairs $(\Pi, \Pi_{0})$ which determine flag manifolds $G/K$ whose isotropy representation decomposes into five pairwise inquivalent irreducible $\Ad(K)$-submodules are the pairs of Type A and B presented in Table 1 of the Introduction.
\end{prop}

\begin{proof}

 {\bf Step 1.}   We follow the notation of \cite{AA} or \cite{Chry2} (see also \cite{GOV}) for the root systems of the simple Lie algebras, their  fundamental systems of simple roots and the associated expressions of the highest root $\tilde{\al}$, 

\smallskip 
\underline{\bf Case of $B_{\ell}=\SO(2\ell+1)$ : Type A.}  
Let  $\Pi\backslash\Pi_{0}=\Pi_{\fr{n}}=\{\al_1, \al_{p+1}: 2\le p\le \ell -1\}$.  This choice corresponds to the painted Dynkin diagram 
  \begin{center} 
{\small
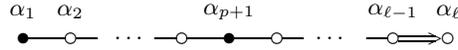
{
\begin{picture}(154,0)(0,-2)
\put(0, 0){\circle*{4}}
\put(0,10){\makebox(0,0){$\al_1$}}
\put(2, 0){\line(1,0){14}}
\put(18, 0){\circle{4}}
\put(20, 0){\line(1,0){8}}
  \put(18,10){\makebox(0,0){$\al_2$}}
 \put(40, 0){\makebox(0,0){$\ldots$}}
 \put(50, 0){\line(1,0){8}}
\put(60, 0){\circle{4.4}}
\put(78, 10){\makebox(0,0){$\al_{p+1}$}}
\put(62, 0){\line(1,0){14}}
\put(78, 0){\circle*{4.4}}
\put(80, 0){\line(1,0){14}}
\put(96, 0){\circle{4}}
\put(98, 0){\line(1,0){8}}
\put(116, 0){\makebox(0,0){$\ldots$}}
\put(130, 0){\line(1,0){8}}
\put(140, 0){\circle{4}}
\put(140, 10){\makebox(0,0){$\al_{\ell-1}$}}
\put(142, 1){\line(1,0){14}}
\put(142, -1){\line(1,0){14}}
\put(152.7, -1.9){\scriptsize $>$}
\put(160.7, 0){\circle{4}}
\put(161, 10){\makebox(0,0){$\al_\ell$}}
\end{picture}
} } 
\end{center}
which determines the flag manifold    $M= \SO(2\ell+1)/(\U(1)\times \U(p) \times \SO(2(\ell-p-1)+1))$ with  $2\le p\le \ell -1$ and $\ell\geq 3$.
  Let $\fr{n}$ be a $B$-orthogonal complement  of the isotropy subalgebra $\fr{k}$ in $\fr{g}=\fr{so}(2\ell+1)$, that is $\fr{g}=\fr{k}\oplus\fr{n}$ with $[\fr{k}, \fr{n}]\subset\fr{n}$. We will prove that the  $\Ad(K)$-module  $\fr{n}\cong T_{o}M$  decomposes into a direct sum of five  non equivalent $\Ad(K)$- submodules $\fr{n}_{i}$
  $(i=1,\dots ,5)$ whose dimensions are given by (\ref{dim1}).   
   
   Let $\{e^1_1, e^2_{i}, \pi_{j}\}$ be an orthonormal basis of $\bb{R}^{\ell}$ with $1\leq i\leq p$ and $1\leq j\leq \ell-p-1$.  The   positive root system $\Delta^{+}$ of $\SO(2\ell+1)$ is given by (see \cite{AA})
\begin{eqnarray*}
 \Delta^{+} &=& \{e^1_1\pm e^2_{i} : 1\leq i\leq p\}\cup\{\ e^{2}_{i}\pm e^{2}_{j} : 1\leq i<j\leq p\}  \cup \{e^1_1\pm\pi_{j} : 1\leq j\leq \ell-p-1\} \\ 
 && \ \cup \ \{e^{2}_{i}\pm\pi_{j}, : 1\leq i\leq p,  \ 1\leq j\leq \ell-p-1\} 
  \cup  \{\pi_{i}\pm\pi_{j} : 1\leq i<j\leq \ell-p-1\} \\ 
  && \ \cup \ \{\ e^{1}_{1}, \ e^{2}_{i} : 1\leq i\leq p\}\cup\{\pi_{j} : 1\leq j\leq\ell-p-1\}.
 \end{eqnarray*}
 We will denote a  basis of simple roots by
   \begin{eqnarray*}
   \Pi(\fr{n}) &=&\{\al_1=e_1^1-e_2^1, \ \al_2=e^2_1-e^2_2, \ldots,  \al_{p}=e^2_{p-1}-e^{2}_{p}, \ \al_{p+1}=e^{2}_{p}-\pi_1,  \\
   && \ \ \ \ \phi_1=\pi_1-\pi_2,  \ldots, \phi_{\ell-p-2}=\pi_{\ell-p-2}-\pi_{\ell-p-1}, \ \phi_{\ell-p-1}=\pi_{\ell-p-1}\}.
   \end{eqnarray*}
 It is $\Delta_{o}^{+}=\{e^2_{i}-e^2_{j}, \ \pi_{i}\pm\pi_{j}, \ \pi_{j}\}$ and thus the positive complementary roots are given by $\Delta_{\fr{n}}^{+}=\{e^{1}_1\pm e^{2}_{i}, \ e^2_{i}+e^{2}_{j}, \ e^1_1\pm\pi_{j}, e^{2}_{i}\pm\pi_{j}, \ e^1_1, e^{2}_{i}\}$. 
   According to  (\ref{troots})   for any $\al=\sum_{k=1}^{\ell}c_{k}\al_{k}\in \Delta_{\fr{n}}^{+}$ it will be
  $\kappa(\al)=c_1\overline{\al}_1+c_{p+1}\overline{\al}_{p+1}$ with $0\leq c_1\leq 1$ and  $0\leq c_{p+1}\leq 2$.   In particular, by expressing the complementary roots in terms of  simple roots we  obtain that 
{\small{  \begin{eqnarray*}
\kappa(e^1_1-e^2_i) &=& \kappa(\al_1+\al_2+\cdots+\al_{i})=\overline{\al}_1, \\
\kappa(e^1_1+e^2_i) &=& \kappa(\al_1+\cdots+\al_{i}+2\al_{i+1}+\cdots 2\al_{p+1}+2\phi_1+\cdots+2\phi_{\ell-p-1})=\overline{\al}_1+2\overline{\al}_{p+1}, \\
\kappa(e^2_i+e^2_j) &=& \kappa(\al_{i+1}+\cdots\al_{j}+2\al_{j+1}+\cdots 2\al_{p+1}+2\phi_1+\cdots +2\phi_{\ell-p-1})= 2\overline{\al}_{p+1}, \\
\kappa(e^1_1-\pi_j) &=& \kappa(\al_1+\cdots + \al_{p+1}+\phi_1+\cdots\phi_{j-1})=\overline{\al}_1+\overline{\al}_{p+1} \\
\kappa(e^1_1+\pi_j) &=& \kappa(\al_1+\cdots+\al_{p+1}+\phi_1+\cdots+\phi_{j-1}+2\phi_{j}+2\phi_{j+1}+\cdots+2\phi_{\ell-p-1})=\overline{\al}_1+\overline{\al}_{p+1}, \\
 \kappa(e^2_i+\pi_j) &=& \kappa(\al_{i+1}+\cdots+\al_{p+1}+\phi_1+\cdots+\phi_{j-1}+2\phi_{j}+\cdots+2\phi_{\ell-p-1}) =\overline{\al}_{p+1}, \\
 \kappa(e^2_i-\pi_j) &=& \kappa(\al_{i+1}+\cdots +\al_{p+1}+\phi_1+\cdots +\phi_{j-1}) =\overline{\al}_{p+1}, \\  
\kappa(e^1_1) &=& \kappa(\al_1+\cdots +\phi_{\ell-p-1})= \overline{\al}_1+\overline{\al}_{p+1}, \\
\kappa(e^2_i) &=& \kappa(\al_{i+1}+\cdots+\al_{p+1}+\phi_1+\cdots +\phi_{\ell-p-1})=\overline{\al}_{p+1}.
\end{eqnarray*} }}
Let $\Delta(\fr{n})_{\fr{t}}^{+}$ be the associated system of positive $\fr{t}$-roots. Then it follows that
$\Delta(\fr{n})_{\fr{t}}^+=\{\overline{\al}_{1}, \ \overline{\al}_{p+1}, \ \overline{\al}_{1}+\overline{\al}_{p+1}, 
 \ 2\overline{\al}_{p+1},  \ \overline{\al}_{1}+2\overline{\al}_{p+1}\}$, and thus according to Proposition \ref{DEC}  (1)  we obtain the decomposition
 \begin{equation}\label{n}
 \fr{n} = \fr{n}_1\oplus\fr{n}_{2}\oplus\fr{n}_{3}\oplus\fr{n}_{4}\oplus\fr{n}_{5}  = \fr{n}(1, 0)\oplus\fr{n}(0, 1)\oplus\fr{n}(1, 1)\oplus\fr{n}(0, 2)\oplus\fr{n}(1, 2),
 \end{equation}
 where the submodules $\fr{n}_{i}$ are given by
\begin{equation}\label{ni}
\left.\begin{tabular}{lll}
 $\fr{n}_{1}=\fr{n}(1, 0)$ &=& $\displaystyle\sum_{\alpha \in  \Delta^{\fr{n}}(1, 0)}
 \left\{ {\mathbb R}A_{\al}+ {\mathbb R}B_{\al}\right\}=\sum_{\al\in \Delta_{\fr{n}}^+ \ :\ \kappa(\al)=\overline{\al}_{1}}\left\{
{\mathbb R} A_{\al} + {\mathbb R}B_{\al}\right\}$ \\
$\fr{n}_{2}=\fr{n}(0, 1)$ &=&  $\displaystyle\sum_{\alpha \in  \Delta^{\fr{n}}(0, 1)}
 \left\{ {\mathbb R}A_{\al}+ {\mathbb R}B_{\al}\right\}=\sum_{\al\in \Delta_{\fr{n}}^+ \ :\ \kappa(\al)=\overline{\al}_{p+1}}\left\{
{\mathbb R} A_{\al} + {\mathbb R}B_{\al}\right\}$ \\
$\fr{n}_{3}=\fr{n}(1, 1)$ &=&  $\displaystyle\sum_{\alpha \in  \Delta^{\fr{n}}(1, 1)}
 \left\{ {\mathbb R}A_{\al}+ {\mathbb R}B_{\al}\right\}=\sum_{\al\in \Delta_{\fr{n}}^+ \ :\ \kappa(\al)=\overline{\al}_{1}+\overline{\al}_{p+1}}\left\{
{\mathbb R} A_{\al} + {\mathbb R}B_{\al}\right\}$ \\
$\fr{n}_{4}=\fr{n}(0, 2)$ &=&  $\displaystyle\sum_{\alpha \in  \Delta^{\fr{n}}(0, 2)}
 \left\{ {\mathbb R}A_{\al}+ {\mathbb R}B_{\al}\right\}=\sum_{\al\in \Delta_{\fr{n}}^+ \ :\ \kappa(\al)=2\overline{\al}_{p+1}}\left\{{\mathbb R} A_{\al} + {\mathbb R}B_{\al}\right\}$ \\
$\fr{n}_{5}=\fr{n}(1, 2)$ &=& $\displaystyle \sum_{\alpha \in  \Delta^{\fr{n}}(1, 2)}
 \left\{ {\mathbb R}A_{\al}+ {\mathbb R}B_{\al}\right\}=\sum_{\al\in \Delta_{\fr{n}}^+ \ :\ \kappa(\al)=\overline{\al}_1+2\overline{\al}_{p+1}}\left\{
{\mathbb R} A_{\al} + {\mathbb R}B_{\al}\right\}$
\end{tabular}\right\}
\end{equation}
The   sets $\Delta^{\fr{n}}(j_{1}, j_{2})$  are given by
  \begin{equation}\label{D1}
   \left. \begin{tabular}{lll}
  $\Delta^{\fr{n}}(1, 0) $ &=& $\{e^1_1-e^2_i : 1\leq i\leq p\}$,\\
  $\Delta^{\fr{n}}(0, 1) $ &=& $ \{e^2_i\pm\pi_j : 1\leq i\leq p, \ 1\leq j\leq \ell-p-1\}\cup\{e^{2}_{i} : 1\leq i\leq p\}$, \\
 $\Delta^{\fr{n}}(1, 1) $ &=& $\{e^1_1\pm \pi_j, \ e^1_1 : 1\leq j\leq \ell-p-1\}$, \\
 $\Delta^{\fr{n}}(0, 2) $ &=& $\{e^2_i+e^{2}_{j} : 1\leq i < j\leq p\}$, \\
 $\Delta^{\fr{n}}(1, 2) $ &=& $\{e^1_1+e^{2}_{i} : 1\leq i\leq p\}$.
 \end{tabular}\right\}
 \end{equation}
 Thus by applying Proposition \ref{DEC} (2), we conclude that the dimensions of these submodules are given as follows: 
   \begin{equation}\label{dim1}
   \left. \begin{tabular}{lll}
   $\dim_{\bb{R}}\fr{n}_{1}$ &=& $ 2\cdot |\{\al\in\Delta_{\fr{n}}^{+} : \kappa (\al)=\overline{\al}_{1}\}| = 2\cdot |\Delta^{\fr{n}}(1, 0)|=2p$, \\
   $\dim_{\bb{R}}\fr{n}_{2}$ &=&  $2\cdot |\{\al\in\Delta_{\fr{n}}^{+} : \kappa (\al)=\overline{\al}_{p+1}\}| = 2\cdot |\Delta^{\fr{n}}(0, 1)|=2p(2\ell-2p-1)$ , \\
 $\dim_{\bb{R}}\fr{n}_{3}$ &=&  $ 2\cdot |\{\al\in\Delta_{\fr{n}}^{+} : \kappa (\al)=\overline{\al}_{1}+\overline{\al}_{p+1} \}| = 2\cdot |\Delta^{\fr{n}}(1, 1)|=2(2\ell-2p-1)$, \\ 
 $\dim_{\bb{R}}\fr{n}_{4}$ &=&  $ 2\cdot |\{\al\in\Delta_{\fr{n}}^{+} : \kappa (\al)=2\overline{\al}_{p+1} \}| = 2\cdot |\Delta^{\fr{n}}(0, 2)|=p(p-1)$,  \\
 $\dim_{\bb{R}}\fr{n}_{5}$ &=&  $2\cdot |\{\al\in\Delta_{\fr{n}}^{+} : \kappa (\al)=\overline{\al}_{1}+2\overline{\al}_{p+1} \}| = 2\cdot |\Delta^{\fr{n}}(1, 2)|=2p$. 
  \end{tabular} \right\}
 \end{equation}
 Note that for $p=\ell-1$, that is $\Pi_{\fr{n}}=\{\al_1, \al_{\ell}\}$,  we have the following painted Dynkin diagram   {\small
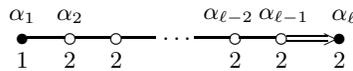
{ \[ \begin{picture}(110,-5)(0,5)
\put(0, 0){\circle*{4}}
\put(0,8.5){\makebox(0,0){$\al_1$}}
\put(0,-8){\makebox(0,0){1}}
\put(2, 0.5){\line(1,0){14}}
\put(18, 0){\circle{4}}
\put(18,8.5){\makebox(0,0){$\al_2$}}
\put(18,-8){\makebox(0,0){2}}
\put(20, 0.5){\line(1,0){13.5}}
\put(35.5, 0){\circle {4}}
\put(35.5,-8){\makebox(0,0){2}}
 \put(37.5, 0.5){\line(1,0){13}}
\put(58.5, 0){\makebox(0,0){$\ldots$}}
\put(65, 0.5){\line(1,0){13.5}}
 \put(80.6, 0){\circle{4}}
 \put(78, 8.5){\makebox(0,0){$\al_{\ell-2}$}}
\put(80.6,-8){\makebox(0,0){2}}
   \put(82, 0.5){\line(1,0){14}}
\put(98, 0){\circle{4}}
\put(99,8.5){\makebox(0,0){$\al_{\ell-1}$}}
\put(98,-8){\makebox(0,0){2}}
\put(100, 1.1){\line(1,0){15.5}}
\put(100, -0.7){\line(1,0){15.5}}
\put(113.5, -1.9){\scriptsize $>$}
\put(120, 0){\circle*{4}}
 \put(121.8,8.5){\makebox(0,0){$\al_\ell$}}
\put(120,-8){\makebox(0,0){2}}
\end{picture}
 \]}}
 
  \noindent 
 The corresponding flag manifold   $M=\SO(2\ell+1)/\U(1)\times \U(\ell-1)$ $(\ell\geq 3)$ has also five isotropy summands.  However, note that the relation $(\al_{1}, \al_{1})=(\al_{p+1}, \al_{p+1})$ ($3\leq p\leq \ell-2$) is no longer true.  It means that for the case $p=\ell-1$ the painted black simple roots have different lengths.  
 
\smallskip 
\underline {\bf Case of $B_{\ell}=\SO(2\ell+1)$ : Type B.}    We now assume that the pair $(\Pi, \Pi_{0})$ is of Type B, that is $\Pi\backslash\Pi_{0}=\Pi_{\fr{m}}=\{\al_p, \al_{p+1} : 2\leq p\leq \ell-1\}$.  This choice corrsponds to the following painted Dynkin diagram: 
   \begin{center} 
{\small
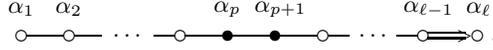
{ 
\begin{picture}(168,0)(0,-2)
\put(0, 0){\circle{4}}
\put(0,10){\makebox(0,0){$\al_1$}}
\put(2, 0){\line(1,0){14}}
\put(18, 0){\circle{4}}
\put(20, 0){\line(1,0){10}}
\put(18,10){\makebox(0,0){$\al_2$}}
\put(40, 0){\makebox(0,0){$\ldots$}}
\put(50, 0){\line(1,0){8}}
\put(60, 0){\circle{4.4}}
\put(78, 10){\makebox(0,0){$\al_p$}}
\put(62, 0){\line(1,0){14}}
\put(78, 0){\circle*{4.4}}
\put(98, 10){\makebox(0,0){$\al_{p+1}$}}
\put(80, 0){\line(1,0){14}}
\put(96, 0){\circle*{4.4}}
\put(98, 0){\line(1,0){14}}
\put(114, 0){\circle{4.4}}
\put(116, 0){\line(1,0){8}}
\put(133, 0){\makebox(0,0){$\ldots$}}
\put(142, 0){\line(1,0){8}}
\put(152, 0){\circle{4}}
\put(154, 10){\makebox(0,0){$\al_{\ell-1}$}}
\put(153.8, 1.){\line(1,0){14}}
\put(153.8, -1){\line(1,0){14}}
\put(164.5, -1.9){\scriptsize $>$}
\put(172.5, 0){\circle{4}}
\put(173, 10){\makebox(0,0){$\al_\ell$}} 
\end{picture}
} } .
\end{center} 
 which also defines the flag manifold    $M= \SO(2\ell+1)/(\U(1)\times \U(p) \times \SO(2(\ell-p-1)+1))$ with $2\leq p\leq \ell-1$  and $\ell\geq 3$.
Let $\fr{g}=\fr{k}\oplus\fr{m}$ be a reductive decomposition of $\fr{g}=\fr{so}(2\ell+1)$, with respect to $B$.    Similarly with Type A, we consider an orthonormal basis of $\bb{R}^{\ell}$ given by $\{e^1_i, e^2_{1}, \pi_{j}\}$  with $1\leq i\leq p$ and $1\leq j\leq \ell-p-1$.  Then,  a system of positive roots   is given by  
\begin{eqnarray*}
 \Delta^{+} &=& \{e^1_i\pm e^1_{j} : 1\leq i<j\leq p\}\cup\{e^{1}_{i}\pm e_{1}^{2} : 1\leq i\leq p\}\cup \{e^1_i\pm\pi_{j} : 1\leq i\leq p, \ 1\leq j\leq \ell-p-1\} \\ && \ \cup\{e_{1}^{2}\pm\pi_{j} : 1\leq j\leq\ell-p-1\}\cup \{\pi_{i}\pm\pi_{j} : 1\leq i<j\leq\ell-p-1\} \\
 && \ \cup\{e^{2}_{1},  \ e^{1}_{i} : 1\leq i\leq p\}\cup\{\pi_{j} : 1\leq j\leq \ell-p-1\}.
 \end{eqnarray*}
In this case, we denote a   base of simple roots by
   \begin{eqnarray*}
   \Pi(\fr{m}) &=&\{\al_1=e_1^1-e^1_2, \ \al_2=e^1_2-e^1_3, \ldots,  \al_{p-1}=e^1_{p-1}-e^{1}_{p}, \ \al_{p}=e^{1}_{p}-e^{2}_{1}, \ \al_{p+1}=e^{2}_{1}-\pi_{1},   \\
   && \ \ \ \  \phi_1=\pi_1-\pi_2,  \ldots, \phi_{\ell-p-2}=\pi_{\ell-p-2}-\pi_{\ell-p-1}, \ \phi_{\ell-p-1}=\pi_{\ell-p-1}\}.
   \end{eqnarray*}
The root system of the semisimple part of the isotropy subgroup $K$ is given by $\Delta_{0}^{+}=\{e^1_{i}-e^1_{j}, \ \pi_{i}\pm\pi_{j}, \ \pi_{j} : i<j\}$ and thus the positive complementary roots are  of the form
  \[
  \Delta_{\fr{m}}^{+}=\{e^{1}_i+ e^{1}_{j}, \ e^1_{i}\pm e^{2}_{1}, \ e^1_i\pm\pi_{j}, \ e^{2}_{1}\pm\pi_{j}, \ e^1_i, e^{2}_{1} : i<j\}.
  \]
 Choose  $\al=\sum_{k=1}^{\ell}c_{k}\al_{k}\in \Delta_{\fr{m}}^{+}$.  Since $\Pi_{\fr{m}}=\{\al_{p}, \al_{p+1}\}$, by applying  relation (\ref{troots})  we obtain  that   $\kappa(\al)=c_{p}\overline{\al}_{p}+c_{p+1}\overline{\al}_{p+1}\in\Delta(\fr{m})_{\fr{t}}^{+}$ where $\Delta(\fr{m})_{\fr{t}}^{+}$ is the associated system of positive $\fr{t}$-roots. Here the  coefficients $c_{p}, c_{p+1}$ are such that $0\leq c_p\leq 2$, and $0\leq c_{p+1}\leq 2$,  and they are not simultaneously equal to zero.  By expressing the complementary roots in terms of   simple roots we conclude that $\Delta(\fr{m})_{\fr{t}}^{+}=\{\overline{\al}_{p}, \overline{\al}_{p+1}, \overline{\al}_{p}+\overline{\al}_{p+1}, \overline{\al}_{p}+2 \overline{\al}_{p+1}, 2\overline{\al}_{p}+2\overline{\al}_{p+1}\}$. Indeed, it is
  \begin{eqnarray*}
\kappa(e^1_i-e^2_1) &=& \kappa(\al_i+\ldots+\al_{p})=\overline{\al}_p, \\
\kappa(e^1_i+e^2_1) &=& \kappa(\al_i+\cdots+\al_{p}+2\al_{p+1}+2\phi_{1}+\cdots +\cdots+2\phi_{\ell-p-1})=\overline{\al}_p+2\overline{\al}_{p+1}, \\
\kappa(e^1_i+e^1_j) &=& \kappa(\al_{i}+\cdots+2\al_{j}+\cdots +2\al_{p}+2\al_{p+1}+2\phi_1+\cdots +2\phi_{\ell-p-1})= 2\overline{\al}_{p}+2\overline{\al}_{p+1}, \\
\kappa(e^1_i-\pi_j) &=& \kappa(\al_i+\cdots + \al_{p}+\al_{p+1}+\phi_{1}+\cdots + \phi_j)=\overline{\al}_p+\overline{\al}_{p+1} \\
\kappa(e^1_i+\pi_j) &=& \kappa(\al_i+\cdots + \al_{p}+\al_{p+1}+\phi_{1}+\cdots +\phi_j+2\phi_{j+1}+\cdots +2\phi_{\ell-p-1})=\overline{\al}_p+\overline{\al}_{p+1}, \\
 \kappa(e^2_1-\pi_j) &=& \kappa(\al_{p+1}+\phi_{1}+\cdots+\phi_{j}) =\overline{\al}_{p+1}, \\
 \kappa(e^2_1+\pi_j) &=& \kappa(\al_{p+1}+\phi_{1}+\cdots +\phi_{j}+2\phi_{j+1}+\cdots +2\phi_{\ell-p-1}) =\overline{\al}_{p+1}, \\  
\kappa(e^2_1) &=& \kappa(\al_{p+1}+\phi_{1}+\cdots +\phi_{\ell-p-1})= \overline{\al}_{p+1}, \\
\kappa(e^1_i) &=& \kappa(\al_{i}+\cdots+\al_{p}+\al_{p+1}+\phi_1+\cdots +\phi_{\ell-p-1})=\overline{\al}_p+\overline{\al}_{p+1}.
\end{eqnarray*}
Thus, by applying Proposition \ref{DEC}  (1) we conclude that  the associated isotropy representation decomposes into a direct sum of five isotropy summands, i.e.
 \begin{equation}\label{m}
 \fr{m} = \fr{m}_1\oplus\fr{m}_{2}\oplus\fr{m}_{3}\oplus\fr{m}_{4}\oplus\fr{m}_{5}  = \fr{m}(1, 0)\oplus\fr{m}(0, 1)\oplus\fr{m}(1, 1)\oplus\fr{m}(1, 2)\oplus\fr{m}(2, 2).
 \end{equation}

 By applying Proposition \ref{DEC} (2) we obtain that
   \begin{equation}\label{dim21}
   \left. \begin{tabular}{lll}
  $\dim_{\bb{R}}\fr{m}_{1}$ &=& $2\cdot |\{\al\in\Delta_{\fr{m}}^{+} : \kappa (\al)=\overline{\al}_{p}\}| = 2\cdot |\Delta^{\fr{m}}(1, 0)|=2p$ , \\
 $\dim_{\bb{R}}\fr{m}_{2}$ &=& $ 2\cdot |\{\al\in\Delta_{\fr{m}}^{+} : \kappa (\al)=\overline{\al}_{p+1}\}| = 2\cdot |\Delta^{\fr{m}}(0, 1)|=2(2\ell-2p-1)$, \\
 $\dim_{\bb{R}}\fr{m}_{3}$ &=&  $ 2\cdot |\{\al\in\Delta_{\fr{m}}^{+} : \kappa (\al)=\overline{\al}_{p}+\overline{\al}_{p+1} \}| = 2\cdot |\Delta^{\fr{m}}(1, 1)|=2p(2\ell-2p-1)$, \\ 
 $\dim_{\bb{R}}\fr{m}_{4}$ &=&  $ 2\cdot |\{\al\in\Delta_{\fr{m}}^{+} : \kappa (\al)=\overline{\al}_{p}+2\overline{\al}_{p+1} \}| = 2\cdot |\Delta^{\fr{m}}(1, 2)|=2p$,  \\
 $\dim_{\bb{R}}\fr{m}_{5}$ &=&  $ 2\cdot |\{\al\in\Delta_{\fr{m}}^{+} : \kappa (\al)=2\overline{\al}_{p}+2\overline{\al}_{p+1} \}| = 2\cdot |\Delta^{\fr{m}}(2, 2)|=p(p-1)$. 
   \end{tabular} \right\}
 \end{equation} 
\noindent  Note that for  $p = 1$ the space  $M= \SO(2\ell+1)/(\U(1)\times \U(1) \times \SO(2(\ell- 2)+1))$  has  four isotropy summands (\cite{Chry2}).  On the other hand,  the  case   $p=\ell-1$ corresponds to the painted Dynkin diagram $\Gamma(\Pi_{\fr{m}})$ with $\Pi_{\fr{m}}=\{\al_{\ell-1}, \al_{\ell}\}$ $(\ell\geq 3)$, that is
{\small
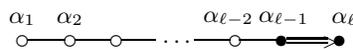
{ \[ \begin{picture}(110,-5)(0,5)
\put(0, 0){\circle{4}}
\put(0,8.5){\makebox(0,0){$\al_1$}}
\put(2, 0.5){\line(1,0){14}}
\put(18, 0){\circle{4}}
\put(18,8.5){\makebox(0,0){$\al_2$}}
\put(20, 0.5){\line(1,0){13.5}}
\put(35.5, 0){\circle {4}}
 \put(37.5, 0.5){\line(1,0){13}}
\put(58.5, 0){\makebox(0,0){$\ldots$}}
\put(65, 0.5){\line(1,0){13.5}}
 \put(80.6, 0){\circle{4}}
 \put(78, 8.5){\makebox(0,0){$\al_{\ell-2}$}}
   \put(82, 0.5){\line(1,0){14}}
\put(98, 0){\circle*{4}}
\put(99,8.5){\makebox(0,0){$\al_{\ell-1}$}}
\put(100, 1.1){\line(1,0){15.5}}
\put(100, -0.7){\line(1,0){15.5}}
\put(113.5, -1.9){\scriptsize $>$}
\put(120, 0){\circle*{4}}
 \put(121.8,8.5){\makebox(0,0){$\al_\ell$}}
\end{picture}
 \]}}
 
  \noindent
The    corresponding flag manifold $M = \SO(2\ell+1)/\U(\ell-1)\times  \U(1)$ ($\ell\geq 3$)  satisfies decomposition (\ref{m}) as well,  but in this case the painted black roots are such that $(\al_{\ell-1}, \al_{\ell-1})=2 (\al_{\ell}, \al_{\ell})$, i.e. they have different length.

\smallskip 
\underline {\bf Case of $D_{\ell}=\SO(2\ell)$ : Type A.}  Let $\Pi\backslash\Pi_{0}=\Pi_{\fr{n}}=\{\al_1, \al_{p+1} : 2 \leq p \leq \ell -3 \}$.  This choice corresponds to the  painted Dynkin diagram  
  \begin{center} 
  \vspace{10pt}
{\small
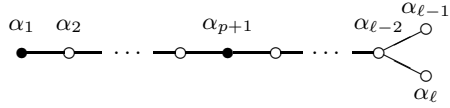
{
\begin{picture}(150,0)(0,-2)
\put(0, 0){\circle*{4}}
\put(0,10){\makebox(0,0){$\al_1$}}
\put(2, 0){\line(1,0){14}}
\put(18, 0){\circle{4}}
\put(20, 0){\line(1,0){10}}
\put(18,10){\makebox(0,0){$\al_2$}}
\put(40, 0){\makebox(0,0){$\ldots$}}
\put(50, 0){\line(1,0){8}}
\put(60, 0){\circle{4}}
\put(78, 10){\makebox(0,0){$\al_{p+1}$}}
\put(62, 0){\line(1,0){14}}
\put(78, 0){\circle*{4}}
\put(80, 0){\line(1,0){14}}
\put(96, 0){\circle{4}}
\put(98, 0){\line(1,0){8}}
\put(115, 0){\makebox(0,0){$\ldots$}}
\put(125, 0){\line(1,0){8}}
\put(135, 0){\circle{4}}
\put(135, 10){\makebox(0,0){$\al_{\ell-2}$}}
\put(137, 1){\line(2,1){14}}
\put(137, -1){\line(2,-1){14}}
\put(153, -9){\circle{4}}
\put(153, 9){\circle{4}}
\put(153, 16){\makebox(0,0){\small $\al_{\ell -1}$}}
\put(153, -16){\makebox(0,0){\small $\al_{\ell }$}}
\end{picture}
} } 
\vspace{10pt}
\end{center}

 \noindent which determines the flag manifold   $M= \SO(2\ell)/(\U(1)\times \U(p) \times \SO(2(\ell-p-1)))$  with $2 \leq p \leq \ell -3 $ and $\ell\geq 5$.

Similarly with the previous cases we obtain that the positive $\fr{t}$-roots are given by 
$\Delta(\fr{n})_{\fr{t}}^+ = \{\overline{\al}_{1}, \ \overline{\al}_{p+1}, \ \overline{\al}_{1}+\overline{\al}_{p+1}, \  2 \overline{\al}_{p+1}, 
 \ \overline{\al}_{1} + 2 \overline{\al}_{p+1} \}$ and 
  according to Proposition \ref{DEC}  (1) we obtain the decomposition (\ref{n}) 
   where the submodules $\fr{n}_{i}$ are determined  by (\ref{ni}).
   
By applying Proposition \ref{DEC} (2) 
we obtain  the dimensions of  $\fr{n}_{i}$:
   \begin{equation}\label{5dimDl}
   \left. \begin{tabular}{ll}
  $\dim_{\bb{R}}\fr{n}_{1}$ &  $=2\cdot |\{\al\in\Delta_{\fr{m}}^{+} : \kappa (\al)=\overline{\al}_{1}\}| = 2 p$, \\
 $\dim_{\bb{R}}\fr{n}_{2}$ & $= 2\cdot |\{\al\in\Delta_{\fr{m}}^{+} : \kappa (\al)=\overline{\al}_{p+1}\}| = 4 p (\ell  - p-1)$, \\
 $\dim_{\bb{R}}\fr{n}_{3}$ &  $= 2\cdot |\{\al\in\Delta_{\fr{m}}^{+} : \kappa (\al)=\overline{\al}_{1}+\overline{\al}_{p+1} \}| = 4  (\ell - p-1)$, \\ 
 $\dim_{\bb{R}}\fr{n}_{4}$ &  $= 2\cdot |\{\al\in\Delta_{\fr{m}}^{+} : \kappa (\al)=2\overline{\al}_{p+1} \}| = p (p-1)$,  \\
 $\dim_{\bb{R}}\fr{n}_{5}$ &  $= 2\cdot |\{\al\in\Delta_{\fr{m}}^{+} : \kappa (\al)=\overline{\al}_{1}+2\overline{\al}_{p+1} \}| = 2 p $. 
   \end{tabular} \right\}
 \end{equation}
 
 \smallskip 
 \underline{\bf Case of $D_{\ell}=\SO(2\ell)$ : Type B.}   We now examine the pair $(\Pi, \Pi_{0})$ of Type B corresponding to $D_{\ell}$, that is  $\Pi\backslash\Pi_{0}=\Pi_{\fr{m}}=\{\al_p, \al_{p+1} : 2\leq p\leq \ell-3\}$.  This choice corresponds to the painted Dynkin diagram  
   \begin{center} 
   \vspace{10pt}
{\small
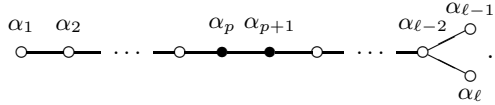
{ 
\begin{picture}(170,0)(0,-2)
\put(0, 0){\circle{4}}
\put(0,10){\makebox(0,0){$\al_1$}}
\put(2, 0){\line(1,0){14}}
\put(18, 0){\circle{4}}
\put(20, 0){\line(1,0){10}}
\put(18,10){\makebox(0,0){$\al_2$}}
\put(40, 0){\makebox(0,0){$\ldots$}}
\put(50, 0){\line(1,0){8}}
\put(60, 0){\circle{4}}
\put(76, 10){\makebox(0,0){$\al_p$}}
\put(62, 0){\line(1,0){14}}
\put(76, 0){\circle*{4}}
\put(78, 0){\line(1,0){14}}
\put(94, 0){\circle*{4}}
\put(94, 10){\makebox(0,0){$\al_{p+1}$}}
\put(96, 0){\line(1,0){14}}
\put(112, 0){\circle{4}}
\put(114, 0){\line(1,0){8}}
\put(132, 0){\makebox(0,0){$\ldots$}}
\put(142, 0){\line(1,0){8}}
\put(152, 0){\circle{4}}
\put(152, 10){\makebox(0,0){$\al_{\ell-2}$}}
\put(154, 1){\line(2,1){14}}
\put(154, -1){\line(2,-1){14}}
\put(170, -9){\circle{4}}
\put(170, 9){\circle{4}}
\put(170, 16){\makebox(0,0){\small $\al_{\ell -1}$}}
\put(170, -16){\makebox(0,0){\small $\al_{\ell }$}}\end{picture}
} }.
\vspace{10pt}
\end{center} 
which  also determines the flag manifold  $M= \SO(2\ell)/(\U(1)\times \U(p) \times \SO(2(\ell-p-1)))$, with $2\leq p\leq \ell-3$ and $\ell\geq 5$. 
 It follows that $\Delta(\fr{m})_{\fr{t}}^{+}=\{\overline{\al}_{p}, \overline{\al}_{p+1}, \overline{\al}_{p}+\overline{\al}_{p+1}, \overline{\al}_{p}+2 \overline{\al}_{p+1}, 2\overline{\al}_{p}+2\overline{\al}_{p+1}\}$, 
  and by using Proposition \ref{DEC} we conclude that the dimensions of these submodules are given as follows: 
    \begin{equation}\label{dim5}
    \left. \begin{tabular}{ll}
   $\dim_{\bb{R}}\fr{m}_{1}$ &  $=2\cdot |\{\al\in\Delta_{\fr{m}}^{+} : \kappa (\al)=\overline{\al}_{p}\}| = 2 p$, \\
  $\dim_{\bb{R}}\fr{m}_{2}$ & $= 2\cdot |\{\al\in\Delta_{\fr{m}}^{+} : \kappa (\al)=\overline{\al}_{p+1}\}| = 4  (\ell  - p-1)$, \\
  $\dim_{\bb{R}}\fr{m}_{3}$ &  $= 2\cdot |\{\al\in\Delta_{\fr{m}}^{+} : \kappa (\al)=\overline{\al}_{p}+\overline{\al}_{p+1} \}| = 4 p (\ell - p-1)$, \\ 
 $\dim_{\bb{R}}\fr{m}_{4}$ &  $= 2\cdot |\{\al\in\Delta_{\fr{m}}^{+} : \kappa (\al)=\overline{\al}_{p}+2\overline{\al}_{p+1} \}| = 2 p$,  \\
  $\dim_{\bb{R}}\fr{m}_{5}$ &  $= 2\cdot |\{\al\in\Delta_{\fr{m}}^{+} : \kappa (\al)= 2\overline{\al}_{p}+2\overline{\al}_{p+1} \}| =  p (p-1) $. 
    \end{tabular} \right\}
 \end{equation}

  \smallskip 
  \underline{\bf Case of $\E_6$ : Type A.} The highest root $\tilde{\al}$ of $\E_6$ is given by $\tilde{\al}=\al_1+2\al_2+3\al_3+2\al_4+\al_5+2\al_6$.  Thus, for $\E_6$ we find  two pairs $(\Pi, \Pi_{0})$ of Type A, which determine flag manifolds with five isotropy summands, namely the choices
  $\Pi\backslash\Pi_{0}=\{\al_1, \al_4\}$ and $\Pi\backslash\Pi_{0}=\{\al_2, \al_5\}$. They  correspond to the painted  Dynkin diagrams 
  \begin{center}
 \small
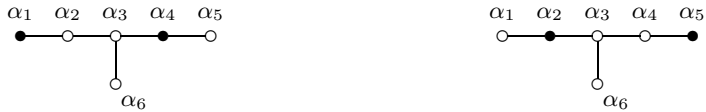
{
   \begin{picture}(100,45)(0,-30)
\put(0, 0){\circle*{4}}
\put(0,8.5){\makebox(0,0){$\alpha_1$}}
\put(2, 0){\line(1,0){14}}
\put(18, 0){\circle{4}}
\put(18,8.5){\makebox(0,0){$\alpha_2$}}
\put(20, 0){\line(1,0){14}}
\put(36, 0){\circle{4}}
\put(36,8.5){\makebox(0,0){$\alpha_3$}}
\put(38, 0){\line(1,0){14}}
\put(54, 0){\circle*{4}}
\put(54,8.5){\makebox(0,0){$\alpha_4$}}
\put(56, 0){\line(1,0){14}}
\put(72, 0){\circle{4}}
\put(72,8.5){\makebox(0,0){$\alpha_5$}}
\put(36, -2){\line(0,-1){14}}
\put(36, -18){\circle{4}}
\put(43,-25){\makebox(0,0){$\alpha_6$}}
\end{picture} 
}
\hspace{70pt}
\small{ 
  \begin{picture}(100,45)(0,-30) 
\put(0, 0){\circle{4}}
\put(0,8.5){\makebox(0,0){$\alpha_1$}}
\put(2, 0){\line(1,0){14}}
\put(18, 0){\circle*{4}}
\put(18,8.5){\makebox(0,0){$\alpha_2$}}
\put(20, 0){\line(1,0){14}}
\put(36, 0){\circle{4}}
\put(36,8.5){\makebox(0,0){$\alpha_3$}}
\put(38, 0){\line(1,0){14}}
\put(54, 0){\circle{4}}
\put(54,8.5){\makebox(0,0){$\alpha_4$}}
\put(56, 0){\line(1,0){14}}
\put(72, 0){\circle*{4}}
\put(72,8.5){\makebox(0,0){$\alpha_5$}}
\put(36, -2){\line(0,-1){14}}
\put(36, -18){\circle{4}}
\put(43,-25){\makebox(0,0){$\alpha_6$}}
\end{picture}  
}
\end{center}
  which both define the flag manifold $\E_6/\SU(4)\times \SU(2)\times \U(1)^{2}$.  However, there is an outer automorphism of $\E_6$\footnote{The group of outer automorphisms of a simple Lie algebra is precisely the group of graph automorphisms of the associated Dynkin diagram.  It is known that for the exceptional simple Lie algebras over $\bb{C}$, outer automorphisms exist only for $\E_6$.}   which makes these painted Dynkin diagrams equivalent (see \cite{BFR}).  Thus we will not  distinguish these two   pairs $(\Pi, \Pi_{0})$ and we will work with the first one.   
    Let $\fr{n}$ be the $B$-orthogonal complement  of the isotropy subalgebra $\fr{k}$ in $\fr{e}_6$. For the root system of $\E_6$ we use the notation of \cite{AA}, where all positive roots are given as linear combinations of the simple roots $\Pi=\{\al_1, \ldots, \al_6\}$.  The root system of the semisimple part of the isotropy subgroup $K$ is given by $\Delta_{0}^{+}=\{\al_2, \al_3, \al_5, \al_6, \al_2+\al_3, \al_3+\al_6, \al_2+\al_3+\al_6\}$,  thus 
 \begin{equation}\label{E61}
\Delta_{\fr{n}}^{+}= {\small{\left\{ 
\begin{tabular}{lll} 
$\al_1+2\al_2+3\al_3+2\al_4+\al_5+2\al_6$ & $\al_1+\al_2+\al_3+\al_4+\al_6$   & $\al_2+2\al_3+2\al_4+\al_5+\al_6$  \\
$\al_1+2\al_2+3\al_3+2\al_4+\al_5+\al_6$  & $\al_1+\al_2+\al_3+\al_6$         & $\al_2+2\al_3+\al_4+\al_5+\al_6$\\
$\al_1+2\al_2+2\al_3+2\al_4+\al_5+\al_6$  & $\al_1+\al_2+\al_3+\al_4+\al_5$   & $\al_2+2\al_3+\al_4+\al_6$\\
$\al_1+2\al_2+2\al_3+\al_4+\al_5+\al_6$   & $\al_1+\al_2+\al_3+\al_4$         & $\al_2+\al_3+\al_4+\al_5+\al_6$\\
$\al_1+2\al_2+2\al_3+\al_4+\al_6$         & $\al_3+\al_4+\al_5$               & $\al_2+\al_3+\al_4+\al_6$\\
$\al_1+\al_2+2\al_3+2\al_4+\al_5+\al_6$   & $\al_3+\al_4+\al_6$               & $\al_2+\al_3+\al_4+\al_5$\\
$\al_1+\al_2+2\al_3+\al_4+\al_5+\al_6$    & $\al_3+\al_4$                     & $\al_2+\al_3+\al_4$\\
$\al_1+\al_2+2\al_3+\al_4+\al_6$          & $\al_4+\al_5$                     & $\al_1+\al_2+\al_3$  \\
$\al_1+\al_2+\al_3+\al_4+\al_5+\al_6$     & $\al_4$                           & $\al_1+\al_2$         \\
$\al_1$                                   & $\al_3+\al_4+\al_5+\al_6$         & 
               \end{tabular} \right. }}
   \end{equation}
Let $\al=\sum_{k=1}^{6}c_{k}\al_{k}\in \Delta_{\fr{n}}^{+}$.  Since $\Pi_{\fr{n}}=\{\al_{1}, \al_{4}\}$, by applying  relation (\ref{troots})  we obtain  that   $\kappa(\al)=c_{1}\overline{\al}_{1}+c_{4}\overline{\al}_{4}$, where the  numbers $c_{1}, c_{4}$ are such that $0\leq c_1, c_4\leq 2$.  So, by using (\ref{E61}), we easily conclude that the positive  $\fr{t}$-roots are given by 
$\Delta(\fr{n})_{\fr{t}}^+ = \{\overline{\al}_{1}, \ \overline{\al}_{4}, \ \overline{\al}_{1}+\overline{\al}_{4}, 
 \ 2\overline{\al}_{4},  \ \overline{\al}_{1} + 2 \overline{\al}_{4} \},$ and 
 thus according to Proposition \ref{DEC}  (1), we obtain the decomposition (\ref{n}) where the sumbodules $\fr{n}_{i}$ are defined by (\ref{ni}). The sets  $\Delta^{\fr{n}}(j_{1}, j_{2})$ are given explicitly as follows:
 {\small{\begin{eqnarray*}
 \Delta^{\fr{n}}(1, 0) &=& \{\al_1, \al_1+\al_2, \al_1+\al_2+\al_3, \al_1+\al_2+\al_3+\al_6\}, \\
 \Delta^{\fr{n}}(0, 1) &=& \{\al_4, \al_3+\al_4, \al_4+\al_5, \al_2+\al_3+\al_4, \al_2+\al_3+\al_4+\al_5, \al_3+\al_4+\al_5,  \al_3+\al_4+\al_5+\al_6, \al_3+\al_4+\al_6, \\
 && \al_2+\al_3+\al_4+\al_6, \al_2+2\al_3+\al_4+\al_6, \al_2+\al_3+\al_4+\al_5+\al_6, \al_2+2\al_3+\al_4+\al_5+\al_6\}, \\   
 \Delta^{\fr{n}}(1, 1) &=&  \{\al_1+\al_2+\al_3+\al_4, \al_1+\al_2+\al_3+\al_4+\al_5, \al_1+\al_2+\al_3+\al_4+\al_6, 
 \al_1+\al_2+2\al_3+\al_4+\al_6,\\ 
 &&  \al_1+\al_2+\al_3+\al_4+\al_5+\al_6, \al_1+\al_2+2\al_3+\al_4+\al_5+\al_6, \al_1+2\al_2+2\al_3+\al_4+\al_5+\al_6, \\ && \al_1+2\al_2+2\al_3+\al_4+\al_6\},\\
  \Delta^{\fr{n}}(0, 2) &=& \{\al_2+2\al_3+2\al_4+\al_5+\al_6\},\\
   \Delta^{\fr{n}}(1, 2) &=& \{\al_1+\al_2+\al_3+2\al_4+\al_6, \al_1+2\al_2+3\al_3+2\al_4+\al_5+\al_6, \al_1+2\al_2+2\al_3+2\al_4+\al_5+\al_6, \\
   && \al_1+2\al_2+3\al_3+2\al_4+\al_5+2\al_6\}.
 \end{eqnarray*}}}
   By applying Proposition \ref{DEC} (2) we easily conclude that  
   \begin{equation}\label{E6dimA} 
   \left. \begin{tabular}{lll}
   $\dim_{\bb{R}}\fr{n}_{1}$ &=& $ 2\cdot |\{\al\in\Delta_{\fr{n}}^{+} : \kappa (\al)=\overline{\al}_{1}\}| = 2\cdot |\Delta^{\fr{n}}(1, 0)|= 2\cdot 4=8$, \\
   $\dim_{\bb{R}}\fr{n}_{2}$ &=&  $2\cdot |\{\al\in\Delta_{\fr{n}}^{+} : \kappa (\al)=\overline{\al}_{4}\}| = 2\cdot |\Delta^{\fr{n}}(0, 1)|= 2\cdot 12=24$, \\
 $\dim_{\bb{R}}\fr{n}_{3}$ &=&  $ 2\cdot |\{\al\in\Delta_{\fr{n}}^{+} : \kappa (\al)=\overline{\al}_{1}+\overline{\al}_{4} \}| = 2\cdot |\Delta^{\fr{n}}(1, 1)|= 2\cdot 8=16$, \\ 
 $\dim_{\bb{R}}\fr{n}_{4}$ &=&  $ 2\cdot |\{\al\in\Delta_{\fr{n}}^{+} : \kappa (\al)=2\overline{\al}_{4} \}| = 2\cdot |\Delta^{\fr{n}}(0, 2)|=2\cdot 1=2$,  \\
 $\dim_{\bb{R}}\fr{n}_{5}$ &=&  $2\cdot |\{\al\in\Delta_{\fr{n}}^{+} : \kappa (\al)=\overline{\al}_{1}+2\overline{\al}_{4} \}| = 2\cdot |\Delta^{\fr{n}}(1, 2)|= 2\cdot 4=8$. 
   \end{tabular} \right\}
 \end{equation}
 
 \smallskip 
      \underline{\bf Case of $\E_6$ : Type B.} The   flag manifold $\E_6/\SU(4)\times \SU(2)\times \U(1)^{2}$ is also defined by two pairs $(\Pi, \Pi_{0})$ of Type B, given by
  $\Pi\backslash\Pi_{0}=\{\al_4, \al_6\}$ and $\Pi\backslash\Pi_{0}=\{\al_2, \al_6\}$. They    correspond to the painted  Dynkin diagrams 
   \begin{center}
 \small
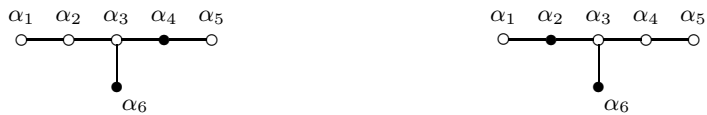
{
   \begin{picture}(100,45)(0,-30)
\put(0, 0){\circle{4}}
\put(0,8.5){\makebox(0,0){$\alpha_1$}}
\put(2, 0){\line(1,0){14}}
\put(18, 0){\circle{4}}
\put(18,8.5){\makebox(0,0){$\alpha_2$}}
\put(20, 0){\line(1,0){14}}
\put(36, 0){\circle{4}}
\put(36,8.5){\makebox(0,0){$\alpha_3$}}
\put(38, 0){\line(1,0){14}}
\put(54, 0){\circle*{4}}
\put(54,8.5){\makebox(0,0){$\alpha_4$}}
\put(56, 0){\line(1,0){14}}
\put(72, 0){\circle{4}}
\put(72,8.5){\makebox(0,0){$\alpha_5$}}
\put(36, -2){\line(0,-1){14}}
\put(36, -18){\circle*{4}}
\put(43,-25){\makebox(0,0){$\alpha_6$}}
\end{picture} 
}
\hspace{70pt}
\small{ 
  \begin{picture}(100,45)(0,-30) 
\put(0, 0){\circle{4}}
\put(0,8.5){\makebox(0,0){$\alpha_1$}}
\put(2, 0){\line(1,0){14}}
\put(18, 0){\circle*{4}}
\put(18,8.5){\makebox(0,0){$\alpha_2$}}
\put(20, 0){\line(1,0){14}}
\put(36, 0){\circle{4}}
\put(36,8.5){\makebox(0,0){$\alpha_3$}}
\put(38, 0){\line(1,0){14}}
\put(54, 0){\circle{4}}
\put(54,8.5){\makebox(0,0){$\alpha_4$}}
\put(56, 0){\line(1,0){14}}
\put(72, 0){\circle{4}}
\put(72,8.5){\makebox(0,0){$\alpha_5$}}
\put(36, -2){\line(0,-1){14}}
\put(36, -18){\circle*{4}}
\put(43,-25){\makebox(0,0){$\alpha_6$}}
\end{picture}  
}
\end{center}
Note that there is also an outer automorphism of $\E_6$ which makes these painted Dynkin diagrams equivalent (\cite{BFR}), and thus we can
work    with  the first   pair $(\Pi, \Pi_{0})$ only.  
    By similar method we obtain that
   the positive $\fr{t}$-roots are   
$\Delta(\fr{m})_{\fr{t}}^+ = \{\overline{\al}_{6}, \ \overline{\al}_{4}, \ \overline{\al}_{6}+\overline{\al}_{4}, 
 \  \overline{\al}_{6} + 2 \overline{\al}_{4},  \  2 \overline{\al}_{6} + 2 \overline{\al}_{4} \}$ and thus  
  according to Proposition \ref{DEC}  (1), we obtain the decomposition (\ref{m})
  where the 
dimensions of the  submodules $\fr{m}_{i}$ are given as follows: 
      \begin{equation}\label{E6dimB}
   \left. \begin{tabular}{lll}
   $\dim_{\bb{R}}\fr{m}_{1}$ &=& $ 2\cdot |\{\al\in\Delta_{\fr{m}}^{+} : \kappa (\al)=\overline{\al}_{6}\}| = 2\cdot |\Delta^{\fr{m}}(1, 0)|= 2\cdot 4=8$, \\
   $\dim_{\bb{R}}\fr{m}_{2}$ &=&  $2\cdot |\{\al\in\Delta_{\fr{m}}^{+} : \kappa (\al)=\overline{\al}_{4}\}| = 2\cdot |\Delta^{\fr{m}}(0, 1)|= 2\cdot 8=16$, \\
 $\dim_{\bb{R}}\fr{m}_{3}$ &=&  $ 2\cdot |\{\al\in\Delta_{\fr{m}}^{+} : \kappa (\al)=\overline{\al}_{6}+\overline{\al}_{4} \}| = 2\cdot |\Delta^{\fr{m}}(1, 1)|= 2\cdot 12=24$, \\ 
 $\dim_{\bb{R}}\fr{m}_{4}$ &=&  $ 2\cdot |\{\al\in\Delta_{\fr{m}}^{+} : \kappa (\al)= \overline{\al}_{6}+ 2\overline{\al}_{4} \}| = 2\cdot |\Delta^{\fr{m}}(1, 2)|=2\cdot 4=8$,  \\
 $\dim_{\bb{R}}\fr{m}_{5}$ &=&  $2\cdot |\{\al\in\Delta_{\fr{m}}^{+} : \kappa (\al)= 2\overline{\al}_{6}+2\overline{\al}_{4} \}| = 2\cdot |\Delta^{\fr{m}}(2, 2)|= 2\cdot 1=2$. 
   \end{tabular} \right\}
 \end{equation}

\smallskip 
  \underline{\bf Case of $\E_7$ : Type A.}  Recall that the highest root $\tilde{\al}$ of $\E_7$ is given by $\tilde{\al}=\al_1+2\al_2+3\al_3+4\al_4+3\al_5+2\al_6+2\al_7$.  Consider the pair $(\Pi, \Pi_{0})$ with  $\Pi\backslash\Pi_{0}=\{\al_{1}, \al_{7} \}$.  This choise corresponds to the painted Dynkin diagram  
     \begin{center}
 \small
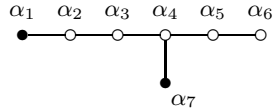
{
   \begin{picture}(100,45)(-20,-28)
   \put(-18, 0){\circle*{4}}
\put(-18,8.5){\makebox(0,0){$\alpha_1$}}
\put(-16, 0){\line(1,0){14}}
\put(0, 0){\circle{4}}
\put(0,8.5){\makebox(0,0){$\alpha_2$}}
\put(2, 0){\line(1,0){14}}
\put(18, 0){\circle{4}}
\put(18,8.5){\makebox(0,0){$\alpha_3$}}
\put(20, 0){\line(1,0){14}}
\put(36, 0){\circle{4}}
\put(36,8.5){\makebox(0,0){$\alpha_4$}}
\put(38, 0){\line(1,0){14}}
\put(54, 0){\circle{4}}
\put(54,8.5){\makebox(0,0){$\alpha_5$}}
\put(56, 0){\line(1,0){14}}
\put(72, 0){\circle{4}}
\put(72,8.5){\makebox(0,0){$\alpha_6$}}
\put(36, -2){\line(0,-1){14}}
\put(36, -18){\circle*{4}}
\put(43,-25){\makebox(0,0){$\alpha_7$}}
\end{picture} 
}
\end{center}
which determines the flag manifold   $\E_7/\SU(6)\times \U(1)^{2}$. Let $\fr{n}$ be a $B$-ortogonal complement of $\fr{e}_{7}$.  By using  the expressions of positive roots of $\E_7$ in terms of the simple roots $\Pi=\{\al_1, \al_2, \al_3, \al_4, \al_5, \al_6, \al_7\}$ (see \cite{Fr} or \cite{PhD}) and by applying (\ref{troots}), we easily conclude that the positive $\fr{t}$-roots are given by 
$\Delta(\fr{n})_{\fr{t}}^+ = \{\overline{\al}_{1}, \ \overline{\al}_{7}, \ \overline{\al}_{1}+\overline{\al}_{7}, 
 \ 2\overline{\al}_{7},  \ \overline{\al}_{1} + 2 \overline{\al}_{7} \}$.
 Thus according to Proposition \ref{DEC}  (1) we obtain the decomposition  (\ref{n}), 
 and the  dimensions of the submodules $\fr{n}_{i}$ are given as follows:
     \begin{equation}\label{E7dimA} 
   \left. \begin{tabular}{lll}
   $\dim_{\bb{R}}\fr{n}_{1}$ &=& $ 2\cdot |\{\al\in\Delta_{\fr{n}}^{+} : \kappa (\al)=\overline{\al}_{1}\}| = 2\cdot |\Delta^{\fr{n}}(1, 0)|= 2\cdot 6$, \\
   $\dim_{\bb{R}}\fr{n}_{2}$ &=&  $2\cdot |\{\al\in\Delta_{\fr{n}}^{+} : \kappa (\al)=\overline{\al}_{7}\}| = 2\cdot |\Delta^{\fr{n}}(0, 1)|= 2\cdot 20$, \\
 $\dim_{\bb{R}}\fr{n}_{3}$ &=&  $ 2\cdot |\{\al\in\Delta_{\fr{n}}^{+} : \kappa (\al)=\overline{\al}_{1}+\overline{\al}_{7} \}| = 2\cdot |\Delta^{\fr{n}}(1, 1)|= 2\cdot 15$, \\ 
 $\dim_{\bb{R}}\fr{n}_{4}$ &=&  $ 2\cdot |\{\al\in\Delta_{\fr{n}}^{+} : \kappa (\al)=2\overline{\al}_{7} \}| = 2\cdot |\Delta^{\fr{n}}(0, 2)|=2\cdot 1$,  \\
 $\dim_{\bb{R}}\fr{n}_{5}$ &=&  $2\cdot |\{\al\in\Delta_{\fr{n}}^{+} : \kappa (\al)=\overline{\al}_{1}+2\overline{\al}_{7} \}| = 2\cdot |\Delta^{\fr{n}}(1, 2)|= 2\cdot 6$. 
   \end{tabular} \right\}
 \end{equation}
  
  \smallskip 
  \underline{\bf Case of $\E_7$ : Type B.}  The flag manifold   $\E_7/\SU(6)\times \U(1)^{2}$
  is also defined by a pair $(\Pi, \Pi_{0})$ of Type B, explicitely given by  $\Pi\backslash\Pi_{0} = \{\al_{6}, \al_{7} \}$.  It corrresponds to the painted Dynkin diagram
 \begin{center}
 \small
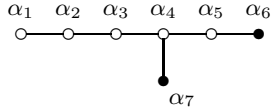
{
    \begin{picture}(100,45)(-25,-28) 
   \put(-18, 0){\circle{4}}
\put(-18,8.5){\makebox(0,0){$\alpha_1$}}
\put(-16, 0){\line(1,0){14}}
\put(0, 0){\circle{4}}
\put(0,8.5){\makebox(0,0){$\alpha_2$}}
\put(2, 0){\line(1,0){14}}
\put(18, 0){\circle{4}}
\put(18,8.5){\makebox(0,0){$\alpha_3$}}
\put(20, 0){\line(1,0){14}}
\put(36, 0){\circle{4}}
\put(36,8.5){\makebox(0,0){$\alpha_4$}}
\put(38, 0){\line(1,0){14}}
\put(54, 0){\circle{4}}
\put(54,8.5){\makebox(0,0){$\alpha_5$}}
\put(56, 0){\line(1,0){14}}
\put(72, 0){\circle*{4}}
\put(72,8.5){\makebox(0,0){$\alpha_6$}}
\put(36, -2){\line(0,-1){14}}
\put(36, -18){\circle*{4}}
\put(43,-25){\makebox(0,0){$\alpha_7$}}
\end{picture} 
}
\end{center}
   In this case the positive $\fr{t}$-roots are given by 
$\Delta(\fr{m})_{\fr{t}}^+ = \{\overline{\al}_{6}, \ \overline{\al}_{7}, \ \overline{\al}_{6}+\overline{\al}_{7}, 
  \overline{\al}_{6} + 2 \overline{\al}_{7},  \  2 \overline{\al}_{6} + 2 \overline{\al}_{7} \}$ and 
  according to Proposition \ref{DEC}  (1), the $B$-orthogonal complement $\fr{m}$ decomposes as (\ref{m}), where the submodules $\fr{m}_{i}$ 
   have dimensions
   \begin{equation}\label{E7dimB}
   \left. \begin{tabular}{lll}
   $\dim_{\bb{R}}\fr{m}_{1}$ &=& $ 2\cdot |\{\al\in\Delta_{\fr{m}}^{+} : \kappa (\al)=\overline{\al}_{6}\}| = 2\cdot |\Delta^{\fr{m}}(1, 0)|= 2\cdot 6$, \\
   $\dim_{\bb{R}}\fr{m}_{2}$ &=&  $2\cdot |\{\al\in\Delta_{\fr{m}}^{+} : \kappa (\al)=\overline{\al}_{7}\}| = 2\cdot |\Delta^{\fr{m}}(0, 1)|= 2\cdot 15$, \\
 $\dim_{\bb{R}}\fr{m}_{3}$ &=&  $ 2\cdot |\{\al\in\Delta_{\fr{m}}^{+} : \kappa (\al)=\overline{\al}_{6}+\overline{\al}_{7} \}| = 2\cdot |\Delta^{\fr{m}}(1, 1)|= 2\cdot 20$, \\ 
 $\dim_{\bb{R}}\fr{m}_{4}$ &=&  $ 2\cdot |\{\al\in\Delta_{\fr{m}}^{+} : \kappa (\al)= \overline{\al}_{6}+ 2 \overline{\al}_{7} \}| = 2\cdot |\Delta^{\fr{m}}(1, 2)|=2\cdot 6$,  \\
 $\dim_{\bb{R}}\fr{m}_{5}$ &=&  $2\cdot |\{\al\in\Delta_{\fr{m}}^{+} : \kappa (\al)= 2\overline{\al}_{6}+2\overline{\al}_{7} \}| = 2\cdot |\Delta^{\fr{m}}(2, 2)|= 2\cdot 1$. 
   \end{tabular} \right\}
 \end{equation}
 
 {\bf Step 2.}   By using \cite[Prop. 5]{Chry2} and Step 1 of the proof we have completed the study of  all possible pairs   $(\Pi, \Pi_{0})$ of Type A.  On the other hand, and due to the form of the highest root of the classical simple Lie groups, we have also studied all possible {\it classical} flag manifolds of Types A and B (the symplectic Lie group $\Sp(\ell)$ was treated in Example \ref{ISO2}).  
 Thus we now focus  on pairs $(\Pi, \Pi_{0})$ of  Type B corresponding to exceptional Lie groups, which  define flag manifolds with
 more than five isotropy summands. Hence these are not listed in Table 1.  
 
 \smallskip 
 \underline{\bf Case of $\E_6$.}   For this Lie group there exists one more  pair $(\Pi, \Pi_{0})$ of Type B given by $\Pi\backslash\Pi_{0}=\{\al_2, \al_4\}$, which determines the flag manifold $\E_6/\SU(3)\times \SU(2)\times \SU(2)\times \U(1)^{2}$.  The isotropy representation of this space decompsoses into six isotropy summands, since we find six positive  $\fr{t}$-roots given by $\{\overline{\al}_{2}, \ \overline{\al}_{4}, \ \overline{\al}_{2}+\overline{\al}_{4}, 
 \  \overline{\al}_{2} + 2 \overline{\al}_{4},  \  2\overline{\al}_{2} + \overline{\al}_{4}, \ 2 \overline{\al}_{2} + 2 \overline{\al}_{4} \}$.
 
  \underline{\bf Case of $\E_7$.}  In this case there are two more pairs $(\Pi, \Pi_{0})$ of Type B, namely the pairs
 $\Pi\backslash\Pi_{0}=\{\al_2, \al_7\}$ and $\Pi\backslash\Pi_{0}=\{\al_2, \al_6\}$ which determine the flag manifolds
 $\E_7/\SU(5)\times \SU(2)\times \U(1)^{2}$ and $\E_7/\SO(8)\times \SU(2)\times \U(1)^{2}$ respectively.  Both of these flag manifolds have six isotropy summands, since the associated positive $\fr{t}$-roots are given by $\{\overline{\al}_{2}, \ \overline{\al}_{7}, \ \overline{\al}_{2}+\overline{\al}_{7}, 
 \  2\overline{\al}_{2} + \overline{\al}_{7},  \   \overline{\al}_{2} + 2\overline{\al}_{7}, \ 2 \overline{\al}_{2} + 2 \overline{\al}_{7} \}$  and $\{\overline{\al}_{2}, \ \overline{\al}_{6}, \ \overline{\al}_{2}+\overline{\al}_{6}, \
  2\overline{\al}_{2},  \  2\overline{\al}_{2} + \overline{\al}_{6}, \ 2 \overline{\al}_{2} + 2 \overline{\al}_{6} \}$ respectively.
  
   \underline{\bf Case of $\E_8$.}  The highest root $\tilde{\al}$ of $\E_{8}$ is given by $\tilde{\al}=2\al_1+3\al_2+4\al_3+5\al_4+6\al_5+4\al_6+2\al_7+3\al_8$.  Thus for $\E_8$ there exists only a pair  $(\Pi, \Pi_{0})$ of Type B, given by $\Pi\backslash\Pi_{0}=\{\al_1, \al_7\}$.  It determines the flag manifold $\E_8/\SO(12)\times \U(1)^{2}$ which has six isotropy summands.  Indeed, by expressing  the positive roots in terms of simple roots  (see \cite{PhD} or \cite{Fr}), and  by applying (\ref{troots}) we obtain six positive $\fr{t}$-roots, namely $\{\overline{\al}_{1}, \ \overline{\al}_{7}, \ \overline{\al}_{1}+\overline{\al}_{7}, 
 \  2\overline{\al}_{1} + \overline{\al}_{7},  \ 2\overline{\al}_{7}, \ 2 \overline{\al}_{1} + 2 \overline{\al}_{7} \}$.
 
  \underline{\bf Case of $\F_4$.} The highest root $\tilde{\al}$ of $\F_{4}$ is given by $\tilde{\al}=2\al_1+3\al_2+4\al_3+2\al_4$.  Thus, the unique pair $(\Pi, \Pi_{0})$ of Type B is given by $\Pi\backslash\Pi_{0}=\{\al_1, \al_4\}$.  It determines the flag manifold $\F_4/\SO(5)\times \U(1)^{2}$ with six isotropy summands.  Indeed, by using the expressions of positive roots in terms of simple roots (see \cite{AA}) and by applying (\ref{troots}) we obtain six positive $\fr{t}$-roots, namely $\{\overline{\al}_{1}, \ \overline{\al}_{4}, \ \overline{\al}_{1}+\overline{\al}_{4}, 
 \  2\overline{\al}_{1} + \overline{\al}_{4},  \ 2\overline{\al}_{1}, \ 2 \overline{\al}_{1} + 2 \overline{\al}_{4} \}$.  
 \end{proof}
 
The following corollary in now immediate. 
\begin{corol}\label{5summands}
The only generalized flag manifolds $M$ with $b_2(M)=2$ whose isotropy representation decomposes into five isotropy summands are those
listed in Table 1.
\end{corol} 

\begin{corol}\label{brackets}
Let $M=G/K$ be a flag manifold of Type A with isotropy representation 
$\fr{n}=\fr{n}_1\oplus\fr{n}_2\oplus\fr{n}_3\oplus\fr{n}_4\oplus\fr{n}_5$.  Then the $\Ad (K)$-modules 
$\fr{n}_i$ satisfy the relations
$[{\frak n}_1, {\frak n}_2] = {\frak n}_3$, $[{\frak n}_1, {\frak n}_4] = {\frak n}_5$, $[{\frak n}_2, {\frak n}_2] \subset {\frak n}_4\oplus{\frak k}$, $[{\frak n}_2, {\frak n}_3] = \fr{n}_1\oplus{\frak n}_5$, $[{\frak n}_1, {\frak n}_3] = {\frak n}_2$,
$[{\frak n}_2, {\frak n}_4] = {\frak n}_2$, $[{\frak n}_1, {\frak n}_5] = {\frak n}_4$, $[{\frak n}_4, {\frak n}_5] = {\frak n}_1$, 
 $[{\frak n}_2, {\frak n}_5] = {\frak n}_3$, and $[{\frak n}_3, {\frak n}_5] = {\frak n}_2$. 
\end{corol} 
 
 \begin{proof}
It is immediate by considering the $T$-root systems of the flag manifolds of Type A in the proof of Proposition
\ref{classif}. 
 \end{proof}
 
  \subsection{  The isometry between flag manifolds of Type A and Type B}
 By using the analysis given in the previous paragraph,  we will prove  that for  any simple Lie group $G$ appearing in Table 1, there is an isometry  which makes the corresponding  flag manifolds $G/K$ with five isotropy summands of Type A and B, isometric as real manifolds.  We will show that this isometry is obtained in a canonical way from the Weyl group $\mathcal{W}$ corresponding to (the root system of) $G$. 
 
  \begin{theorem}\label{isometry}
For any Lie group $G$ appearing in Table 1, the correpsonding pairs $(\Pi, \Pi_{0})$ of Type A and B, define isometric  flag manifolds (as real manifolds).  
 \end{theorem} 
 
 \begin{proof} 
 At first we consider a Dynkin diagram of type $A_p$ :  \ \ 
\small{
\begin{picture}(84,0)(0,-2)
\put(0, 0){\circle{4}}
\put(0,10){\makebox(0,0){$\al_1$}}
\put(2, 0){\line(1,0){14}}
\put(18, 0){\circle{4}}
\put(20, 0){\line(1,0){10}}
\put(18,10){\makebox(0,0){$\al_2$}}
\put(40, 0){\makebox(0,0){$\ldots$}}
\put(50, 0){\line(1,0){10}}
\put(62, 0){\circle{4}}
\put(62, 10){\makebox(0,0){$\al_{p-1}$}} 
\put(64, 0){\line(1,0){14}}
\put(80, 0){\circle{4}}
\put(80, 10){\makebox(0,0){$\al_{p}$}} 
\end{picture}.
}
Then there exists an element $w_0$ of the Weyl group 
 $\mathcal{W}$ of $\SU(p+1)$ with $ w_0(\al_{i}) = - \al_{p+1-i}$ for $i = 1, \dots, p$. In fact, $w_0$ is given by 
 \begin{center} $ w_0 = s^{}_{\al^{}_{k+1}} \cdot  s^{}_{\al^{}_k+ \al^{}_{k+1}+  \al^{}_{k+2}} \cdots \  s^{}_{\al^{}_2+\cdots +  \al^{}_{k+1}+\cdots+  \al^{}_{p-1}}\cdot   s^{}_{\al^{}_1+\cdots +  \al^{}_{k+1}+\cdots+  \al^{}_{p}}$ for $p = 2 k +1$ 
 \end{center}
 and
\begin{center} $ w_0 = s^{}_{\al^{}_{k}+\al^{}_{k+1}} \cdot s^{}_{\al^{}_k+ \al^{}_{k+1}+  \al^{}_{k+2}} \cdots \ s^{}_{\al^{}_2+\cdots +  \al^{}_{k+1}+\cdots+  \al^{}_{p-1}}\cdot   s^{}_{\al^{}_1+\cdots +  \al^{}_{k+1}+\cdots+  \al^{}_{p}}$ for $p = 2 k $,  
 \end{center}
 where  $s^{}_{\be}$ is the reflection defined by a root $\be$. 
 
 For a moment we write $\Pi_{0}(A)$ and $\Pi_{0}(B)$ for $\Pi_{0}$ of Type A and $\Pi_{0}$ of Type B  and also write $\Delta_{0}(A)$ and  $\Delta_{0}(B)$ for $\Delta_{0}$ of Type A and $\Delta_{0}$ of Type B respectively. 
 Now for  $B_{\ell}$ and $D_\ell$ we see that  the  pair $(\Pi, \Pi_{0})$ is given by  $\Pi_{\fr{n}}=\Pi\backslash\Pi_{0}=\{\al_1, \al_{p+1}\}$ (Type A) and  $\Pi\backslash\Pi_{0}= \Pi_{\fr{m}}=\{\al_{p}, \al_{p+1}\}$ (Type B) and thus  
the painted Dynkin diagram of Type A and Type B contain a Dynkin subdiagram of type $A_p$, where $ p = 2, \dots, \ell-1$ for   $B_{\ell}$ and  $ p = 2, \dots, \ell-3$ for   $D_{\ell}$.  We regard  $w_0$ as an element of the Weyl group 
of type $B_{\ell}$ and $D_{\ell}$. 
 Then we have   $w_0(\al_{p+1}) = \al_{p+1} + ( \al_1 + \cdots + \al_{p})$ and $w_0(\al_{p+k}) = \al_{p+k}$ for $ k =2, \dots, \ell -p$, 
and  it follows  that $w_{0}\big(\Delta_{0}(A)\big) = \Delta_{0}(B)$ and 
   \[
 \begin{tabular}{lll}
 $w_{0}\big(\Delta^{\fr{n}}(1, 0)\big) = -\Delta^{\fr{m}}(1, 0),$ & 
 $w_{0}\big(\Delta^{\fr{n}}(0, 1)\big) = \Delta^{\fr{m}}(1, 1),$ &
 $w_{0}\big(\Delta^{\fr{n}}(1, 1)\big) = \Delta^{\fr{m}}(0, 1),$ \\
 & &  \\
 $w_{0}\big(\Delta^{\fr{n}}(0, 2)\big) = \Delta^{\fr{m}}(2, 2),$ &
 $w_{0}\big(\Delta^{\fr{n}}(1, 2)\big) = \Delta^{\fr{m}}(1, 2).$ & 
 \end{tabular}  \]
 
 \medskip
 
 For $\E_6$  the  pair $(\Pi, \Pi_{0})$ is given by  $\Pi_{\fr{n}}=\Pi\backslash\Pi_{0}=\{\al_1, \al_{4}\}$ (Type A) and  $\Pi\backslash\Pi_{0}= \Pi_{\fr{m}}=\{\al_{4}, \al_{6}\}$ (Type B) and thus the painted Dynkin diagrams of Type A and Type B contain a Dynkin subdiagram of  type $A_4$ :  \ \ 
 
 \medskip
\small{
\begin{picture}(56,0)(0,-2)
\put(0, 0){\circle{4}}
\put(0,7){\makebox(0,0){$\al_1$}}
\put(2, 0){\line(1,0){14}}
\put(18, 0){\circle{4}}
\put(20, 0){\line(1,0){14}}
\put(18,7){\makebox(0,0){$\al_2$}}
\put(36, 0){\circle{4}}
\put(36, 7){\makebox(0,0){$\al_{3}$}} 
\put(38, 0){\line(1,0){14}}
\put(54, 0){\circle{4}}
\put(54, 7){\makebox(0,0){$\al_{6}$}} 
\end{picture} . 
}
Let $w_0$ be the element of the Weyl group 
 $\mathcal{W}$ of $\SU(5)$ given by $ w_0 = s^{}_{\al^{}_{2}+\al^{}_{3}} \cdot s^{}_{\al^{}_1+ \al^{}_{2}+  \al^{}_{3}+ \al^{}_{6}}$.  We regard  $w_0$ as an element of the Weyl group 
of type $\E_{6}$. Then we have that $w_0(\al_{1}) = -\al_{6}$,  $w_0(\al_{2}) = -\al_{3}$, $w_0(\al_{3}) = -\al_{2}$,  $w_0(\al_{6}) = -\al_{1}$, $w_0(\al_{4}) = \al_{1}+2 \al_{2}+2\al_{3}+\al_{4}+\al_{6}$ and $w_0(\al_{5}) = \al_{5}$. 
  Thus we get $w_{0}\big(\Delta_{0}(A)\big) = \Delta_{0}(B)$ and 
   \[
 \begin{tabular}{lll}
 $w_{0}\big(\Delta^{\fr{n}}(1, 0)\big) = -\Delta^{\fr{m}}(1, 0),$ & 
 $w_{0}\big(\Delta^{\fr{n}}(0, 1)\big) = \Delta^{\fr{m}}(1, 1),$ &
 $w_{0}\big(\Delta^{\fr{n}}(1, 1)\big) = \Delta^{\fr{m}}(0, 1),$ \\
 & &  \\
 $w_{0}\big(\Delta^{\fr{n}}(0, 2)\big) = \Delta^{\fr{m}}(2, 2),$ &
 $w_{0}\big(\Delta^{\fr{n}}(1, 2)\big) = \Delta^{\fr{m}}(1, 2).$ & 
 \end{tabular}  \]
 
  \medskip
 
 For $\E_7$ we see that  the  pair $(\Pi, \Pi_{0})$ is given by  $\Pi_{\fr{n}}=\Pi\backslash\Pi_{0}=\{\al_1, \al_{7}\}$ (Type A) and  $\Pi\backslash\Pi_{0}= \Pi_{\fr{m}}=\{\al_{6}, \al_{7}\}$ (Type B) and thus the painted Dynkin diagrams of Type A and Type B contain a Dynkin subdiagram of  type $A_6$ :  \ \ 
 
 \medskip
\small{
\begin{picture}(92,0)(0,-2)
\put(0, 0){\circle{4}}
\put(0,7){\makebox(0,0){$\al_1$}}
\put(2, 0){\line(1,0){14}}
\put(18, 0){\circle{4}}
\put(20, 0){\line(1,0){14}}
\put(18,7){\makebox(0,0){$\al_2$}}
\put(36, 0){\circle{4}}
\put(36, 7){\makebox(0,0){$\al_{3}$}} 
\put(38, 0){\line(1,0){14}}
\put(54, 0){\circle{4}}
\put(54, 7){\makebox(0,0){$\al_{4}$}} 
\put(56, 0){\line(1,0){14}}
\put(72, 0){\circle{4}}
\put(72, 7){\makebox(0,0){$\al_{5}$}}
\put(74, 0){\line(1,0){14}}
\put(90, 0){\circle{4}}
\put(90, 7){\makebox(0,0){$\al_{6}$}}
\end{picture} .
}
Let $w_0$ be the element of the Weyl group 
 $\mathcal{W}$ of $\SU(7)$ given by $ w_0 = s^{}_{\al^{}_{3}+\al^{}_{4} }\cdot s^{}_{\al^{}_2+ \al^{}_{3}+  \al^{}_{4}+ \al^{}_{5}} \cdot s^{}_{\al^{}_1+ \al^{}_{2}+  \al^{}_{3}+ \al^{}_{4}+  \al^{}_{5}+ \al^{}_{6}}$.  We regard  $w_0$ as an element of the Weyl group 
of type $E_{7}$. Then we have that $w_0(\al_{1}) = -\al_{6}$,  $w_0(\al_{2}) = -\al_{5}$,  $w_0(\al_{3}) = -\al_{4}$,  $w_0(\al_{4}) = -\al_{3}$, $w_0(\al_{5}) = -\al_{2}$, $w_0(\al_{6}) = -\al_{1}$,  $w_0(\al_{7}) = \al_{1}+2 \al_{2}+3\al_{3}+3\al_{4}+2\al_{5}+\al_{6}+\al_{7}$. 
  TIt follows that  $w_{0}\big(\Delta_{0}(A)\big) = \Delta_{0}(B)$ and 
   \[
 \begin{tabular}{lll}
 $w_{0}\big(\Delta^{\fr{n}}(1, 0)\big) = -\Delta^{\fr{m}}(1, 0),$ & 
 $w_{0}\big(\Delta^{\fr{n}}(0, 1)\big) = \Delta^{\fr{m}}(1, 1),$ &
 $w_{0}\big(\Delta^{\fr{n}}(1, 1)\big) = \Delta^{\fr{m}}(0, 1),$ \\
 & &  \\
 $w_{0}\big(\Delta^{\fr{n}}(0, 2)\big) = \Delta^{\fr{m}}(2, 2),$ &
 $w_{0}\big(\Delta^{\fr{n}}(1, 2)\big) = \Delta^{\fr{m}}(1, 2).$ & 
 \end{tabular}  \]
Hence we get  an isometry between the corresponding tangent spaces $\fr{n}$ (Type A) and $\fr{m}$ (Type B)  and, therefore  we obtain an isometry between flag manifolds of Type A and  Type B.

\end{proof}

  \section{
           K\"ahler--Einstein metrics}
 In computing the Ricci tensor for a generalized flag manifold $M=G/K$ by using Riemannian submersions we will use the well known
 fact that $M$ admits a finite number of  $G$-invariant K\"ahler--Einstein metrics.
 Recall that 
 if $M=G/K$ is determined by
  a pair $(\Pi, \Pi_{K})$ with reductive decomposition  
  $\fr{g}=\fr{k}\oplus\fr{m}$, then
  $G$-invariant  complex structures are in 
one-to-one  
correspondence with invariant orderings $\Delta_{\fr{m}}^{+}$ in $\Delta_{\fr{m}}$ (\cite{Ale}, \cite[p.~625]{Bo}).
 Put $ \displaystyle Z_{\frak t} = \Bigg\{\Lambda \in {\frak t} \ \Big\vert \ \frac{2(\Lambda, \ \alpha) }{( \alpha, \  \alpha)}  \in { \mathbb Z} \  \mbox{ for }  \mbox{ each } \alpha \in \Delta \Bigg\}$. 
Then  $ Z_{\frak t} $   is a lattice of  ${\frak t} $ generated by  the fundamental weights $\{ \Lambda^{}_{i_1}, \cdots, \Lambda^{}_{i_r}\}$.  
Set $ Z_{\frak t}^{+} = \left\{ \lambda \in Z_{\frak t} \ \big\vert \  (\lambda, \alpha) > 0 \ \  \mbox{for} \  \alpha \in \Pi\setminus \Pi^{}_0 \right\}$. Then we have $\displaystyle Z_{\frak t}^{+} =  \sum_{\alpha \in  \Pi\setminus \Pi^{}_0} {\mathbb Z}^+ \Lambda_{\alpha}$ and 
 define the element   
 $\displaystyle  \delta_{\frak m} = \frac{1}{2} \sum_{\alpha \in \Delta_{\frak m}^{+} } \alpha \ \in  \sqrt{-1}\frak h$. 
 Put $ \displaystyle k_\alpha =  \frac{2 ( 2 \delta_{\frak m}, \alpha)}{(\alpha, \alpha)}$ for $ \alpha \in  \Pi \setminus \Pi^{}_0$.  Then 
 $\displaystyle 2 \delta_{\frak m} =   \sum_{ \alpha  \in  \Pi \setminus \Pi^{}_0}  k_\alpha \Lambda_\alpha
 = k_{\alpha_{i_1}}\Lambda_{\alpha_{i_1}} +\cdots + k_{\alpha_{i_r}}\Lambda_{\alpha_{i_r}} $ and each $k_{\alpha_{i_s}}$ is a positive integer. 
 We have the following:
 
 \begin{prop}
The $G$-invariant metric $g_{ 2 \delta_{\frak m} }^{}$ on $G/K$ corresponding to $ 2 \delta_{\frak m} $ is a K\"ahler Einstein metric which is given by 
$$ g_{ 2 \delta_{\frak m} }^{} =  \sum_{j_1, \cdots, j_r} \left( \sum_{\ell = 1}^{r}
k_{\alpha_{i_{\ell}}}  j_{\ell} \frac{(\alpha_{j_\ell},  \alpha_{j_\ell})}{ 2} \right) B|_{ {\frak m}( j_1, \cdots,  j_r )}.  
$$
\end{prop}

\smallskip
\begin{example} 
 \underline{\bf Case of $B_{\ell}= \SO(2\ell+1)$ : Type A.}  Let  $\Pi\backslash\Pi_{0}=\Pi_{\fr{n}}=\{\al_1, \al_{p+1}: 2 \le p \le \ell -1\}$. For the flag manifold    $M= \SO(2\ell+1)/(\U(1)\times \U(p) \times \SO(2(\ell-p-1)+1))$ with  $2\le p\le \ell -1$ and $\ell\geq 3$,  we see that    $\displaystyle 2 \delta_{\frak n}  =  (p+1) \Lambda_{\alpha_{1}}  + (2  \ell -p -2 )\Lambda_{\alpha_{p+1}} $.  
Thus the K\"ahler Einstein metric $g_{ 2 \delta_{\frak n} }^{}$ on $G/K$  is given by 
\begin{eqnarray*}
 g_{ 2 \delta_{\frak n} }^{} =  
(p+1) B|_{ {\frak n}( 1, 0 )} + (2 \ell - p -2) B|_{ {\frak n}( 0, 1 )}  
+ (2 \ell - 1) B|_{ {\frak n}( 1,  1)} \\ + 2 (2 \ell - p -2) B|_{ {\frak n}( 0, 2 )}   
+(4 \ell - p -3) B|_{ {\frak n}( 1, 2 )}. 
\end{eqnarray*}
\end{example}

\begin{example} 
 \underline{\bf Case of $D_{\ell}= \SO(2\ell)$ : Type A.}  Let  $\Pi\backslash\Pi_{0}=\Pi_{\fr{n}}=\{\al_1, \al_{p+1}: 2 \le p \le \ell -3\}$. For the flag manifold    $M= \SO(2\ell)/(\U(1)\times \U(p) \times \SO(2(\ell-p-1)))$ with  $2\le p\le \ell -3$ and $\ell\geq 5$,  we see that    $\displaystyle 2 \delta_{\frak n}  =  (p+1) \Lambda_{\alpha_{1}}  + (2  \ell -p - 3 )\Lambda_{\alpha_{p+1}} $.  
Thus the K\"ahler Einstein metric $g_{ 2 \delta_{\frak n} }^{}$ on $G/K$  is given by 
\begin{eqnarray*}
 g_{ 2 \delta_{\frak n} }^{} =  
(p+1) B|_{ {\frak n}( 1, 0 )} + (2 \ell - p - 3) B|_{ {\frak n}( 0, 1 )}  
+ (2 \ell - 2) B|_{ {\frak n}( 1,  1)} \\ + 2 (2 \ell - p -3) B|_{ {\frak n}( 0, 2 )}   
+(4 \ell - p - 5) B|_{ {\frak n}( 1, 2 )}. 
\end{eqnarray*}
\end{example}

\begin{example} 
 \underline{\bf Case of $\E_{6}$ : Type A.}  Let  $\Pi\backslash\Pi_{0}=\Pi_{\fr{n}}=\{\al_1, \al_{4} \}$. For the flag manifold    $M= \E_{6}/(\U(4)\times \U(2) )$  we see that    $\displaystyle 2 \delta_{\frak n}  =  5 \Lambda_{\alpha_{1}}  + 7 \Lambda_{\alpha_{4}} $.  
Thus the K\"ahler Einstein metric $g_{ 2 \delta_{\frak n} }^{}$ on $G/K$  is given by 
\begin{eqnarray*}
 g_{ 2 \delta_{\frak n} }^{} =  
5 B|_{ {\frak n}( 1, 0 )} + 7 B|_{ {\frak n}( 0, 1 )}  
+ 12 B|_{ {\frak n}( 1,  1)}  +14 B|_{ {\frak n}( 0, 2 )}   
+19 B|_{ {\frak n}( 1, 2 )}. 
\end{eqnarray*}
\end{example} 
  
 \begin{example} 
 \underline{\bf Case of $\E_{6}$  : Type B.}  Let  $\Pi\backslash\Pi_{0}=\Pi_{\fr{m}}=\{\al_6, \al_{4} \}$. For the flag manifold  $M= \E_{6}/(\U(4)\times \U(2) )$  we see that    $\displaystyle 2 \delta_{\frak m}  =  5  \Lambda_{\alpha_{6}}  + 6 \Lambda_{\alpha_{4}} $.  
Thus the K\"ahler Einstein metric $g_{ 2 \delta_{\frak m} }^{}$ on $G/K$  is given by 
\begin{eqnarray*}
  g_{ 2 \delta_{\frak m} }^{} =  
5 B|_{ {\frak m}( 1, 0 )} + 6 B|_{ {\frak m}( 0, 1 )}  
+ 11 B|_{ {\frak m}( 1,  1)}  + 17 B|_{ {\frak m}( 1, 2 )}   
+ 22 B|_{ {\frak m}( 2, 2 )}. 
\end{eqnarray*}
\end{example} 

\begin{example} 
 \underline{\bf Case of $\E_{7}$ : Type A.}  Let  $\Pi\backslash\Pi_{0}=\Pi_{\fr{n}}=\{\al_1, \al_{7} \}$. For the flag manifold    $M= \E_{7}/(\U(1)\times \U(6) )$  we see that    $\displaystyle 2 \delta_{\frak n}  =  7 \Lambda_{\alpha_{1}}  + 11 \Lambda_{\alpha_{7}} $.  
Thus the K\"ahler Einstein metric $g_{ 2 \delta_{\frak n} }^{}$ on $G/K$  is given by 
\begin{eqnarray*}
 g_{ 2 \delta_{\frak n} }^{} =  
7 B|_{ {\frak n}( 1, 0 )} + 11 B|_{ {\frak n}( 0, 1 )}  
+18 B|_{ {\frak n}( 1,  1)}  + 2 2 B|_{ {\frak n}( 0, 2 )}   
+29 B|_{ {\frak n}( 1, 2 )}. 
\end{eqnarray*}
\end{example} 
  
 \begin{example} 
 \underline{\bf Case of $\E_{7}$ : Type B.}   Let  $\Pi\backslash\Pi_{0}=\Pi_{\fr{n}}=\{\al_6, \al_{7} \}$. For the flag manifold    $M= \E_{7}/(\U(1)\times \U(6) )$  we see that    $\displaystyle 2 \delta_{\frak m}  =  7 \Lambda_{\alpha_{6}}  + 10 \Lambda_{\alpha_{7}} $.  
Thus the K\"ahler Einstein metric $g_{ 2 \delta_{\frak m} }^{}$ on $G/K$  is given by 
\begin{eqnarray*}
 g_{ 2 \delta_{\frak m} }^{} =  
7 B|_{ {\frak m}( 1, 0 )} + 10 B|_{ {\frak m}( 0, 1 )}  
+ 17 B|_{ {\frak m}( 1,  1)}  + 27 B|_{ {\frak m}( 1, 2 )}   
+ 34 B|_{ {\frak m}( 2, 2 )}.  
\end{eqnarray*}
\end{example} 

%
%

%

  \section{The Ricci tensor on flag manifolds with five isotropy summands }
We now proceed to the calculation of the Ricci tensor $r$ corresponding to a $G$-invariant metric (\ref{eq2}) on $G/K$ of Type A.  
  In order to apply  Lemma \ref{ric2}   we  first need to find  the non zero structure constants  $\displaystyle\genfrac{[}{]}{0pt}{}{k}{ij}$ of $G/K$. 
    Due to the bracket relations in Corollary \ref{brackets} we obtain that the non zero structure constant are
$$
\displaystyle \genfrac{[}{]}{0pt}{}{3}{12}, \ 
\genfrac{[}{]}{0pt}{}{4}{22}, \ 
\genfrac{[}{]}{0pt}{}{5}{23}, \ \genfrac{[}{]}{0pt}{}{5}{14}. 
$$
We write  $G$-invariant metrics $g$ on $G/K$  as  
\begin{eqnarray}
 g  = x_{1 }  B|_{ {\frak n}_1} + x_{2}  B|_{ {\frak n}_2} +x_{3 }  B|_{ {\frak n}_3} + x_{4 }  B|_{ {\frak n}_4} +x_{5}  B|_{ {\frak n}_5}   \label{eqAA}
\end{eqnarray}
where $ x_{j}$ ($ j =1, \dots, 5$) are   positive numbers. 

Put $d_i = \dim {\frak n}_i$ for $ i = 1, \dots, 5$.   From Lemma \ref{ric2} we obtain the following proposition. 
 \begin{prop}\label{componentsII}   The  components $r_i$  {\em(}$ i =1, \dots, 5${\em)} of the Ricci tensor
  for a   $G$-invariant Riemannian metric {\em(\ref{eqAA})}  on $G/K$ are given as follows: 
   \begin{equation}\label{ricci_5comp} 
   \left. \begin{tabular}{l}
   $ \displaystyle r_1   =  \frac{1}{2x_1} + \frac{1}{2d_1}{3 \brack {12}}\Big( \frac{x_1}{x_2 x_3}- \frac{x_2}{x_1 x_3}- \frac{x_3}{x_1 x_2}\Big)  +  \frac{1}{2d_1} {5 \brack {14}}\Big( \frac{x_1}{x_4 x_5}- \frac{x_5}{x_1 x_4}- \frac{x_4}{x_1 x_5}\Big),$ 
\\ \\
 $ r_2 \displaystyle  =\frac{1}{2x_2}+ \frac{1}{2d_2}{3 \brack {12}} \Big( \frac{x_2}{x_1x_3}- \frac{x_1}{x_2x_3}- \frac{x_3}{x_1x_2}\Big)  -\frac{1}{2d_2}{4 \brack {22}}\frac{x_4}{{x_2}^2} +  \frac{1}{2d_2}{5 \brack {23}} \Big( \frac{x_2}{x_3 x_5}- \frac{x_5}{x_2 x_3} - \frac{x_3}{x_2 x_5}\Big),$ \\ \\
  $ r_3 \displaystyle  =\frac{1}{2x_3}+  \frac{1}{2d_3}{3 \brack {12}} \Big( \frac{x_3}{x_1x_2}- \frac{x_2}{x_1x_3} - \frac{x_1}{x_2x_3}\Big)  +  \frac{1}{2d_3}{5 \brack {23}} \Big( \frac{x_3}{x_2 x_5}- \frac{x_5}{x_2 x_3}- \frac{x_2}{x_3 x_5}\Big),$ 
  \\
 \\
 $ r_4 \displaystyle  = \frac{1}{2x_4}+  \frac{1}{2d_4} {5 \brack {14}}\Big( \frac{x_4}{x_1x_5}- \frac{x_5}{x_1x_4}- \frac{x_1}{x_4 x_5}\Big) + \frac{1}{4 d_4} {4 \brack {22}} \Big(- \frac{2}{x_4} +  \frac{x_4}{{x_2}^2} \Big),$ 
 \\ 
\\
   $ r_5 \displaystyle  = \frac{1}{2x_5} + \frac{1}{2d_5}{5 \brack {23}}\Big( \frac{x_5}{x_2 x_3}- \frac{x_2}{x_3 x_5}- \frac{x_3}{x_2 x_5}\Big)  +  \frac{1}{2d_5} {5 \brack {14}}\Big( \frac{x_5}{x_1 x_4}- \frac{x_1}{x_4 x_5}- \frac{x_4}{x_1 x_5}\Big).$  
   \end{tabular} \right\}
 \end{equation}
\end{prop}  
\smallskip

Let ${\frak k}$ be the subalgebra of ${\frak g}$ corresponding to the Lie subgroup $K$. 
We consider a subspace ${\frak l} = {\frak k} \oplus {\frak n}_1$ of ${\frak g}$. Then ${\frak l} $ is a subalgebra of ${\frak g}$ and  
we have a natural  fibration  
  $\pi :  G/K \to G/L$  with fiber $L/K$. 
We decompose ${\frak p} = {\frak p}_1 \oplus {\frak p}_2$ and ${\frak q} = {\frak q}_1$, 
where $ {\frak p}_1 = {\frak n}_2 \oplus {\frak n}_3 =  {\frak m}_{1,1} \oplus  {\frak m}_{1,2}$, $ {\frak p}_2 = {\frak n}_4 \oplus {\frak n}_5 =  {\frak m}_{2,1}\oplus  {\frak m}_{2,2} $ and $ {\frak q}_1 = {\frak n}_1 $. 
 We consider a $G$-invariant metric on $G/K$ defined by a  Riemannian submersion  $\pi : ( G/K, \,  g )  \to  ( G/L, \, \check{g} )$ given by 
\begin{eqnarray}
g  =  y_1   B|_{\mbox{\footnotesize $ \frak p$}_1} +  
 y_{2}   B|_{\mbox{\footnotesize$ \frak p$}_2} + z_1 B|_{\mbox{\footnotesize$ \frak q$}_1} \label{5comp_eq4}
 \end{eqnarray} 
 and the metric $\check{g} $ on $ G/L$ 
 $$  \check{g}   = y_1   B|_{\mbox{\footnotesize$ \frak p$}_1} +  
 y_{2}   B|_{\mbox{\footnotesize$ \frak p$}_{2}}  $$ 
for positive real numbers $ y_1,  y_{2}, z_1$. Note that, when we write the metric (\ref{5comp_eq4}) as in the form (\ref{eqAA}), we have 
\begin{eqnarray}
 g  =  y_{1}  B|_{ {\frak n}_2} +y_{1 }  B|_{ {\frak n}_3} + y_{2}  B|_{ {\frak n}_4} +  y_{2}  B|_{ {\frak n}_5} +z_{1 }  B|_{ {\frak n}_1}.   \label{eqsubmer}
\end{eqnarray}

From (\ref{ricci_5comp}) we obtain the components $r_i$  of the Ricci tensor
  for the  metric {(\ref{eqsubmer})}  on $G/K$ as follows:
 $$ 
  r_3=\frac{1}{2 {y_1}} - \frac{1}{2 {d_3}}{3 \brack {12}}\frac{{z_1}}{
   {y_1}^2}- \frac{1}{2 {d_3}} {5 \brack {23}}\frac{  {y_2}}{  {y_1}^2}, $$ 
 $$ r_4  =  \frac{1}{2 y_2}-\frac{1}{2 {d_4}} {5 \brack {14}}\frac{{z_1}}{{y_2}^2}+ \frac{1}{4 {d_4}}{4 \brack {22}}
   \left(\frac{y_2}{{y_1}^2} - \frac{2}{{y_2}}\right). 
   $$

We put $ {\tilde d}_1= \dim{\frak p}_1 $ and $  {\tilde d}_2 = \dim {\frak p}_2$.  Then $ {\tilde d}_1 = d_2 + d_3$ and $ {\tilde d}_2 = d_4 + d_5$.     By Lemma \ref{ric2} (cf. also \cite[p. 10]{ACS})   the components $\check{ r}_{1}, \check{r}_{2}$  of Ricci tensor $\check{r}$ of  a $G$-invariant metric $  \check{g}   = y_1   B|_{\mbox{\footnotesize$ \frak p$}_1} +  
 y_{2}   B|_{\mbox{\footnotesize$ \frak p$}_{2}}  $  are given by 
\begin{equation}
\left\{
\begin{array}{ll} 
\check{r}_1 &=  \displaystyle{\frac{1}{2 y_1} -
\frac{y_2}{2\, \tilde{ d}_1\,{y_1}^2}  \left[{2 \brack 11}\right]
\; }
\\ & \\
\check{r}_2 &=  \displaystyle{\frac{1}{2 y_2}  - \frac{1}{2\,  \tilde{d}_2 \, y_2}  \left[{2 \brack 11}\right]
 +
\frac{y_2}{4\,  \tilde{d}_2\,{y_1}^2} \left[{2 \brack 11}\right]
, }
\end{array}
\right.
\end{equation}
where $\displaystyle \left[{2 \brack 11}\right] = \frac{\tilde{d}_1  \tilde{d}_2}{\tilde{d}_1 + 4 \tilde{d}_2}$.

Note that,  in the notation of  Lemma \ref{submersion_ricci}, we have that $r _{(1, 1)}  = r_2$, $r _{(1, 2)}  = r_3$, $r _{(2, 1)}  = r_4$ and $r _{(2, 2)}  = r_5$. 
From Lemma \ref{submersion_ricci} we see that    the horizontal part of  $r _{(1, 2)} ( = r_3 )$ equals to   $\check{r}_1 $ and the horizontal part of  $r _{(2, 1)} ( = r_4 ) $ equals to    $\check{r}_2 $, and thus we get 
\begin{eqnarray}
   {5 \brack {23}} = d_3 \frac{1}{\tilde{d}_1}\left[{2 \brack 11}\right] = \frac{d_3 (d_4+d_5)}{ (d_2+d_3)+ 4 ( d_4+d_5)},  \quad   {4 \brack {22}} =  d_4 \frac{1}{\tilde{d}_2}\left[{2 \brack 11}\right] = \frac{d_4 (d_2+d_3)}{ (d_2+d_3)+ 4 ( d_4+d_5)}.   \label{c224c235}
\end{eqnarray}

We determine  the structure constants $\displaystyle\genfrac{[}{]}{0pt}{}{k}{ij}$ in each case. 

\smallskip

\underline{\bf Case of $B_{\ell}=\SO(2\ell+1)$ : Type A.} 

 In this case  $ G= \SO(2\ell+1)$, $K=\U(1)\times \U(p) \times \SO(2(\ell-p-1)+1)$,
 $L=\U(p+1)\times\SO(2(\ell -p-1)+1)$ and we have  $ d_1= 2 p, d_2 = 2 p ( 2 \ell  -2 p -1), d_3= 2 ( 2 \ell  -2 p -1), d_4= p( p-1), d_5= 2 p$. Thus, from (\ref{c224c235}),  we see that 
 $$\ {5 \brack {23}} =\frac{(2 \ell -2 p -1) p}{2 \ell -1}, \quad  
  {4 \brack {22}} =\frac{(2 \ell -2 p -1) p (p-1)}{2 \ell -1}.$$ 
 
 Since  
 the K\"ahler Einstein metric $g_{ 2 \delta_{\frak n} }^{}$ on $G/K$  is given by 
\begin{eqnarray*}
 g_{ 2 \delta_{\frak n} }^{} =  
(p+1) B|_{ {\frak n}_1} + (2 \ell - p -2) B|_{ {\frak n}_2}  
+ (2 \ell - 1) B|_{ {\frak n}_3}  + 2 (2 \ell - p -2) B|_{ {\frak n}_4}   
+(4 \ell - p -3) B|_{ {\frak n}_5},  
\end{eqnarray*}
we  substitute the values $ x_1= p+1$, $x_2 = 2 \ell - p -2$, $x_3= 2 \ell - 1$, $x_4 =2 (2 \ell - p -2)$, $x_5=4 \ell - p -3$ into (\ref{ricci_5comp}). 
Consider the components $r_2$, $r_3$, $r_4$ and  $r_5$ of the Ricci tensor for these values. Then, from $r_2 - r_3 = 0$ and $r_4 - r_5 = 0$,  we obtain 
$$\ {3 \brack {12}} =\frac{(2 \ell -2 p -1) p}{2 \ell -1}, \quad  
  {5 \brack {14}} =\frac{ p (p-1)}{2 \ell -1}.$$ 
 
\underline{\bf Case of $D_{\ell}=\SO(2\ell)$ : Type A.} 

 In this case  $ G= \SO(2\ell)$, $K=\U(1)\times \U(p) \times \SO(2(\ell-p-1))$,
 $L=\U(p+1)\times\SO(2(\ell -p-1))$  and we have  $ d_1= 2 p, d_2 = 4 p (  \ell  -p -1), d_3= 4 ( \ell  -  p -1), d_4= p( p-1), d_5= 2 p$. Thus, from (\ref{c224c235}),  we see that 
 $$\ {5 \brack {23}} =\frac{( \ell - p -1) p}{ \ell -1}, \quad  
  {4 \brack {22}} =\frac{( \ell - p -1) p (p-1)}{  \ell -1}.$$ 
 
 Since  
 the K\"ahler Einstein metric $g_{ 2 \delta_{\frak n} }^{}$ on $G/K$  is given by 
\begin{eqnarray*}
 g_{ 2 \delta_{\frak n} }^{} =  
(p+1) B|_{ {\frak n}_1} + (2 \ell - p -3) B|_{ {\frak n}_2}  
+ (2 \ell - 2) B|_{ {\frak n}_3}  + 2 (2 \ell - p -3) B|_{ {\frak n}_4}   
+(4 \ell - p -5) B|_{ {\frak n}_5},  
\end{eqnarray*}
we  substitute the values $ x_1= p+1$, $x_2 = 2 \ell - p -3$, $x_3= 2 \ell - 2$, $x_4 =2 (2 \ell - p -3)$, $x_5=4 \ell - p -5$ into (\ref{ricci_5comp}). 
Consider the components $r_2$, $r_3$, $r_4$ and  $r_5$ of the Ricci tensor for these values. Then, from $r_2 - r_3 = 0$ and $r_4 - r_5 = 0$,  we obtain 
$$\ {3 \brack {12}} =\frac{( \ell - p -1) p}{ \ell -1}, \quad  
  {5 \brack {14}} =\frac{ p (p-1)}{2( \ell -1)}.$$ 
  
  \smallskip
  Note that we can put the cases of  $B_{\ell}$  and $D_{\ell}$ together. Consider $ G= \SO(m)$ and $K=\U(1)\times \U(p) \times \SO(m - 2(p+1))$. Then we have  $ d_1= 2 p, d_2 = 2 p ( m  -2(p +1)), d_3= 2 ( m  -  2(p+1)), d_4= p( p-1), d_5= 2 p$ thus it follows that 
    $$\ {5 \brack {23}} =\frac{( m - 2(p +1)) p}{m -2}, \quad  
  {4 \brack {22}} =\frac{(m - 2(p +1)) p (p-1)}{ m-2}$$ 
and 
$$\ {3 \brack {12}} =\frac{(m - 2(p +1) ) p}{m-2}, \quad  
  {5 \brack {14}} =\frac{ p (p-1)}{m-2}.$$

  \underline{\bf Case of $\E_{6}$ : Type A.} 
  
  In this case  $ G= \E_{6}$, $K=\U(1)\times \U(1)\times \SU(2) \times \SU(4)$,
  $L=\U(5)\times \SU(2)$ and we have  $ d_1= 8 , d_2 = 24, d_3= 16, d_4= 2, d_5= 8$. Thus, from (\ref{c224c235}),  we see that 
 $$\ {5 \brack {23}} =2, \quad  
  {4 \brack {22}} = 1.$$ 
 
 Since  
 the K\"ahler Einstein metric $g_{ 2 \delta_{\frak n} }^{}$ on $G/K$  is given by 
\begin{eqnarray*}
 g_{ 2 \delta_{\frak n} }^{} =  
5 B|_{ {\frak n}_1} + 7 B|_{ {\frak n}_2}  
+ 12 B|_{ {\frak n}_3}  + 14 B|_{ {\frak n}_4}   
+ 19 B|_{ {\frak n}_5},  
\end{eqnarray*}
we  substitute the values $ x_1= 5$, $x_2 = 7$, $x_3= 12$, $x_4 = 14$, $x_5= 19$ into (\ref{ricci_5comp}). 
Consider the components $r_2$, $r_3$, $r_4$ and  $r_5$ of the Ricci tensor for these values. Then, from $r_2 - r_3 = 0$ and $r_4 - r_5 = 0$,  we obtain 
$$\ {3 \brack {12}} =2, \quad  
  {5 \brack {14}} =\frac{ 1}{3}.$$ 
  
   \underline{\bf Case of $\E_{7}$ : Type A.}  
  
  In this case  $ G= \E_{7}$, $K=\U(1)\times \U(1) \times \SU(6)$, $L=\U(7)$ and we have  $ d_1= 12 , d_2 = 40, d_3= 30, d_4= 2, d_5= 12$. Thus, from (\ref{c224c235}),  we see that 
 $$\ {5 \brack {23}} =\frac {10} {3}, \quad  
  {4 \brack {22}} = \frac {10} {9}.$$ 
 
 Since  
 the K\"ahler Einstein metric $g_{ 2 \delta_{\frak n} }^{}$ on $G/K$  is given by 
\begin{eqnarray*}
 g_{ 2 \delta_{\frak n} }^{} =  
7 B|_{ {\frak n}_1} + 11 B|_{ {\frak n}_2}  
+18 B|_{ {\frak n}_3}  + 22 B|_{ {\frak n}_4}   
+29 B|_{ {\frak n}_5},  
\end{eqnarray*}
we  substitute the values $ x_1= 7$, $x_2 = 11$, $x_3= 18$, $x_4 =22$, $x_5=29$ into (\ref{ricci_5comp}). 
Consider the components $r_2$, $r_3$, $r_4$ and  $r_5$ of the Ricci tensor for these values. Then, from $r_2 - r_3 = 0$ and $r_4 - r_5 = 0$,  we obtain 
$$\ {3 \brack {12}} =\frac{10}{ 3}, \quad  
  {5 \brack {14}} =\frac{1}{3}.$$ 
  
%
%
%
%
%
%
%
%

  \section{Einstein metrics on flag manifolds with five isotropy summnads } 
  We consider the system of equations:  
  \begin{eqnarray}
   r_1 = r_5,\ \   r_2 = r_3,\ \   r_3=r_4,\ \   r_4=r_5. \label{einstein_equations} 
 \end{eqnarray} 

   \underline{\bf Case of $\E_{6}$ : Type A.} 
   
  The  components $r_i$  ($ i =1, \dots, 5$) of the Ricci tensor
  for a   $G$-invariant Riemannian metric { (\ref{eqAA})}  on $G/K$ are  now given as follows: 
   \begin{equation}\label{ricci_5compE6} 
   \left. \begin{tabular}{l}
   $ \displaystyle r_1   =  \frac{1}{2 x_1} + \frac{1}{8} \Big( \frac{x_1}{x_2 x_3}- \frac{x_2}{x_1 x_3}- \frac{x_3}{x_1 x_2}\Big)  +  \frac{1}{48} \Big( \frac{x_1}{x_4 x_5}- \frac{x_5}{x_1 x_4}- \frac{x_4}{x_1 x_5}\Big),$ 
\\ \\
 $ r_2 \displaystyle  =\frac{1}{2 x_2}+ \frac{1}{24} \Big( \frac{x_2}{x_1 x_3}- \frac{x_1}{x_2 x_3}- \frac{x_3}{x_1x_2}\Big)  -\frac{1}{48}\frac{x_4}{{x_2}^2} +  \frac{1}{24} \Big( \frac{x_2}{x_3 x_5}- \frac{x_5}{x_2 x_3} - \frac{x_3}{x_2 x_5}\Big),$ \\ \\
  $ r_3 \displaystyle  =\frac{1}{2x_3}+  \frac{1}{16} \Big( \frac{x_3}{x_1x_2}- \frac{x_2}{x_1x_3} - \frac{x_1}{x_2x_3}\Big)  +  \frac{1}{16} \Big( \frac{x_3}{x_2 x_5}- \frac{x_5}{x_2 x_3}- \frac{x_2}{x_3 x_5}\Big),$ 
  \\
 \\
 $ r_4 \displaystyle  = \frac{1}{2 x_4}+  \frac{1}{12} \Big( \frac{x_4}{x_1x_5}- \frac{x_5}{x_1x_4}- \frac{x_1}{x_4 x_5}\Big) + \frac{1}{8}  \Big(- \frac{2}{x_4} +  \frac{x_4}{{x_2}^2} \Big),$ 
 \\ 
\\
   $ r_5 \displaystyle  = \frac{1}{2x_5} + \frac{1}{8}\Big( \frac{x_5}{x_2 x_3}- \frac{x_2}{x_3 x_5}- \frac{x_3}{x_2 x_5}\Big)  +  \frac{1}{48} \Big( \frac{x_5}{x_1 x_4}- \frac{x_1}{x_4 x_5}- \frac{x_4}{x_1 x_5}\Big).$  
   \end{tabular} \right\}
 \end{equation}
  From $r_1-r_5= 0$, we see that 
  $$({x_1}-{x_5}) \left({x_1} {x_2}
   {x_3}+3 {x_1} {x_4} {x_5}+3
   {x_2}^2 {x_4}-12 {x_2} {x_3}
   {x_4}+{x_2} {x_3} {x_5}+3
   {x_3}^2 {x_4}\right) =0.$$ 
   
   {\bf Case of $x_5 =x_1$.}  We normalize our equations by setting $x_1 =1$. 
   We see that the system of 
   equations (\ref{einstein_equations}) reduces to the following system of polynomial equations: 
\begin{equation} \label{equ_E6_x1=1=x5}
\left.
\begin{array}{l} 
f_1= 10 {x_2}^3+{x_2}^2 {x_3} {x_4}-24
   {x_2}^2 {x_3}+2 {x_2} {x_3}^2+24 {x_2}
   {x_3}-10 {x_2}-{x_3} {x_4}=0,\\
  f_2=  10 {x_2}^3-24
  {x_2}^2-10 {x_2} {x_3}^2+24 {x_2} {x_3}+2
   {x_2}-{x_3} {x_4}=0,\\
   f_3= 3 {x_2}^3 {x_4}+2
   {x_2}^2 {x_3} {x_4}^2+2 {x_2}^2 {x_3}-12
   {x_2}^2 {x_4}-3 {x_2} {x_3}^2 {x_4}+3
   {x_2} {x_4}+3 {x_3} {x_4}^2=0
 \end{array}  
    \right\} 
    \end{equation}  
    To find non zero solutions of equations (\ref{equ_E6_x1=1=x5})
    we consider a polynomial ring $R_1= {\mathbb Q}[y, x_2, x_3, x_4] $ and an ideal $I_1$ generated by 
$$\{ f_1, \,f_2,  \,f_3, \,y \, x_2  x_3  x_4  -1\}. 
$$
We take a lexicographic order $>$  with $ y > x_2 >  x_3 > x_4$ for a monomial ordering on $R_1$. Then 
a  Gr\"obner basis for the ideal $I_1$ contains the following polynomials: 
    \begin{equation*} \label{equ_E6_x5 =1_groebner}
\begin{array}{lll}  
& & h_1 =   512683897\, {x_4}^{12}-26586224544 \,
   {x_4}^{11}+613729012600\, {x_4}^{10}-8672203136256\,
   {x_4}^9 \\ & &   +79425819414800\, {x_4}^8 -364553102019072 \,
   {x_4}^7+901989582472192\, {x_4}^6 \\ & &-1275600747577344 \,
   {x_4}^5 +1046901453080576\,  {x_4}^4   -491806714331136 \,
   {x_4}^3 \\ & &+129330076549120 \, {x_4}^2 -17647691366400\, 
   {x_4}+969515008000 , \\ & & 
  h_2 = 114848188839160612119624999242582277039963322780084212652611305472000 \, 
   {x_3}  \\ & &-752320404408788199702048033270865700909380495228817968360883312339 \, 
   {x_4}^{11}  \\ & &+38758220515867322791999260031297235508730394711323449754223470870360 \, 
   {x_4}^{10}+\cdots  \\ & & -70726659216761168399944465106568342085848237958573454324691582101708800 \,
   {x_4}  \\ & & +4794499893690636543924823161512975441415943523673600438772484784128000 , \\ & & 
  h_3= 86136141629370459089718749431936707779972492085063159489458479104000 \, 
   {x_2}  \\ & &+523691563864872091386883937890449783253444913328748812397494319729 \, 
   {x_4}^{11}  \\ & &-27032170704374631808506232904135459706304200757258747894928587117499 \, 
   {x_4}^{10}+\cdots   \\ & &+52958437343374493285824611500843861525218627552705339462466929126522880 \, 
   {x_4}  \\ & & -3633544639518951458129167566718404600177262066660959706336805062451200. 
 \end{array}  
      \end{equation*}
 By solving the first equation $h_1=0$ for $x_4$ numerically we obtain exactly  four real solutions which are approximately given by $x_4 \approx 0.1882101376884833$, $x_4  \approx 0.3421847475947193$, $x_4  \approx 1.334632880397468$ and  $x_4  \approx 1.601718258421132$. 
     We substitute these values for $ x_4 $  into  the second and  third   
equation $h_2=0$, $h_3=0$ and  we get  real solutions of the equations (\ref{einstein_equations})  which are approximately given by 
 \begin{equation*} \begin{array}{ll} 
1) & \ 
 x_1 =1, 
  x_2 \approx 0.7945133013133368, x_3 \approx 0.6083856170340604, 
  x_4 \approx 0.1882101376884833, x_5 =1,\\
2)  & \
  x_1=1, x_2 \approx 1.366407998279779, 
  x_3 \approx 1.632222678282746, 
  x_4 \approx 0.3421847475947193, 
  x_5 =1,\\
 3) & \
  x_1=1, x_2 \approx 0.7499994899122792, 
  x_3 \approx 0.6673176327222041, 
  x_4 \approx 1.334632880397468, x_5 =1,\\
 4)  & \
x_1=1, x_2 \approx 1.590451006762520, 
  x_3 \approx 1.633523267052982, x_4 \approx 1.601718258421132, x_5=1. 
  \end{array}  
    \end{equation*}   
We substitute these values for $\{ x_1, x_2, x_3, x_4, x_5 \}$  into (\ref{ricci_5compE6}) and get 
 \begin{equation*} \begin{array}{ll} 
1)  \ 
 r_1 = r_2 = r_3 = r_4 = r_5 \approx  0.4957209368544092, \quad &
2)   \
   r_1 = r_2 = r_3 = r_4 = r_5 \approx  0.2949577540873313, \\
 3)  \
   r_1 = r_2 = r_3 = r_4 = r_5 \approx  0.4702440377042893, \quad &
 4)  \
 r_1 = r_2 = r_3 = r_4 = r_5 \approx  0.2646548256739946. \\
  \end{array}  
    \end{equation*}   
Thus,  in this case we obtain four Einstein metrics with Einstein constant 1: 
\begin{equation*} \begin{array}{ll} 
1) & \ 
 x_1 \approx 0.49572094, 
  x_2 \approx 0.39385688, x_3 \approx 0.30158949, 
  x_4 \approx 0.093299706, x_5 \approx 0.49572094,\\
2)  & \
  x_1 \approx 0.29495775, x_2 \approx 0.40303263, 
  x_3 \approx 0.48143674, 
  x_4 \approx 0.10093004, 
  x_5  \approx 0.29495775,\\
 3) & \
  x_1 \approx  0.47024404, x_2 \approx 0.35268279, 
  x_3 \approx 0.31380214, 
  x_4 \approx 0.62760315, x_5  \approx  0.47024404,\\
 4)  & \
x_1 \approx 0.26465483, x_2 \approx 0.42092053, 
  x_3 \approx 0.43231982, x_4 \approx 0.42390247, x_5 \approx 0.26465483. 
  \end{array}  
    \end{equation*}   
    
  {\bf Case of $x_5 \neq x_1$.}  We normalize our equations by setting $x_1 =1$. 
 We see that the system of polynomial equations (\ref{einstein_equations}) reduces to the following system of polynomial equations: 
\begin{equation} \label{equ_E6_x1neqx5}
\left.
\begin{array}{l} 
p_1= -8 {x_2}^3 {x_4} {x_5}-2 {x_2}^3
   {x_4}-{x_2}^2 {x_3} {x_4}^2+24 {x_2}^2
   {x_3} {x_4} {x_5}-{x_2}^2 {x_3}
   {x_5}^2+{x_2}^2 {x_3}-4 {x_2} {x_3}^2
   {x_4} {x_5} \\ +2 {x_2} {x_3}^2 {x_4}-24
   {x_2} {x_3} {x_4} {x_5}+2 {x_2} {x_4}
   {x_5}^2+8 {x_2} {x_4} {x_5}+{x_3}
   {x_4}^2 {x_5}=0, \\
 p_2=5 {x_2}^3 {x_5}+5 {x_2}^3-24
   {x_2}^2 {x_5}-5 {x_2} {x_3}^2 {x_5}-5
   {x_2} {x_3}^2+24 {x_2} {x_3} {x_5}+{x_2}
   {x_5}^2+{x_2} {x_5}-{x_3} {x_4} {x_5}=0, \\
p_3=-3{x_2}^3 {x_4} {x_5}-3 {x_2}^3 {x_4}-4
   {x_2}^2 {x_3} {x_4}^2+4 {x_2}^2 {x_3}
   {x_5}^2-12 {x_2}^2 {x_3} {x_5}+4 {x_2}^2
   {x_3}+24 {x_2}^2 {x_4} {x_5} \\
   +3 {x_2}
   {x_3}^2 {x_4} {x_5}+3 {x_2} {x_3}^2
   {x_4}-3 {x_2} {x_4} {x_5}^2-3 {x_2}
   {x_4} {x_5}-6 {x_3} {x_4}^2 {x_5}=0, \\
p_4=  3{x_2}^2 {x_4}-12 {x_2} {x_3} {x_4}+{x_2}
   {x_3} {x_5}+{x_2} {x_3}+3 {x_3}^2
   {x_4}+3 {x_4} {x_5} =0.  \end{array}  
    \right\} 
    \end{equation}  
    To find non zero solutions of equations (\ref{equ_E6_x1neqx5}),
    we consider a polynomial ring $R_2= {\mathbb Q}[y, x_2, x_3, x_4, x_5] $ and an ideal $I_2$ generated by 
$$\{ p_1, \, p_2,  \, p_3,  \, p_4,  \, y x_2  x_3  x_4 x_5  -1\}. 
$$
We take a lexicographic order $>$  with $ y >  x_2  >  x_5 >  x_3 >  x_4$ for a monomial ordering on $R_2$. Then 
a  Gr\"obner basis for the ideal $I_2$ contains a polynomial $$(5 {x_4}-22) (5 {x_4}-14) (17 {x_4}-22) (19 {x_4}-14)q_1, $$
 where  
    \begin{equation} \label{equ_E6_x5neq1_groebner}
\begin{array}{lll}  
& & q_1 = 25684944948354308203125 {x_4}^{24} 
   -312330714783423219879187500
   {x_4}^{23} \\ 
   & &-14789576030598686784365775000 {x_4}^{22} 
   +169312435225853499159893370000 
   {x_4}^{21}+\cdots \\ 
   & &-597859726821790689492624998400
   {x_4}^4 +84059799581674625557541683200
   {x_4}^3 \\ 
   & &-2979131989754489205686272000
   {x_4}^2 -1842910805533143334912000000
   {x_4} \\ 
   & &+333622121893933875200000000. 
\end{array}
  \end{equation}
   For the case when  $(5 {x_4}-22) (5 {x_4}-14) (17 {x_4}-22) (19 {x_4}-14) = 0$,  
      we consider  ideals $I_3$, $I_4$, $I_5$, $I_6$ of  the polynomial ring $R_2= {\mathbb Q}[y, x_2, x_3, x_4, x_5] $ generated by 
\begin{eqnarray*}
\{ p_1, \, p_2,  \, p_3,  \, p_4,  \,y,  \, x_2  x_3  x_4 x_5  -1, 5 {x_4}-22 \},  & 
 \{ p_1, \, p_2,  \, p_3,  \, p_4,  \,y,  \, x_2  x_3  x_4 x_5  -1, 5 {x_4}-14 \}, \\
\{ p_1, \, p_2,  \, p_3,  \, p_4,  \,y,  \, x_2  x_3  x_4 x_5  -1, 17 {x_4}-22 \}, &
\{ p_1, \, p_2,  \, p_3,  \, p_4,  \,y,  \, x_2  x_3  x_4 x_5  -1, 17 {x_4}-14 \}
\end{eqnarray*}
 respectively. 

We take a lexicographic order $>$  with $ y >  x_2  >  x_5 >  x_3 >  x_4$ for a monomial ordering on $R_2$. Then 
 Gr\"obner bases for the ideals $I_3$, $I_4$, $I_5$, $I_6$ contain polynomials  
\begin{eqnarray*}
\{5 {x_4}-22,5 {x_3}-6,5 {x_5}-17,5 {x_2}-11 \},  & \{5 {x_4}-14,5 {x_3}-12,5 {x_5}-19,5 {x_2}-7 \}, \\
\{17 {x_4}-22,17 {x_3}-6,17 {x_5}-5,17 {x_2}-11 \}, & \{19 {x_4}-14,19 {x_3}-12,19 {x_5}-5,19 {x_2}-7 \}. 
\end{eqnarray*}
 respectively. Thus we obtain  the following solutions of equations (\ref{equ_E6_x1neqx5}): 
  
 \begin{equation*} 
 \begin{array}{ll} 
\displaystyle  1) \   x_1 = 1, {x_2}= \frac{11}{5},  {x_3}= \frac{6}{5}, {x_4}= \frac{22}{5}, {x_5}= \frac{17}{5},  & \displaystyle  2)  \  x_1 = 1, {x_2}= \frac{7}{5},  {x_3}= \frac{12}{5}, {x_4}= \frac{14}{5}, {x_5}= \frac{19}{5}, 
\\ &
\\ 
  \displaystyle  3)  \ x_1 = 1,  {x_2}= \frac{11}{17},  {x_3}= \frac{6}{17},  {x_4}= \frac{22}{17},  {x_5}= \frac{5}{17}, 
 &   \displaystyle  4) \ x_1 = 1,  {x_2}= \frac{7}{19},  {x_3}= \frac{12}{19},  {x_4}= \frac{14}{19}, {x_5}= \frac{5}{19}.
\end{array}
 \end{equation*} 
We normalize  these solutions as follows: 
\begin{eqnarray*}  \displaystyle  1) \  x_1 = 5, {x_2}= 11, {x_3}=6, {x_4}= 22, {x_5}= 17,   \quad  \displaystyle  2) \ x_1 = 5, {x_2}= 7, {x_3}=12, {x_4}= 14, {x_5}= 19, \\ 
  \displaystyle  3) \  x_1 = 17,  {x_2}= 11, {x_3}= 6,  {x_4}= 22, {x_5}= 5, \quad 
   \displaystyle 4) \ x_1 = 19,  {x_2}= 7, {x_3}= 12, {x_4}= 14, {x_5}= 5.
\end{eqnarray*}
and we get K\"ahler Einstein metrics for these values of $x_i$'s. Note that  the metrics corresponding to the cases 1) and 3) are isometric and the cases 2) and 4) are isometric.

For the case when  $q_1 = 0$ and $(5 {x_4}-22) (5 {x_4}-14) (17 {x_4}-22) (19 {x_4}-14) \neq 0$,   
      we consider  a ideal $I_7$ of  the polynomial ring $R_2= {\mathbb Q}[y, x_2, x_3, x_4, x_5] $ generated by 
$$
\{ p_1, \, p_2,  \, p_3,  \, p_4,   \, y (5 {x_4}-22) (5 {x_4}-14) (17 {x_4}-22) (19 {x_4}-14) x_2  x_3  x_4 x_5  -1  \}. $$
We take the same lexicographic order $>$  with $ y >  x_2  >  x_5 >  x_3 >  x_4$ for a monomial ordering on $R_2$.
Then 
a  Gr\"obner basis for the ideal $I_7$ contains the polynomial $q_1$ and polynomials of the form  
\begin{eqnarray} b_2 x_2 + v_2(x_4), \quad  b_3 x_3 + v_3(x_4), \quad  b_5 x_5 + v_5(x_4)
\end{eqnarray}
where $b_2, b_3, b_5$ are positive integers and $v_2(x_4), v_3(x_4), v_5(x_4)$ are polynomials  of degree 23 with integer coefficients. 

 By solving the equation $q_1=0$ for $x_4$ numerically, we obtain exactly  6 positive solutions, 8 negative solutions and 10 non-real  solutions. The 6 positive solutions are approximately given by 
 \begin{eqnarray*}  & & 1) \, x_4 \approx 1.157018562397866, \quad 2) \, x_4 \approx 2.075646788197390, \quad 3) \, x_4 \approx  2.145057741729789, \\ & & 
 4) \,  x_4 \approx 2.163849575049888,\quad  5) \,  x_4 \approx  12.97930323340096, \quad 6) \, x_4 \approx 12207.19468694106.
 \end{eqnarray*}
     We substitute the values for $ x_4 $  into the equations  $b_2 x_2 + v_2(x_4)=0, \ b_3 x_3 + v_3(x_4)=0, \  b_5 x_5 + v_5(x_4)=0$. Then we obtain the following values approximately: 
     \begin{eqnarray*}   
&& 1) \, x_4 \approx 1.15702,  x_2 \approx  0.641194, x_3 \approx  0.566074, x_5 \approx  0.557426,\\
&& 2) \, x_4 \approx 2.07565, x_2 \approx  1.15028, x_3 \approx  1.01551, x_5 \approx  1.79396, \\
&& 3) \, x_4 \approx  2.14506, x_2 \approx  8.87367, x_3 \approx  33.3409, x_5 \approx  -1.12628, \\ 
&& 4) \,  x_4 \approx 2.16385, x_2 \approx  27.3523, x_3 \approx  7.26471, x_5 \approx  -1.16127, \\
&& 5) \,  x_4 \approx  12.9793, x_2 \approx  1.3699, x_3 \approx  5.42602, x_5 \approx  -1.49194, \\
&& 6) \, x_4 \approx 12207.2, x_2 \approx  18.0447, x_3 \approx  1.46532, x_5 \approx  -221.833.
 \end{eqnarray*}
Thus we see that only the cases 1) and 2) correspond to Einstein metrics. 
We substitute these values for $\{ x_1, x_2, x_3, x_4, x_5 \}$  into (\ref{ricci_5compE6}) and get 
 \begin{equation*} \begin{array}{ll} 
1)  \ 
 r_1 = r_2 = r_3 = r_4 = r_5 \approx  0.31855, \quad &
2)   \
   r_1 = r_2 = r_3 = r_4 = r_5 \approx  0.571467. 
  \end{array}  
    \end{equation*}   
Thus we obtain two Einstein metrics with Einstein constant 1: 
\begin{equation*} \begin{array}{ll} 
1) & \ 
 x_1 \approx 0.31855, 
  x_2 \approx 0.366421, x_3 \approx 0.323492, 
  x_4 \approx 0.661198, x_5 \approx 0.571467,\\
2)  & \
 x_1 \approx 0.571467, 
  x_2 \approx 0.366421, x_3 \approx 0.323492, 
  x_4 \approx 0.661198, x_5 \approx 0.31855. 
  \end{array}  
    \end{equation*}   
Now we see that these two metrics are isometric. 

  \begin{theorem}\label{E_6} The flag manifold $\E_6/(\SU(4)\times \SU(2)\times \U(1)\times \U(1) )$ admits exactly seven $\E_6$-invariant Einstein metrics up   to isometry. There are two K\"ahler-Einstein metrics  {\em(}up to scalar{\em )} given by 
    \begin{eqnarray*}   \{ x_1 = 5, {x_2}= 7, {x_3}=12, {x_4}= 14, {x_5}= 19 \},  \quad \{x_1 = 5, {x_2}= 11, {x_3}=6, {x_4}= 22, {x_5}= 17 \}. 
  \end{eqnarray*} 
 The other five are non-K\"ahler. These metrics are given approximately by 
     \begin{eqnarray*}   
&& \{ x_1 \approx 0.571467, 
  x_2 \approx 0.366421, x_3 \approx 0.323492, 
  x_4 \approx 0.661198, x_5 \approx 0.31855\},  \\
& & \{x_1 \approx 0.49572094, 
  x_2 \approx 0.39385688, x_3 \approx 0.30158949, 
  x_4 \approx 0.093299706, x_5 \approx 0.49572094\}, \\
& & \{
  x_1 \approx 0.29495775, x_2 \approx 0.40303263, 
  x_3 \approx 0.48143674, 
  x_4 \approx 0.10093004, 
  x_5  \approx 0.29495775\}, \\
& & \{
  x_1 \approx  0.47024404, x_2 \approx 0.35268279, 
  x_3 \approx 0.31380214, 
  x_4 \approx 0.62760315, x_5  \approx  0.47024404\}, \\
 & & \{
x_1 \approx 0.26465483, x_2 \approx 0.42092053, 
  x_3 \approx 0.43231982, x_4 \approx 0.42390247, x_5 \approx 0.26465483\}. 
  \end{eqnarray*}   
  \end{theorem}

 \underline{\bf Case of $\E_{7}$ : Type A.} 
 
  The  components $r_i$  ($ i =1, \cdots, 5$) of the Ricci tensor
  for a   $G$-invariant Riemannian metric { (\ref{eqAA})}  on $G/K$ are given as follows: 
   \begin{equation}\label{ricci_5compE7} 
   \left. \begin{tabular}{l}
   $ \displaystyle r_1   =  \frac{1}{2 x_1} + \frac{5}{36} \Big( \frac{x_1}{x_2 x_3}- \frac{x_2}{x_1 x_3}- \frac{x_3}{x_1 x_2}\Big)  +  \frac{1}{72} \Big( \frac{x_1}{x_4 x_5}- \frac{x_5}{x_1 x_4}- \frac{x_4}{x_1 x_5}\Big),$ 
\\ \\
 $ r_2 \displaystyle  =\frac{1}{2 x_2}+ \frac{1}{24} \Big( \frac{x_2}{x_1 x_3}- \frac{x_1}{x_2 x_3}- \frac{x_3}{x_1x_2}\Big)  -\frac{1}{72}\frac{x_4}{{x_2}^2} +  \frac{1}{24} \Big( \frac{x_2}{x_3 x_5}- \frac{x_5}{x_2 x_3} - \frac{x_3}{x_2 x_5}\Big),$ \\ \\
  $ r_3 \displaystyle  =\frac{1}{2x_3}+  \frac{1}{18} \Big( \frac{x_3}{x_1x_2}- \frac{x_2}{x_1x_3} - \frac{x_1}{x_2x_3}\Big)  +  \frac{1}{18} \Big( \frac{x_3}{x_2 x_5}- \frac{x_5}{x_2 x_3}- \frac{x_2}{x_3 x_5}\Big),$ 
  \\
 \\
 $ r_4 \displaystyle  = \frac{1}{2 x_4}+  \frac{1}{12} \Big( \frac{x_4}{x_1x_5}- \frac{x_5}{x_1x_4}- \frac{x_1}{x_4 x_5}\Big) + \frac{5}{36}  \Big(- \frac{2}{x_4} +  \frac{x_4}{{x_2}^2} \Big),$ 
 \\ 
\\
   $ r_5 \displaystyle  = \frac{1}{2x_5} + \frac{5}{36}\Big( \frac{x_5}{x_2 x_3}- \frac{x_2}{x_3 x_5}- \frac{x_3}{x_2 x_5}\Big)  +  \frac{1}{72} \Big( \frac{x_5}{x_1 x_4}- \frac{x_1}{x_4 x_5}- \frac{x_4}{x_1 x_5}\Big).$  
   \end{tabular} \right\}
 \end{equation}
 By using similar method as for the case of $\E_6$ we end up to the following:

  \begin{theorem}\label{E_7} The flag manifold $\E_7/( \U(1)\times \U(6) )$ admits exactly seven $\E_7$-invariant Einstein metrics up   to isometry. There are two K\"ahler-Einstein metrics {\em(}up to scalar{\em )} given by 
    \begin{eqnarray*}    \{  x_1 = 7, {x_2}= 11, {x_3}=18, {x_4}= 22, {x_5}= 29 \},  \quad  \{ x_1 = 7, {x_2}= 17, {x_3}= 10, {x_4}= 34, {x_5}= 27 \}.
  \end{eqnarray*} 
 The other five are non-K\"ahler. These metrics are given approximately by 
     \begin{eqnarray*}   
&& \{ x_1 \approx 0.63931715, 
 x_2 \approx 0.37800271, x_3 \approx 0.34993635, x_4 \approx 0.69900421,  x_5 \approx 0.27564786 \},  \\
& & \{ x_1 \approx 0.52602201, 
  x_2 \approx 0.38291429, x_3 \approx 0.32460549, 
  x_4 \approx 0.060058655, x_5 \approx 0.52602201 \}, \\
& & \{
   x_1 \approx 0.26773609, x_2 \approx 0.42433469, 
  x_3 \approx 0.46801223, 
  x_4 \approx 0.063305828, 
  x_5  \approx 0.2677360 \}, \\
& & \{
  x_1 \approx  0.50711535, x_2 \approx 0.35565283, 
  x_3 \approx 0.33123840, 
  x_4 \approx 0.64238182, x_5  \approx  0.50711535 \}, \\
 & & \{
x_1 \approx 0.24046904, x_2 \approx 0.43874160, 
  x_3 \approx 0.44384361, x_4 \approx 0.39782398, x_5 \approx 0.24046904 \}. 
 \end{eqnarray*}   
  \end{theorem} 
 
Now we consider the cases of $B_{\ell}=\SO(2\ell+1)$ and  $D_{\ell}=\SO(2\ell)$ together. 

\smallskip
  \underline{\bf Case of $\SO(m)$ : Type A.} 
 \smallskip
 
 The  components $r_i$  ($ i =1, \dots, 5$) of the Ricci tensor
  for a   $G$-invariant Riemannian metric { (\ref{eqAA})}  on $G/K = \SO(m)/( \U(1)\times \U(p) \times \SO(m -2 - 2 p) )$ are  now given as follows: 
  \begin{equation}\label{ricci_5compB_ell} 
   \left. \hspace{-5pt} \begin{tabular}{l}
 $ \displaystyle r_1   =  \frac{1}{2 x_1} + \frac{ m - 2 - 2 p}{4( m - 2)} \Big( \frac{x_1}{x_2 x_3}- \frac{x_2}{x_1 x_3}- \frac{x_3}{x_1 x_2}\Big)  +  \frac{ p-1}{4 (m - 2)}\Big( \frac{x_1}{x_4 x_5}- \frac{x_5}{x_1 x_4}- \frac{x_4}{x_1 x_5}\Big),$ 
\\   $ r_2 \displaystyle  =\frac{1}{2 x_2}+ \frac{1}{4 (m - 2)} \Big( \frac{x_2}{x_1 x_3}- \frac{x_1}{x_2 x_3}- \frac{x_3}{x_1x_2}\Big)  -\frac{ p-1 }{4 (m - 2) }\frac{x_4}{{x_2}^2}  +  \frac{1}{4 (m - 2) } \Big( \frac{x_2}{x_3 x_5}- \frac{x_5}{x_2 x_3} - \frac{x_3}{x_2 x_5}\Big),$ \\  
  $ r_3 \displaystyle  =\frac{1}{2x_3}+  \frac{p}{4 (m - 2)} \Big( \frac{x_3}{x_1x_2}- \frac{x_2}{x_1x_3} - \frac{x_1}{x_2x_3}\Big)  +  \frac{p}{4 (m - 2)} \Big( \frac{x_3}{x_2 x_5}- \frac{x_5}{x_2 x_3}- \frac{x_2}{x_3 x_5}\Big),$ 
  \\
 $ r_4 \displaystyle  = \frac{1}{2 x_4}+  \frac{1}{2 (m - 2)} \Big( \frac{x_4}{x_1x_5}- \frac{x_5}{x_1x_4}- \frac{x_1}{x_4 x_5}\Big) + \frac{(m - 2 - 2 p)}{4 (m - 2)}  \Big(- \frac{2}{x_4} +  \frac{x_4}{{x_2}^2} \Big),$ 
 \\ 
   $ r_5 \displaystyle  = \frac{1}{2x_5} + \frac{m - 2 -2 p}{4 (m - 2)}\Big( \frac{x_5}{x_2 x_3}- \frac{x_2}{x_3 x_5}- \frac{x_3}{x_2 x_5}\Big)  +  \frac{ p-1}{4 (m - 2)} \Big( \frac{x_5}{x_1 x_4}- \frac{x_1}{x_4 x_5}- \frac{x_4}{x_1 x_5}\Big).$  
   \end{tabular}   \hspace{-6pt} \right\}
 \end{equation}
 From $r_1-r_5= 0$, we see that 
 \begin{equation*}
 \begin{array}{l}  ({x_1}-{x_5}) 
 \left( (m -2 -2 p) {x_1} {x_4}{x_5} + (m -2 -2 p) {x_2}^2 {x_4} + (m -2 -2 p) {x_3}^2 {x_4}  \right. \\ -2 (m-2) {x_2} {x_3} {x_4} 
 \left. 
 +2 (p-1) {x_1} {x_2} {x_3}+2 (p-1) {x_2} {x_3} {x_5}\right) =0. 
 \end{array}
   \end{equation*} 
   
   {\bf Case of $x_5 =x_1$.}  We normalize our equations by setting $x_1 =1$. 
   We see that the system of polynomial equations (\ref{einstein_equations}) reduces to the following system of polynomial equations: 
      \begin{equation} \label{equ_SO(m)_x1=1=x5}
\left.
\begin{array}{l} 
f_1 =  -(m -2 p){x_2}^3 - (m -4 -2 p) {x_2} {x_3}^2 + (m-2 p) {x_2} 
   \\  +2 (m-2) {x_2}^2 {x_3} 
   -2 (m-2) {x_2} {x_3}-(p-1) {x_2}^2 {x_3}
   {x_4}+(p-1) {x_3} {x_4}=0,\\
  f_2 = -2 (m-2) {x_2}^2+2 (m-2) {x_2} {x_3}+2
   (p+1) {x_2}^3-2 (p+1) {x_2}
   {x_3}^2 \\  +2 (p-1) {x_2}-(p-1) {x_3}
   {x_4}=0,\\
   f_3 =  - (m - 2 - 2 p ) {x_3} {x_4}^2 +2
   (m - 2) {x_2}^2 {x_4}-2 p\, {x_2}^3
   {x_4}-4 (p-1) {x_2}^2 {x_3} 
   \\   +2 p\, {x_2} {x_3}^2 {x_4}-2 p\, {x_2}
   {x_4}-2 \, {x_2}^2 {x_3} {x_4}^2 =0. 
 \end{array}  
    \right\} 
    \end{equation}  
  
  To find non zero solutions of equations (\ref{equ_SO(m)_x1=1=x5}),
    we consider a polynomial ring $R= {\mathbb Q}[y, x_2, x_3, x_4] $ and an ideal $I_1$ generated by 
$$\{ f_1, \,f_2,  \,f_3, \,y \, x_2  x_3  x_4  -1\}. 
$$
 We take a lexicographic order $>$  with $ y > x_2 >  x_4 > x_3$  for a monomial ordering on $R$. Then 
 we see that a  Gr\"obner basis for the ideal $I_1$ contains the following polynomial $h_1(x_3)$ of degree 12 :   
    { \small  
   \begin{equation*} \label{equ_SO(m)_x3}
\begin{array}{l}  
h_1(x_3) = 16 (p+1)^5 \left(-p^2+m p-5 p+2 m-4\right)^2 \left(p^2+4
   p-1\right) \left(p^3+5 p^2-16 m p+35 p+8 m^2-32 m+31\right)
   {x_3}^{12} \\
   -32 (m-2) (p+1)^4 \left(-p^2+m p-5 p+2
   m-4\right) (-5 p^7+5 m p^6-62 p^6-m^2 p^5+116 m p^5-403
   p^5-88 m^2 p^4 \\
   +863 m p^4-1562 p^4+24 m^3 p^3-600 m^2
   p^3+2642 m p^3-3067 p^3+128 m^3 p^2-1226 m^2 p^2+3241 m
   p^2-2570 p^2 \\
   +128 m^3 p-655 m^2 p+1074 m p-557 p-24 m^3+138
   m^2-261 m+162 ) {x_3}^{11}+\cdots \\
   -4 (m-2) (m-2 p-2) (m-2 p) (m-p-1)^3 (m+2 p-2)(56 p^7-100 m p^6+272 p^6+58 m^2 p^5-368 m p^5+448 p^5 \\
   -7 m^3 p^4+64 m^2 p^4-172 m p^4+72 p^4-4 m^4 p^3+100 m^3 p^3-408 m^2 p^3+624 m p^3-352 p^3+m^5 p^2-41 m^4 p^2\\
   +253 m^3p^2-622 m^2 p^2+696 m p^2-304 p^2-2 m^5 p-m^4 p+66 m^3 p-212 m^2 p+248 m p-96 p+2 m^6-15 m^5 \\+40 m^4-40 m^3+16 m ) {x_3} \\
   +(m-2 p-2)^2 (m-2 p)^2 (m-p-1)^4 (m+2 p-2)^2  (4 p^4-4 m p^3+16 p^3+m^2 p^2-8 m p^2+16 p^2-4 m^2 p \\ 
   +16 m p-16 p+2 m^3-12 m^2+24 m-16 ) 
  
      \end{array}
\end{equation*}  
 }
and polynomials of the form  
\begin{eqnarray}\label{equ_SO(m)_x2x3}
 b_2 x_2 + v_2(x_3), \quad  b_3 x_4 + v_3(x_3), 
\end{eqnarray}
where $b_2, b_3$ are  integers depending on $m$ and $p$ and $v_2(x_3), v_3(x_3)$ are polynomials  of degree 11 with integer coefficients depending on $m$ and $p$. 

Note that for $\displaystyle  2 \leq p \leq \frac{m-3}{2}$,  we see that 
 \begin{equation*} 
\begin{array}{l}h_1(0) = (m-2 p-2)^2 (m-2 p)^2 (m-p-1)^4 (m+2 p-2)^2 \times \\ (2 (m- 2 p-2)^3+( p^2+8 p ) (m-2 p-2)^2+4 p^2 (m-2 p-2)+4 p^2 )  > 0 
 \end{array}
\end{equation*}   
and the head coefficient of $h_1(x_3)$ (that is the coefficient of degree 12)  is given by 
 \begin{equation*} 
\begin{array}{l}16 (p + 1)^5 \left (p^2 + 4 p - 1 \right) \left ((p+2) (m-2 p-2)+p^2+p \right)^2 \times \\  \left ( 8 (m - 2 p - 2)^2 +16 p (m - 2 p - 2) +  p^3 + 5 p^2 + 3 p - 1 \right)  > 0.  
 \end{array}
\end{equation*}
 We claim that there exists ${ x_3}^0 > 0$ such that  $h_1({x_3}^0) < 0$. Then we see that  there exist at least two positive solutions of the equation $h_1(x_3) = 0$. For fixed $m$ we divide $p$ into the following 4 cases: 

(1) \  the case when  $\displaystyle  2 \leq p \leq \frac{m}{4}$ 
\quad  
 (2) \  the case when  $\displaystyle  \frac{m}{4}+1 \leq p \leq   \frac{m}{3} $ 

(3) \ 
  the case when   $\displaystyle   \frac{m}{3} +1 \leq p \leq  \frac{3}{8} m$ 
\quad 
 (4)\  the case when  $\displaystyle  \frac{3}{8} m+1  \leq p \leq \frac{m- 3}{2}$.

\medskip 

Case (1). \ We put $\displaystyle {x_3}^0 = \frac{1}{2} + \frac{13}{16 m} - \frac{ 5 p}{ 16 m}$. We claim that for 
$\displaystyle 3 \leq p \leq  \frac{m}{4}$, $h_1({x_3}^0) < 0$. 
Consider the value $\displaystyle h_1(\frac{1}{2} + \frac{13}{16 m} - \frac{ 5 p}{ 16 m})$. We see that 
$$ \displaystyle  h_1(\frac{1}{2} + \frac{13}{16 m} - \frac{ 5 p}{ 16 m}) = - \frac{1}{17592186044416 m^{12}} G_1(m, p),$$
where $G_1(m, p)$ is a polynomial of $m$ and $p$ with integer coefficients of degree 23 for $m$. 
This is given by 
$$G_1(m, p) = \sum_{k =0}^{23} a_k(p) (m- 4 p)^k,$$
where $a_k(p)$ are  polynomials of $p$ with integer coefficients. We expand each $a_k(p)$ by $p-3$  and we see that
these are polynomials  of $p-3$ with positive integer coefficients. For example, we have 
 \begin{equation*} 
\begin{array}{l}
a_0(p) = 178237127754237399126183 (p-3)^{26}
+13265901008221449505213854
   (p-3)^{25} \\
   +468508725568700912318217147
   (p-3)^{24}+10476580337328823577318977524
   (p-3)^{23}+\cdots \\
   +5459366557078936923770981445634359296
   (p-3)^2+102099050788760068306158510149730304
   (p-3)\\+9153796573419518258107893315272704. 
 \end{array}
\end{equation*} 
Thus we see that for $p \geq 3$ and $m- 4 p \geq 0$, $G_1(m, p)$ is positive. For $p =2$, we have that 
 \begin{equation*} 
\begin{array}{l}
G_1(m, 2) = 4947802324992 m^{23}-238731462180864 m^{22}+4683833634979840
   m^{21} \\
   -46281577591734272 m^{20}+202492806617366528
   m^{19}+347599071281676288 m^{18}\\
   -9310980063572787200
   m^{17}+52574830445585235968 m^{16}-150817861595192885248
   m^{15} \\
   +203948015024640884736 m^{14}+24172844877444808704
   m^{13}-443862342994666192896 m^{12} \\
   +385551424965459050496
   m^{11}+234121922151674609664 m^{10}-374049639831778762752
   m^9 \\
   -75293632155127080960 m^8+131014157184763195392
   m^7+41471745352938388224 m^6 \\
   -4230801125626406400
   m^5-3773984791973043456 m^4-724276391563682496
   m^3\\ 
   -67160272036488624 m^2-3137789825780976 m-59372964780228. 
 \end{array}
\end{equation*} 
By expanding  $G_1(m, 2) $ by $m-13$, we obtain that 
 \begin{equation*} 
\begin{array}{l}
G_1(m, 2) = 
 4947802324992 (m-13)^{23}+1240661432991744
   (m-13)^{22}+147959819460935680
   (m-13)^{21} \\ +11163907197560160256
   (m-13)^{20}+598016367241983950848
   (m-13)^{19}+\cdots  \\ +863786663385687333093846137528883728
   (m-13)^2+505363778599954771716113864626795760
   (m-13)\\ +91549876964199619601498344378250268. 
 \end{array}
\end{equation*} 
Thus we obtain that $G_1(m, 2) > 0$ for $m \geq 13$. 

\smallskip

Case (2). \ We put $\displaystyle {x_3}^0 = \frac{19}{50} $. We claim that for 
$\displaystyle  \frac{m}{4}+2 \leq p \leq   \frac{m}{3}$, $h_1({x_3}^0) < 0$.
Consider the value $\displaystyle h_1(\frac{19}{50})$ for $\displaystyle  p = \frac{m}{4} + s$ where $s$ is a positive integer.  We see that 
$$ \displaystyle  h_1( \frac{19}{50}) = - \frac{1}{4096000000000000000000000000} G_2(m, s),$$
where $G_2(m, s)$ is a polynomial of $m$ and $s$ with integer coefficients of degree 14 for $m$. 
We see that the polynomial $G_2(m, s)$ is of the form given by 
$$G_2(m, s) = \sum_{k =0}^{14} b_k(s) (m- 12 s)^k,$$
where $b_k(s)$ are  polynomials of $p$ with integer coefficients. We  see that each 
$b_k(s)$ is a polynomial  of $s-2$ with positive integer coefficients. For example, we have 
{\small 
 \begin{equation*} 
\begin{array}{l}
b_0(s) = 1506786986744786694940025493563375616
   (s-2)^{14}  +36998433298516093734987416088141103104
   (s-2)^{13} \\ +416136307149363560959687947416881856512
   (s-2)^{12}  
   +2842633983062558684587917475200569966592
   (s-2)^{11}+\cdots \\ +57949158057391373824741968217195103125504
   (s-2)^2+14661733981405213296399078855588634951680
   (s-2) \\+1759509038746291790869701479717803130880. 
 \end{array}
\end{equation*} }
Thus we see that, for $s \geq 2$ and $m- 12 s \geq 0$, $G_2(m, s)$ is positive.  Note that $2 \leq s \leq m/12$ and thus 
$p \leq  m/4 + m/12 = m/3$.  

For $s =1$, that is $p = m/4 +1$,  we consider $ \displaystyle  h_1( \frac{21}{50})$. Then we see that 
$$
 \displaystyle  h_1( \frac{21}{50}) = -\frac{9}{4096000000000000000000000000} H_2(m, 1)$$
 where 
 {\small 
 \begin{equation*} 
\begin{array}{l}  H_2(m, 1) = 1054050555264795935559 
   m^{14}-25389815873416469983512 
   m^{13}  -966477257093633919382992 
   m^{12} \\ +46723891545491804668385536 
   m^{11}-692407952396029127541554176 
   m^{10}  +2877161421721862355752550400 
   m^9 \\ +38703106006797200198212583424 
   m^8  -637677740991893898125100711936 
   m^7 \\ +4388384700195221430188604653568 
   m^6  -17925423535989571036538101301248 
   m^5 \\
   +47001713463749690546636544016384 
   m^4 -80077090514342627715801111592960 
   m^3 \\ +85882621352257394136232639856640 
   m^2  -52734119195798771677768817049600 
   m  \\ +14142949365346227611634342297600.  
     \end{array}
\end{equation*} 
}
We see that 
 \begin{equation*} 
\begin{array}{l}  H_2(m, 1) =
   1054050555264795935559 (m-13)^{14}+166447385184776390288226
   (m-13)^{13} \\
   +10952887349716279346365341
   (m-13)^{12}+404197548045208433956000372
   (m-13)^{11}+\cdots \\ +15512402577878159456329789083376125
   (m-13)^2+19060854444302720753441098137077730
   (m-13) \\+1049301029441675428621431247614375. 
 \end{array}
\end{equation*} 
Thus $H_2(m, 1)$ is positive for $ m \geq 13$.

 \smallskip 
 
Case (3). \ We put $\displaystyle {x_3}^0 = \frac{1}{3}$. We claim that for 
$\displaystyle \frac{m}{3} +1 \leq p \leq  \frac{3}{8} m$, $h_1({x_3}^0) < 0$. 
Consider the value $\displaystyle h_1(\frac{1}{3})$ for $\displaystyle  p = \frac{m}{3} + s$ where $s$ is a positive integer.  We see that 
$$ \displaystyle  h_1( \frac{1}{3}) = - \frac{1}{2541865828329} G_3(m, s),$$ 
where $G_3(m, s)$ is a polynomial of $m$ and $s$ with integer coefficients of degree 14 for $m$. 
We see that the polynomial $G_3(m, s)$ is of the form given by 
$$G_3(m, s) = \sum_{k =0}^{14} c_k(s) (m- 24 s)^k,$$
where $c_k(s)$ are  polynomials of $p$ with integer coefficients. We  see that each 
$c_k(s)$ is a polynomial  of $s-1$ with positive integer coefficients. For example, we have  
 \begin{equation*} 
\begin{array}{l}
c_0(s) = 2688886554248702829059568 (s-1)^{14}+28773011904669834888459456
   (s-1)^{13} \\+141508288769505404340266208
   (s-1)^{12}+425988148380665862862038816
   (s-1)^{11}+\cdots\\ +19077956371801372323151872
   (s-1)^2+2309278434711832223023104
   (s-1) \\ +108869460718905531039744. 
 \end{array}
\end{equation*} 
Thus we see that for $s \geq 1$ and $m- 24 s \geq 0$, $G_3(m, s)$ is positive.  Note that $1 \leq s \leq m/24$ and thus 
$m/3 + 1  \leq p \leq  m/3 + m/24 = 3 m/8$.

\medskip 

Case (4). \ We put $q =  m/2 - p$ and $\displaystyle {x_3}^0 =  \frac{4q}{m} - \frac{4}{m} - 8 (\frac{q}{m})^2 + 16\frac{q}{m^2}$. We claim that, for $\displaystyle  2  \leq q \leq  \frac{1}{8} m $, that is, 
$\displaystyle  \frac{3}{8} m \leq p \leq \frac{m}{2}-2$, $h_1({x_3}^0) < 0$. 
Consider the value $\displaystyle h_1(\frac{4 q}{m} - \frac{4}{m} - 8 (\frac{q}{m})^2 + 16\frac{q}{m^2})$. We see that 
$$ \displaystyle  h_1( \frac{4q}{m} - \frac{4}{m} - 8 (\frac{q}{m})^2 + 16\frac{q}{m^2}) = - \frac{16}{ m^{24}} G_4(m, q),$$
where $G_4(m, q)$ is a polynomial of $m$ and $q$ with integer coefficients of degree 30 for $m$. 
We see that the polynomial $G_4(m, q)$ is of the form given by 
$$G_4(m, q) = \sum_{k =0}^{30} u_k(q) (m- 8 q)^k,$$
where $u_k(q)$ are  polynomial of $q$ with integer coefficients. We  see that each 
$u_k(q) $ is a polynomial  of $q-2$ with positive integer coefficients. For example, we have 
{\small 
 \begin{equation*} 
\begin{array}{l}
u_0(q)  = 
26951178076734183104839680
   (q-2)^{38}+2212338096952683249388224512
   (q-2)^{37}\\
   +83992678988503465460710244352
   (q-2)^{36}+1995946208003865782253049085952
   (q-2)^{35}+\cdots\\+229210217524459650051376283663636365312
   (q-2)^2+22366794926378575054826400937697869824
   (q-2)\\+1012881211339770900930920976868179968. 
    \end{array}
\end{equation*} 
}
Thus we obtain that $G_4(m, q) > 0$ for $2 \leq q \leq m/8$, that is, for $\displaystyle  \frac{3}{8} m \leq p \leq \frac{m}{2}-2$. 

Note that, for the case when $m = 2 \ell$, we have $ p \leq \ell -3 = m/2 -3<\frac{m-3}2$, and for the case when $m = 2 \ell +1$, we have $ p \leq \ell -1 = (m-1)/2 -1$. 
We consider the case $ p = \ell -1$ where $m = 2 \ell +1$. We put $\displaystyle {x_3}^0 = \frac{1}{\ell}- \frac{1}{4 \ell^2}$. 
 We see that 
$$ \displaystyle  h_1( \frac{1}{\ell}- \frac{1}{4 \ell^2} ) = - \frac{1}{1048576 \ell^{19}} H_4(\ell),$$
where 
 \begin{equation*} 
\begin{array}{l} H_4(\ell) =28311552 \ell^{25}+127926272 \ell^{24}-3676700672
   \ell^{23}+25529024512 \ell^{22}-99774468096 \ell^{21} \\
   +251607146496
   \ell^{20}-413969008640 \ell^{19}+386872995840 \ell^{18}+11909161728
   \ell^{17}-698697895936 \ell^{16} \\
   +1310915822464
   \ell^{15}-1514173731328 \ell^{14}+1280110627808
   \ell^{13}-839682485472 \ell^{12}+439251246304 \ell^{11} \\
   -185919232072
   \ell^{10}+64111849503 \ell^9-18030414660 \ell^8+4116853866
   \ell^7-755179592 \ell^6 \\+109349551 \ell^5-12168116 \ell^4+999284
   \ell^3-56776 \ell^2+1984 \ell-32. 
     \end{array}
\end{equation*} 
We see that 
 \begin{equation*} 
\begin{array}{l} H_4(\ell) =
28311552 (\ell-3)^{25}+2251292672 (\ell-3)^{24}+81975181312
   (\ell-3)^{23}+1847752916992 (\ell-3)^{22}+\cdots \\  +287547059522662005140
   (\ell-3)^3+97552658701667320160 (\ell-3)^2+21045285340535234500
   (\ell-3) \\ +2167673762760385300, 
  \end{array}
\end{equation*} 
so $ \displaystyle  h_1( \frac{1}{\ell}- \frac{1}{4 \ell^2} ) < 0$ for $\ell \geq 3$. 



 \smallskip
 
 We now take a lexicographic order $>$  with $ y > x_3 >  x_4 > x_2$  for a monomial ordering on $R$. Then we see that 
a  Gr\"obner basis for the ideal $I_1$ contains the following polynomial $h_2(x_2)$ of degree 12 :   
    { \small  
   \begin{equation*} \label{equ_SO(m)_x2}
\begin{array}{l}  
h_2(x_2) = 
16 (p+1)^5 \left(p^2+4 p-1\right) \left(p^3+5 p^2-16 m p+35
   p+8 m^2-32 m+31\right) {x_2}^{12} \\
   -32 (m-2) (p+1)^4 \left(p^5+m p^4+9 p^4-26 m p^3+88 p^3+16 m^2 p^2-172 m p^2+308 p^2+56 m^2 p-222 m p \right.  \\ \left. 
   +207 p-8 m^2+35 m-37\right) {x_2}^{11}+\cdots\\
   -4 (m-2)^4 \left(m^6-2 p m^5-10 m^5-3 p^2 m^4+8 p m^4+31 m^4+8 p^3 m^3+40 p^2 m^3+50 p m^3-18 m^3-2 p^4 m^2 \right.  \\ 
  -56 p^3 m^2-176 p^2 m^2-288 p m^2-54 m^2-3 p^5 m+23 p^4 m+86 p^3 m+226 p^2 m+397 p m+39 m \\   \left. +p^6-6 p^5-23 p^4+52 p^3+27 p^2-94 p+43\right) {x_2} \\
   +(m-2)^4 (m-p-3) \left(m^2-p m-7 m-p^2+6 p+11\right) \left(m^3-2 p m^2-2 m^2+p^3-3 p^2+3 p-1\right). 
 \end{array}
\end{equation*}  
  }
  We claim that the equation  $h_2(x_2) =0$ has at least one positive real root. 
 We write 
$$ \displaystyle h_2(x_2)  =  \sum_{k =0}^{12} b_k(m, p) (-1)^k {x_2}^k.$$
Then $b_k(m, p)$ are  polynomial of $m$ and $p$ with integer coefficients. It is enough to see that $b_k(m, p)$ are positive.  Note that, if we denote by $n_k$ the degree  of $b_k(m, p)$  with respect to $m$, then we see that  $ n_{12} =2$, $n_{11}=3$, $n_{10}=4$, $n_9=5$, $n_8=6$, $n_7=6$, $n_6=6$, $n_5=6$, $n_4=6$, $n_3=6$, $n_2=6$, $n_1=6$, $n_0=6$. We see that each polynomial $b_k(m, p)$ is of the form given by 
$$b_k(m, p) = \sum_{j =0}^{n_k} u_j^k(p) (m- 2 p -3)^j,$$
where $u_j^k(p)$ are  polynomials of $p$ with integer coefficients. Now we  see that each 
$u_j^k(p)$ is a polynomial  of $p-2$ with positive integer coefficients. For example, we have 
\begin{equation*}
\begin{array}{l}  
u_0^0(p)=16 (p-2)^{10}+384 (p-2)^9+4360 (p-2)^8+30712 (p-2)^7+146905
   (p-2)^6+491510 (p-2)^5 \\
   +1149975 (p-2)^4+1839750
   (p-2)^3+1913125 (p-2)^2+1162500 (p-2)+312500. 
    \end{array}
\end{equation*} 
Thus we obtain that $b_k(m, p) > 0$ for $ 2 \leq p  \leq (m-3)/2$. From (\ref{equ_SO(m)_x2x3}), we see that there exists  a real solution for the equation  $h_2(x_2) =0$ and hence, it is a  positive solution of $h_2(x_2) =0$. 
  
  We  take a lexicographic order $>$  with $ y > x_3 >  x_2 > x_4$  for a monomial ordering on $R$. Then we see that 
a  Gr\"obner basis for the ideal $I_1$ contains the following polynomial $h_3(x_4)$ of degree 12 :   
    { \small  
   \begin{equation*} \label{equ_SO(m)_x4}
\begin{array}{l}  
h_3(x_4) = (m-p-1)^4 (p-1)^4 \left(-p^2+m p-5 p+2 m-4\right)^2 \left(4 p^4-4 m p^3+16 p^3+m^2 p^2-8 m p^2 \right. \\
\left. +16 p^2-4 m^2 p+16 m p-16 p+2 m^3-12 m^2+24 m-16\right) {x_4}^{12}\\
-2 (m-2) (m-p-1)^3 (p-1)^3 \left(-p^2+m p-5 p+2 m-4\right) \left(16 p^7-52 m p^6+192 p^6+60 m^2 p^5 -488 m p^5 \right.\\ 
+872  p^5-29 m^3 p^4 +412 m^2 p^4-1624 m p^4+1872 p^4+5 m^4 p^3-118 m^3 p^3+816 m^2 p^3-2056 m p^3 \\
+1688 p^3-7 m^4 p^2+37 m^3 p^2-100 m^2 p^2 +212 m p^2-208 p^2+6 m^5 p-90 m^4 p+542 m^3 p-1588 m^2 p \\
+2248 m p-1232 p+8 m^5-96 m^4+448 m^3 \left. -1024 m^2+1152 m-512\right) {x_4}^{11}+\cdots
\\
-256 (m-2)^4 (m-2 p)^2 (p-1)^3 \left(m^7-2 p m^6-14 m^6-4 p^2 m^5+24 p m^5+76 m^5 \right. +14 p^3 m^4+34 p^2  m^4-110 p m^4 \\ 
   -194 m^4-14 p^4 m^3-130 p^3 m^3-6 p^2 m^3+202 p m^3+204 m^3+4 p^5 m^2+168 p^4 m^2+240 p^3 m^2-456 p^2 m^2+12 p m^2 \\ 
   +32 m^2+5 p^6 m-90 p^5 m-425 p^4 m+180 p^3 m+1043 p^2 m-506 p m-207 m  -4 p^7+6 p^6+160 p^5 +306 p^4 \\ \left.
   -436 p^3-598 p^2+472 p+94\right) {x_4} 
   +256 (m-2)^4 (m-2 p)^2 (m-p-3) (p-1)^4
   \left(m^2-p m-7 m-p^2+6 p+11\right) \times \\
      \left(m^3-2 p m^2-2  m^2+p^3-3 p^2+3 p-1\right). 
  \end{array}
\end{equation*}  
 }

 We claim that the equation  $h_3(x_4) =0$ has at least one positive real root. 
 We write 
$$ \displaystyle h_3(x_4)  =  \sum_{k =0}^{12} c_k(m, p) (-1)^k {x_4}^k.$$
Then $c_k(m, p)$ are  polynomial of $m$ and $p$ with integer coefficients. It is enough to see that $c_k(m, p)$ are positive.  Note that, if we denote by $n_k$ the degree  of $c_k(m, p)$  with respect to $m$, then we see that  $ n_{12} =9$, $n_{11}=10$, $n_{10}=12$, $n_9=12$, $n_8=13$, $n_7=13$, $n_6=13$, $n_5=13$, $n_4=13$, $n_3=13$, $n_2=13$, $n_1=13$, $n_0=12$. We see that each polynomial $c_k(m, p)$ is of the form given by 
$$c_k(m, p) = \sum_{j =0}^{n_k} v_j^k(p) (m- 2 p -3)^j,$$
where $v_j^k(p)$ are  polynomials of $p$ with integer coefficients. Now we  see that each 
$v_j^k(p)$ is a polynomial  of $p-2$ with positive integer coefficients. For example, we have 
\begin{equation*}
\begin{array}{l}  
v_0^0(p)=36864 (p-2)^{14}+1032192 (p-2)^{13}+13805568
   (p-2)^{12}+116398080 (p-2)^{11} \\
   +685359360 (p-2)^{10}+2951944704 (p-2)^9+9503200512
   (p-2)^8+23056224768 (p-2)^7\\
   +42128455680 (p-2)^6 +57473072640 (p-2)^5+57485318400
   (p-2)^4+40820544000 (p-2)^3 \\+19441440000
   (p-2)^2+5558400000 (p-2) +720000000. 
    \end{array}
\end{equation*} 
Thus we obtain that $c_k(m, p) > 0$ for $ 2 \leq p  \leq (m-3)/2$. From (\ref{equ_SO(m)_x2x3}), we see that there exists  a real solution for the equation  $h_3(x_4) =0$ and hence, it is a  positive solution of $h_3(x_4) =0$. 

Therefore we obtain the following
  \begin{theorem}\label{equ_SO(m)_x1=1=x5_p=2sol} 
   The system of equations {\em (\ref{equ_SO(m)_x1=1=x5})} has at least two positive solutions. 
   Thus the flag manifold $M= \SO(m)/(\U(1)\times \U(p) \times \SO(m-2(p+1))$  admits at least two  $ \SO(m)$-invariant non-K\"ahler Einstein metrics for   any $p \geq 3 $,  and for $p=2$ when $m \geq 13$. 
\end{theorem}

For 
 $p=2$  
it is possible to use a similar analysis as before and sharpen the above result. In fact, as for the case (1)
we see that 
$$ \displaystyle  h_1\left(\frac{1}{2} \left(\frac{20 m}{33}-\frac{493}{198}\right) \right) = - \frac{1}{17385744542547923243473
   3056} K_1(m),$$
where $K_1(m)$ is a polynomial of $m$ with integer coefficients of degree 16.  By expanding  $K_1(m) $ by $m-11$, we obtain that 
 \begin{equation*} 
\begin{array}{l}
K_1(m) =  
 3583035271261716480000 (m-11)^{16}+511250673311500625510400
   (m-11)^{15} \\ 
   +31818296826635532065832960
   (m-11)^{14}+1173386825534099922378817536
   (m-11)^{13} \\
   +29027304327079432810844872704
   (m-11)^{12}+513765537329981469527636631552
   (m-11)^{11} +\cdots \\ 
   +99971319095011447419581739885675068
   (m-11)+15879235006092866105410345517930679. 
 \end{array}
\end{equation*} 
Thus we obtain that $K_1(m) > 0$ for $m \geq 11$. 

We also  see that 

 \begin{equation*} 
\begin{array}{l}  h_1(1) = (m-5)^2 \left(m^2-9 m+19\right)^2  \left( (m-10)^7+24 (m-10)^6+211 (m-10)^5+850 (m-10)^4 \right.  \\ \left. +1620
   (m-10)^3+1448 (m-10)^2+544 (m-10)+72\right). 
   \end{array}
\end{equation*} 
Thus we obtain that $h_1(1) > 0$ for $m \geq 10$.

Together with the above result, 
 we obtain the following:  
    \begin{theorem}\label{equ_SO(m)_x1=1=x5_p=2sol} 
   The flag manifolds $\SO(2\ell +1)/(\U(1)\times\U (2)\times\SO(2\ell -5))$  $(\ell \geq 6)$ and    
 $\SO(2\ell)/(\U(1)\times\U (2)\times\SO(2(\ell -3)))$ $(\ell \geq 7)$
 admit at least four invariant non-K\"ahler Einstein metrics.
 \end{theorem}

Moreover, for small values of $\ell$ and $p$ it is possible to obtain the precise number of invariant
 Einstein metrics  depending on type $B_{\ell}$ and $D_{\ell}$ as follows.

\medskip
\smallskip
{\small{\begin{center}
{\bf Table 4. The number of  non-isometric  homogeneous Einstein metrics} 
\end{center}
\begin{center}
 \begin{tabular}{|c|c|c|c|l|l|}
 \hline
     \ $(\ell, p)$  \  & Non-K\"ahler Einstein  &  Non-K\"ahler Einstein     & K\"ahler Einstein  &  \ \ Generalized flag manifold \\
          &    of type a &    of type b & & \\
                 \hline
    $(3, 2)$  &  2  & 1 &  2  &   $\SO(7)/\U(1)\times \U(2)$ \\ \hline
   $(4, 2)$ & 4 &  2 & 2&   $\SO(9)/\U(1)\times \U(2)\times \SO(3)$ \\
       \hline
   $(4, 3)$     & 2  & 1 &2&   $\SO(9)/\U(1)\times \U(3)$  \\ 
        \hline
     $(5, 2)$    &   4 &  2 & 2 &   $\SO(11)/\U(1)\times \U(2)\times \SO(5)$ \\   \hline
    $(5, 3)$  &    2 &  1 & 2 &   $\SO(11)/\U(1)\times \U(3)\times \SO(3)$  \\    \hline
       $(5, 4)$   &   2 &  1 & 2  &   $\SO(11)/\U(1)\times \U(4)$ \\   \hline
       $(6, 2)$  &     4 &  2 & 2 &   $\SO(13)/\U(1)\times \U(2)\times \SO(7)$ \\
        \hline
       $(6, 3)$  &     2 &  2 & 2 &   $\SO(13)/\U(1)\times \U(3)\times \SO(5)$ \\
        \hline
       $(6, 4)$  &     2 &  1 & 2 &   $\SO(13)/\U(1)\times \U(4)\times \SO(3)$ \\
     \hline
       $(6, 5)$  &     2 &  1 & 2 &   $\SO(13)/\U(1)\times \U(5)$   \\
  \hline
   $(5, 2)$  &  4  & 2 &  2  &   $\SO(10)/\U(1)\times \U(2)\times \SO(4)$ \\ \hline
   $(6, 2)$ & 4 &  2 & 2&   $\SO(12)/\U(1)\times \U(2)\times \SO(6)$ \\
       \hline
   $(6, 3)$     & 2  & 2 &2&   $\SO(12)/\U(1)\times \U(3)\times \SO(4)$  \\ 
        \hline
     $(7, 2)$    &   4 &  2 & 2 &   $\SO(14)/\U(1)\times \U(2)\times \SO(8)$ \\   \hline
    $(7, 3)$  &    4 &  2 & 2 &   $\SO(14)/\U(1)\times \U(3)\times \SO(6)$  \\    \hline
       $(7, 4)$   &   2 &  2 & 2  &   $\SO(14)/\U(1)\times \U(4)\times \SO(4)$ \\   \hline
 \end{tabular}
 \end{center}}}
 \smallskip
  \ {\small Non-K\"ahler Einstein metric of type a means that the metric of the form with $x_1 = x_5$ and 
 Non-K\"ahler Einstein metric of type b means that the metric of the form with $x_1 \neq x_5$.  }  
 
\medskip
\smallskip

 We conjecture that for the classical flag manifolds studied in the present work the {\it total number} of non isometric invariant Einstein metrics is precisely five, six or eight.   Note that two of them are
 K\"ahler-Einstein metrics.

\end{document}